\documentclass[reqno]{amsart}

\usepackage{amsmath}
\usepackage{amsfonts}
\usepackage{amsthm}
\usepackage{amssymb}
\usepackage{graphicx}
\usepackage{graphics}
\usepackage{bm}
\usepackage{dsfont}
\usepackage{color}
\usepackage[font=footnotesize]{caption}
\numberwithin{equation}{section}

\newtheoremstyle{personal}%
{12pt}
{12pt}
{\slshape}
{}
{\bfseries}
{.}
{.5em}
{}
\theoremstyle{personal}%
\newtheorem{thm}{Theorem}[section]
\newtheorem{cor}[thm]{Corollary}
\newtheorem{lem}[thm]{Lemma}
\newtheorem{prop}[thm]{Proposition}
\theoremstyle{definition}
\newtheorem{rem}[thm]{Remark}

\newcommand{\N}{\mathds{N}}
\newcommand{\Z}{\mathds{Z}}
\newcommand{\R}{\mathds{R}}
\newcommand{\C}{\mathds{C}}
\newcommand{\T}{\mathds{T}}
\newcommand{\K}{\mathds{K}}
\newcommand{\V}{\mathds{V}}
\newcommand{\M}{\mathds{M}}
\newcommand{\diff}{\mathrm{d}}
\newcommand{\dist}{\mathrm{dist}}
\newcommand{\Tan}{\mathrm{T}}
\newcommand{\cu}{c_{\mathrm{u}}}
\newcommand{\AC}{\mathrm{AC}}
\newcommand{\SSS}{\mathcal{S}}

\newcommand{\mmu}{\bm{\mu}}
\newcommand{\ggamma}{\bm{\gamma}}
\newcommand{\zzeta}{\bm{\zeta}}
\newcommand{\qq}{\bm{q}}
\newcommand{\vv}{\bm{v}}
\newcommand{\ww}{\bm{w}}
\newcommand{\inj}{{\mathrm{inj}}}
\newcommand{\leb}{{\mathrm{leb}}}
\newcommand{\Mult}{\mathcal{M}}
\newcommand{\ACMult}{\mathcal{C}}
\newcommand{\ELMult}{\mathcal{D}}
\newcommand{\UU}{\mathcal{U}}
\newcommand{\VV}{\mathcal{V}}
\newcommand{\WW}{\mathcal{W}}
\newcommand{\MM}{\mathcal{M}}
\newcommand{\PP}{\mathcal{P}}
\newcommand{\Jac}{\mathcal{J}}
\newcommand{\nneg}{n_{\mathrm{neg}}}
\DeclareMathOperator*{\toup}{\longrightarrow} 
\newcommand{\llangle}{\langle\!\langle}
\newcommand{\rrangle}{\rangle\!\rangle}
\newcommand{\ind}{\mathrm{ind}}
\newcommand{\nul}{\mathrm{nul}}

\begin{document}

\title[On Tonelli periodic orbits with low energy on surfaces]{On Tonelli periodic orbits\\ with low energy on surfaces}
\author{Luca Asselle}
\address{Ruhr Universit\"at Bochum, Fakult\"at f\"ur Mathematik\newline\indent Geb\"aude NA 4/33, D-44801 Bochum, Germany}
\email{luca.asselle@ruhr-uni-bochum.de}
\author{Marco Mazzucchelli}
\address{CNRS, \'Ecole Normale Sup\'erieure de Lyon, UMPA\newline\indent  69364 Lyon Cedex 07, France}
\email{marco.mazzucchelli@ens-lyon.fr}
\date{January 25, 2016. \emph{Revised}: January 16, 2017.}

\begin{abstract}
We prove that, on a closed surface, a Lagrangian system defined by a Tonelli Lagrangian $L$ possesses a periodic orbit that is a local minimizer of the free-period action functional on every energy level belonging to the low range of energies $(e_0(L),\cu(L))$. We also prove that almost every energy level in $(e_0(L),\cu(L))$ possesses infinitely many periodic orbits. These statements extend two results, respectively due to Taimanov and Abbondandolo-Macarini-Mazzucchelli-Paternain, valid for the special case of electromagnetic Lagrangians.
\tableofcontents
\end{abstract}

\maketitle

\section{Introduction}

Hamiltonian systems on cotangent bundles sometimes admit a dual description as Lagrangian systems, which are second order dynamical systems on the configuration space. This is the case for geodesic flows, which can be viewed as Reeb flows on unit-sphere cotangent bundles, or equivalently as the family of curves $\gamma$ on the configuration space satisfying the geodesic equation $\nabla_t\dot\gamma\equiv0$. The widest class of fiberwise convex Hamiltonian functions admitting a dual Lagrangian is the Tonelli one: a Hamiltonian $H:\Tan^* M\to\R$ is said to be Tonelli when its restriction to any fiber $H|_{\Tan_q^*M}$ is a superlinear function with everywhere positive-definite Hessian; the class of Lagrangians dual to the Tonelli Hamiltonians is given by the so-called Tonelli Lagrangians $L:\Tan M\to\R$, which satisfy the analogous properties as $H$ on the tangent bundle. The name of these Lagrangians is due to the seminal work of Tonelli \cite{Tonelli:1934ta}, who established the existence and regularity of action minimizers joining a given pair of points in the configuration space. More modern applications of 1-dimensional calculus of variations, notably Aubry-Mather theory \cite{Mather:1991xd, Mather:1991fu, Mather:1993dk} and weak KAM theory \cite{Fathi:1997kq}, also concerns Lagrangians and Hamiltonians of Tonelli type. 
The background and further applications of Tonelli systems can be found in, e.g., \cite{Fathi:2008xl, Contreras:1999fm, Mazzucchelli:2012nm, Sorrentino:2015ai}. The purpose of this paper is to extend to the Tonelli class two important existence results for periodic orbits, which were proved for the smaller class of electromagnetic Lagrangians.

Given a closed manifold $M$ and a Tonelli Lagrangian $L:\Tan M\to\R$, the associated Euler-Lagrange flow $\phi_L^t:\Tan M\to \Tan M$ is defined by $\phi_L^t(\gamma(0),\dot\gamma(0))=(\gamma(t),\dot\gamma(t))$, where $\gamma:\R\to M$ is a smooth solution of the Euler-Lagrange equation, which in local coordinates can be written as
\begin{align}\label{e:Euler_Lagrange}
\tfrac{\diff}{\diff t} L_v(\gamma,\dot\gamma) - L_q(\gamma,\dot\gamma)=0.
\end{align}
The energy function $E:\Tan M\to\R$, given by
$E(q,v)=L_v(q,v)v-L(q,v)$,
is a first integral for the Euler-Lagrange flow. The periodic orbits of the Euler-Lagrange flow contained in the energy hypersurface $E^{-1}(k)$ are in one to one correspondence 
with the critical points of the free-period action functional 
\begin{gather*}
\SSS_k:W^{1,2}(\R/\Z,M)\times (0,\infty)\to\R, \\ 
\SSS_k(\Gamma,\tau) = \tau\int_0^1 \Big [L (\Gamma(t),\dot \Gamma(t)/\tau)+k\Big ]\, \diff t.
\end{gather*}
We see the pair $(\Gamma,\tau)$ as the $\tau$-periodic curve $\gamma(t):=\Gamma(t/\tau)$, so that $\SSS_k(\Gamma,\tau)$ represents the action of $\gamma$ at level $k$. In this paper, a crucial role will be played by the  energy values $e_0(L)\leq \cu(L)$
, which are given by
\begin{align*}
e_0(L) & = \max_{q\in M} E(q,0),\\
\cu(L) & = -\inf\left\{\left. \frac1\tau \int_0^\tau L(\gamma(t),\dot\gamma(t))\,\diff t\ \right|\ \gamma:\R/\tau\Z\toup^{C^\infty} M \mbox{ contractible}\right\}. 
\end{align*}
In the literature, $\cu(L)$ is known as the Ma\~n\'e critical value of the universal cover (see \cite{Contreras:2006yo, Abbondandolo:2013is} and references therein).

Our first result is the following.
\begin{thm}\label{t:local_minimizers}
Let $M$ be a closed surface, and $L:\Tan M\to\R$ a Tonelli Lagrangian. For every energy level $k\in(e_0(L),\cu(L))$, the Lagrangian system of $L$ possesses a periodic orbit $\gamma_k$ with energy $k$ that is a local minimizer of the free-period action functional $\SSS_k$ over the space of absolutely continuous periodic curves. Moreover, $\SSS_k(\gamma_k)<0$ and $\gamma_k$ lifts to a simple closed curve in some finite cover of $M$.
\end{thm}

\begin{rem}
For certain Tonelli Lagrangians $L$, the interval $(e_0(L),\cu(L))$ is empty. This is the case if,  for all $q\in M$,  the function $v\mapsto L(q,v)$ has its minimum at the origin. Nevertheless, in plenty of examples $(e_0(L),\cu(L))$ is not empty. An important class of such examples is the one of electromagnetic Lagrangians \[L(q,v)=\tfrac12g_q(v,v)-\theta_q(v),\] where $g$ is a Riemannian metric on $M$ and $\theta$ is a 1-form on $M$ such that $\diff\theta\neq 0$. The interval $(e_0(L),\cu(L))$ is non-empty also whenever $L$ is sufficiently $C^1$-close to a given electromagnetic Lagrangian.\hfill\qed
\end{rem}

Theorem~\ref{t:local_minimizers} was first established by Taimanov \cite{Taimanov:1992sm, Taimanov:1991el} for electromagnetic Lagrangians. His beautiful observation is that the free-period action functional, which is not bounded from below in the energy range concerned by the statement, becomes bounded from below when restricted to a suitable space of null-homologous, embedded, periodic multicurves. This latter space is not compact, but a clever argument of Taimanov shows that the global minima of the free period action functional are contained in the interior of the space, and as such are critical points. Unfortunately, the papers \cite{Taimanov:1992sm, Taimanov:1991el} are rather short and some crucial details are only sketched. An alternative proof based on a regularity theorem for almost minimal currents in the sense of Almgren, a deep result from geometric measure theory, was provided later by Contreras, Macarini, and Paternain \cite{Contreras:2004lv}. Our proof of Theorem~\ref{t:local_minimizers}, which will be carried out in Section~\ref{s:local_minimizers}, fills the details in Taimanov's arguments, while at the same time generalizes the result to the class of Tonelli Lagrangians.

Our second result is the following.
\begin{thm}\label{t:multiplicity}
Let $M$ be a closed surface, and $L:\Tan M\to\R$ a Tonelli Lagrangian. For almost every energy level $k$ in the interval $(e_0(L),\cu(L))$ the Lagrangian system of $L$ possesses infinitely many periodic orbits with energy $k$ and negative action.
\end{thm}

For the special case of electromagnetic Lagrangians, Theorem~\ref{t:multiplicity} was proved by the second author together with Abbondandolo, Macarini, and Paternain in the recent paper \cite{Abbondandolo:2014rb} (see also \cite{Abbondandolo:2015lt} for a previous preliminary result in this direction). The proof involves a highly non-trivial argument of local critical point theory for the free-period action functional. Such a functional, in the classical $W^{1,2}$ functional setting, is $C^\infty$ for the case of electromagnetic Lagrangians, but is not $C^2$ for general Tonelli Lagrangians. This poses major difficulties, as the local arguments from critical point theory involve in an essential way the Hessian of the functional. Our way to circumvent this lack of regularity is to develop a finite dimensional functional setting, which may have independent interest, inspired by Morse's broken geodesics approximation of path spaces \cite[Section~16]{Milnor:1963rf}. The free-period action functional in such a finite dimensional setting, which we will call the discrete free-period action functional, is $C^\infty$. This will allow us to carry out the proof of Theorem~\ref{t:multiplicity} along the line of \cite{Abbondandolo:2014rb}.  The construction of the finite dimensional functional setting will be given in Section~\ref{s:discretizations}, and the proof of Theorem~\ref{t:multiplicity} in Section~\ref{s:multiplicity}

Both Theorems~\ref{t:local_minimizers} and~\ref{t:multiplicity} concern energy levels above $e_0(L)$. One may wonder whether the assertions of the theorems are still valid below $e_0(L)$. This is not the case. Indeed, in Section~\ref{s:counterexample} we will provide an example of Tonelli Lagrangian $L$ such that $e_0(L)>\min E$ and, for all energy values $k>\min E$ sufficiently close to $\min E$, the corresponding Lagrangian system possesses only two periodic orbits with energy $k$, and none of such orbits is a local minimizer of the free-period action functional.

Finally, Theorems~\ref{t:local_minimizers} and~\ref{t:multiplicity} can be extended to the case where the Euler-Lagrange dynamics, or actually the dual Hamiltonian dynamics, is modified by a non-exact magnetic term. The precise setting and the statements will be provided in Section~\ref{s:non_exact}.

\subsection*{Acknowledgments} The authors are grateful to Gabriele Benedetti for pointing out an inaccuracy in the first draft, and to an anonymous referee for providing many important corrections and remarks. L.A. is partially supported by the DFG grant AB 360/2-1 ``Periodic orbits of conservative systems below the Ma\~n\'e critical energy value''.
M.M. is partially supported by the ANR projects WKBHJ (ANR-12-BS01-0020) and COSPIN (ANR-13-JS01-0008-01). Part of the work presented in this paper was carried out in November 2015 at the Ruhr-Universit\"at Bochum. M.M. wishes to thank Alberto Abbondandolo for the kind hospitality, and the Mathematics Department in Bochum for providing a stimulating working environment.

\section{Local minimizers of the free-period action functional}
\label{s:local_minimizers}

\subsection{Unique free-time local minimizers}
\label{s:unique_free_time_minimizers}

Let $M$ be a closed manifold equipped with an arbitrary Riemannian metric $g$, which induces a distance $\mathrm{dist}:M\times M\to[0,\infty)$. Without loss of generality, we can assume that $M$ is orientable: if this is not the case, we replace it by its orientation double cover. We consider a Tonelli Lagrangian $L:\Tan M\to\R$. We recall that such an $L$ is a smooth function whose restriction to any fiber of $\Tan M$ is superlinear with everywhere positive-definite Hessian. We are interested in the Euler-Lagrange dynamics, that is, in the solutions $\gamma:\R\to M$ of the Euler-Lagrange equation~\eqref{e:Euler_Lagrange}. This type of dynamics is conservative: the energy $E(q,v)= L_v(q,v) v - L(q,v)$ is constant along the lift $t\mapsto(\gamma(t),\dot\gamma(t))$ of a solution $\gamma$ of the Euler-Lagrange equation.

For all $q_0,q_1\in M$, we denote by $\AC(q_0,q_1)$ the space of absolutely continuous curves connecting $q_0$ with $q_1$, that is, the space of all $\gamma:[0,\tau]\to M$ such that $\tau\in(0,\infty)$, $\gamma(0)=q_0$, and $\gamma(\tau)=q_1$. For any constant $k\in\R$, the action of one such curve $\gamma$ with respect to the Lagrangian $L+k$ is given by
\begin{align*}
\SSS_k(\gamma)
=
\int_0^\tau L(\gamma(t),\dot\gamma(t))\,\diff t + k \tau.
\end{align*}
It is well known that the space $\AC(q_0,q_1)$ has the structure of a Banach manifold: it can be seen as the product of the interval $(0,\infty)$ with the space of absolutely continuous curves from the interval $[0,1]$ to $M$ joining $q_0$ and $q_1$. In this setting, the principle of least action guarantees that the extremal points of $\SSS_k:\AC(q_0,q_1)\to\R$ are precisely those $\gamma$ that are solution of the Euler-Lagrange equation with energy $E(\gamma,\dot\gamma)\equiv k$. In the following, with a slight abuse of notation, when $q_0=q_1$ we will also add to $\AC(q_0,q_0)$ the \textbf{collapsed curve} $\gamma:\{0\}\to\{q_0\}$; a fundamental system of open neighborhoods of such a $\gamma$ can be defined as follows: a small neighborhood is given by the absolutely continuous curves $\zeta:[0,\tau]\to M$ such that $\zeta(0)=\zeta(\tau)=q_0$, $\tau\geq0$ is small, and the weak derivative of $\zeta$ has small $L^1$ norm. We want $\SSS_k$ to be lower semi-continuous on $\AC(q_0,q_0)$, and therefore the action of a collapsed curve $\gamma$ is $\SSS_k(\gamma)=0$.
Following Ma\~n\'e \cite{Mane:1997nw, Contreras:1997jq}, 
we define the \textbf{action potential}
\begin{gather*}
\Phi_k:M\times M\to\R\cup\{-\infty\},\\
\Phi_k(q_0,q_1) := \inf\big\{ \SSS_k(\gamma)\ \big|\ \gamma\in\AC(q_0,q_1)  \big\}.
\end{gather*}
The \textbf{Ma\~n\'e critical value} $c(L)$ is defined as the infimum of the set of $k\in\R$ such that the function $\Phi_k$ is everywhere finite. It is easy to see that $c(L)$ is a finite value: it must be larger than or equal to 
\[e_0(L):=\max_{q\in M} E(q,0).\] 
If $k\geq c(L)$, the action potential $\Phi_k$ is a Lipschitz function vanishing on the diagonal submanifold of $M\times M$. If $k>c(L)$, $\Phi_k$ is strictly positive outside the diagonal, and the infimum in its definition is actually a minimum: for any pair of points $q_0,q_1\in M$ there exists a curve $\gamma\in\AC(q_0,q_1)$ satisfying $\Phi_k(q_0,q_1)=\SSS_k(\gamma)$; such a $\gamma$ is necessarily a solution of the Euler-Lagrange equation with energy $k$.

It will be useful for us to consider also the Hamiltonian characterization of the Ma\~n\'e critical value, which goes as follows. Let $H:\Tan^*M\to\R$ be the Tonelli Hamiltonian dual to $L$, that is
\begin{align}\label{e:Hamiltonian}
 H(q,p)=\max_{v\in\Tan_q M} \Big [p v - L(q,v)\Big].
\end{align}
The critical value $c(L)$ can be defined as the infimum of $k$ such that there exists a smooth function $u:M\to\R$ that satisfies the Hamilton-Jacobi inequality 
\[H\big(q,\diff u(q)\big )<k.\]

It is well known that, for any energy level $k>c(L)$, there exists a diffeomorphism of the cotangent bundle $\Tan^*M$ sending the energy hypersurface $H^{-1}(k)$ to the unit-sphere cotangent bundle of a Finsler metric $F$ on $M$, see \cite[Corollary~2]{Contreras:1998lr}; in particular, the solutions of the Euler-Lagrange equation with energy $k$ are reparametrizations of the geodesics of $F$, and the following holds.

\begin{lem}\label{l:global_minimizers}
For each $k>c(L)$, there exists $\tau_\inj>0$ such that the following hold:
\begin{itemize}
\item[(i)] Every point $q_0\in M$ admits an open neighborhood $U_{q_0}\subset M$ such that the smooth map
\begin{align*}
\psi_{q_0}:\big [0,\tau_\inj\big )\times \big (E^{-1}(k)\cap\Tan_{q_0}M\big ) &\to  U_{q_0}
\end{align*}
given by $\psi_{q_0}(\tau,v_0)=\pi\circ\phi_L^\tau(q_0,v_0)$ restricts to a diffeomorphism
\begin{align*}
\psi_{q_0}:\big (0,\tau_\inj\big )\times \big (E^{-1}(k)\cap\Tan_{q_0}M\big ) &\to   U_{q_0}\setminus\{q_0\}.
\end{align*}

\item[(ii)] For each $q_0\in M$, $q_1\in U_{q_0}$, and $(\tau,v_0)\in\psi_{q_0}^{-1}(q_1)$, the embedded curve $\gamma:[0,\tau]\hookrightarrow U_{q_0}$ given by $\gamma(t):=\psi_{q_0}(t,v_0)$ is the unique global minimizer of the action functional $\SSS_k:\AC(q_0,q_1)\to\R$.
\end{itemize}
\end{lem}

\begin{proof}
We already remarked that, since $k>c(L)$, there exists a Finsler metric $F$ on $M$ such that every solution of the Euler-Lagrange equation of $L$ with energy $k$ is a reparametrization of a unit-speed geodesic of $F$. Therefore, point~(i) follows from the analogous statement about the exponential maps of closed Finsler manifolds (see, e.g., 
\cite{Abate:1994kn, Bao:2000hv}). As for point~(ii), notice that it is enough to establish its assertion for all pair of points $q_0,q_1\in M$ sufficiently close. The original assertion will follow by replacing $\tau_\inj$ by a smaller positive constant.

For all  $q_0,q_1\in M$, we denote by $\gamma_{q_0,q_1}:[0,\tau_{q_0,q_1}]\to M$ a solution of the Euler-Lagrange equation with energy $k$ such that $\SSS_k(\gamma_{q_0,q_1})=\Phi_k(q_0,q_1)$. In particular, $\gamma_{q_0,q_1}$ is a global minimizer of $\SSS_k:\AC(q_0,q_1)\to\R$. Notice that
\begin{align*}
\Phi_k(q_0,q_1) &  =  \SSS_{c(L)}(\gamma_{q_0,q_1}) + (k-c(L))\tau_{q_0,q_1} \\
& \geq \Phi_{c(L)}(q_0,q_1) + (k-c(L))\tau_{q_0,q_1},
\end{align*}
and therefore 
\begin{align}
\label{e:bound_tau}
\tau_{q_0,q_1}
\leq 
\Omega_k(q_0,q_1):= \frac{\Phi_k(q_0,q_1)-\Phi_{c(L)}(q_0,q_1)}{k-c(L)}.
\end{align}
The function $\Omega_k:M\times M\rightarrow [0,\infty)$ is Lipschitz and vanishes on the diagonal submanifold of $M\times M$. We denote by $\ell_k>0$ the Lipschitz constant of $\Omega_k$, so that 
\begin{align*}
\Omega_k(q_0,q_1) =  \underbrace{\Omega_k(q_0,q_0)}_{=0} + \underbrace{\Omega_k(q_0,q_1) - \Omega_k(q_0,q_0)}_{\leq \ell_k\, \dist(q_0,q_1)}
\leq 
\ell_k\, \text{dist}(q_0,q_1).
\end{align*}
If $\dist(q_0,q_1)<\tau_\inj/\ell_k$, the estimate~\eqref{e:bound_tau} implies that $\tau_{q_0,q_1}<\tau_\inj$, and therefore any global minimizer of the action functional $\SSS_k:\AC(q_0,q_1)\to\R$ is contained in the open set $U_{q_0}$. Point~(i) implies that such a minimizer is unique and is an embedded curve.
\end{proof}

The following lemma shows that on any energy level in $(e_0(L),c(L)]$ the qualitative properties of the Euler-Lagrange dynamics are locally the same as on energy levels in $(c(L),\infty)$.

\begin{lem}\label{l:locally_above_Mane}
If $k>e_0(L)$, every point of the closed manifold $M$ admits an open neighborhood $W$ and a Tonelli Lagrangian $\tilde L:\Tan M\to\R$ such that 
$\tilde L|_{\Tan W}= L|_{\Tan W}$ and $c(\tilde L)<k$.
\end{lem}

\begin{proof}
Let $q_0\in M$ be a given point. We choose a smooth function $u:M\to\R$ such that $\diff u(q_0)=\partial_v L(q_0,0)$. This implies that
\begin{align*}
 H(q_0,\diff u(q_0))  = -L(q_0,0)\leq  e_0(L) < k.
\end{align*}
We choose an open neighborhood $W$ of $q_0$ whose closure is contained in an open set $W'\subset M$ such that $H(q,\diff u(q))<k$ for all $q\in W'$. Let $\chi:M\to(0,1]$ be a smooth bump function that is identically equal to $1$ on $W$ and is identically equal to a constant smaller than 
$k/\max \{H(q,\diff u(q))\,|\,q\in M \}$
outside $W'$. We define a new Tonelli Hamiltonian $\tilde H:\Tan^*M\to\R$ by 
$\tilde H(q,p):= \chi(q)\,H(q,p)$, 
and its dual Tonelli Lagrangian $\tilde L:\Tan M\to\R$. By construction, $\tilde L$ coincides with $L$ on $\Tan W$, and we have $\tilde H(q,\diff u(q))<k$ for all $q\in M$. In particular, $k>c(\tilde L)$.
\end{proof}

If $k<c(L)$, the action potential is identically equal to $-\infty$, and in particular there are no curves joining two given points that are global minimizers of the action. However, if $k>e_0(L)$ there are still local minimizers joining sufficiently close points.

\begin{lem}\label{l:local_minimizers}
For each $k>e_0(L)$, there exist $\tau_\inj>0$ and $\rho_\inj>0$ such that the following hold:
\begin{itemize}
\item[(i)] Every point $q_0\in M$ admits an open neighborhood $U_{q_0}\subset M$ containing the compact Riemannian ball $\overline{B_g(q_0,\rho_\inj)}=\{q\in M\, |\, \dist(q_0,q)\leq\rho_\inj\}$ such that the smooth map
\begin{align*}
\psi_{q_0}:\big [0,\tau_\inj\big )\times \big (E^{-1}(k)\cap\Tan_{q_0}M\big ) &\to  U_{q_0}
\end{align*}
given by $\psi_{q_0}(\tau,v_0)=\pi\circ\phi_L^\tau(q_0,v_0)$ restricts to a diffeomorphism
\begin{align*}
\psi_{q_0}:\big (0,\tau_\inj\big )\times \big (E^{-1}(k)\cap\Tan_{q_0}M\big ) &\to  U_{q_0}\setminus\{q_0\}.
\end{align*}

\item[(ii)] For each $q_0\in M$, $q_1\in U_{q_0}$, and $(\tau,v_0)\in\psi_{q_0}^{-1}(q_1)$, the embedded curve $\gamma:[0,\tau]\to U_{q_0}$ given by $\gamma(t):=\psi_{q_0}(t,v_0)$ is the unique global minimizer of the restriction of the action functional $\SSS_k:\AC(q_0,q_1)\to\R$ to the open subset of absolutely continuous curves contained in $U_{q_0}$.
\end{itemize}
\end{lem}

In the following, we will say that the curve $\gamma$ of point~(ii) above is the \textbf{unique free-time local minimizer} with energy $k$ joining $q_0$ and $q_1$.

\begin{rem}
If $k\leq e_0(L)$, the assertion of Lemma~\ref{l:locally_above_Mane} is not valid for all points of $M$ anymore, but only for those points $q_0\in M$ such that $E(q_0,0)<k$. Analogously, the assertions of Lemma~\ref{l:local_minimizers} are still valid for all points $q_0$ such that  $E(q_0,0)<k$ and for some $\tau_\inj>0$ depending on $q_0$.
\hfill\qed
\end{rem}

\begin{proof}[Proof of Lemma~\ref{l:local_minimizers}]
By Lemma~\ref{l:locally_above_Mane}, we can find an open cover $W_1,...,W_h$ of $M$ and, for all $i=1,...,h$, a Tonelli Lagrangian $L_i:\Tan M\to\R$ such that $c(L_i)<k$ and $L_i|_{\Tan W_i}=L|_{\Tan W_i}$. Let $E_i:\Tan M\to\R$ be the associated energy functions 
$$E_i(q,v)=\partial_v L_i(q,v) v-L_i(q,v).$$
We choose $r>0$ large enough so that every energy hypersurface $E_i^{-1}(k)$ is contained in the ball tangent bundle of radius $r$, i.e.,
\begin{align*}
g_q(v,v) < r^2,\qquad \forall (q,v)\in E_i^{-1}(k).
\end{align*}
We apply Lemma~\ref{l:global_minimizers} to each Lagrangian $L_i$, thus obtaining the constant $\tau_i>0$ and the maps 
\[\psi_{i,q_0}:\big [0,\tau_i\big )\times \big (E^{-1}(k)\cap\Tan_{q_0}M\big ) \to  U_{i,q_0}.\]
We denote by $\rho_\leb>0$ the Lebesgue number of the open cover $W_1,...,W_h$. We recall that $\rho_\leb$ is such that, for every $q_0\in M$, the open Riemannian ball $B_g(q_0,\rho_\leb)$ is contained in some open set $W_i$ of the cover. Notice that, for all $q_0\in M$ and $i\in\{1,...,h\}$, the open set $U_{i,q_0}$ is contained in the Riemannian ball $B_g(q_0,\tau_i r)$. We reduce the positive constants $\tau_i$, so that they all coincide to a same positive constant $\tau_\inj <\rho_\leb/r$. This implies that, for all $q_0\in M$, there exists $i\in\{1,...,h\}$ such that $U_{i,q_0}$ is contained in the open set $W_i$, above which the Lagrangians $L_i$ and $L$ coincide; hence, we set $U_{q_0}:=U_{i,q_0}$ and we define the map 
\begin{gather*}
\psi_{q_0}:\big [0,\tau_\inj\big )\times \big (E^{-1}(k)\cap\Tan_{q_0}M\big ) \to   U_{q_0},\\
\psi_{q_0}(\tau,v_0):= \pi\circ\phi_L^\tau(q_0,v_0)=\pi\circ\phi_{L_i}^\tau(q_0,v_0).
\end{gather*}
The claims of points~(i-ii), except that $U_{q_0}\supset \overline{B_g(q_0,\rho_\inj)}$, follow from Lemma~\ref{l:global_minimizers} applied to the Lagrangian $L_i$. The inclusion $\overline{B_g(q_0,\rho_\inj)}\subset U_{q_0}$ follows if we choose the constant $\rho_\inj$ such that 
\begin{equation*}
0<\rho_\inj<\min \big\{\dist \big (q_0,\pi\circ\phi_L^{\tau_\inj}(q_0,v_0)\big ) \ \big|\ (q_0,v_0)\in E^{-1}(k)\big\}.  \qedhere
\end{equation*}
\end{proof}

As a corollary of Lemma~\ref{l:local_minimizers}, we reobtain the following well known statement (see, e.g., \cite[Proof of Lemma~7.2]{Abbondandolo:2013is} for the original proof).

\begin{cor}\label{c:small_loops}
Every absolutely continuous periodic curve $\gamma:\R/\tau\Z\to M$ with $\tau>0$ and image contained in a Riemannian ball of diameter $\rho_\inj$ is contractible and satisfies $\SSS_k(\gamma)>0$.
\end{cor}

\begin{proof}
We set $q_0:=\gamma(0)$. The curve $\gamma$ is entirely contained in the Riemannian ball $B_g(q_0,\rho_\inj)$, which in turn is contained in the open set $U_{q_0}$ given by Lemma~\ref{l:local_minimizers}. Since $U_{q_0}$ is contractible, $\gamma$ is contractible as well. Consider the stationary curve $\gamma_0:\{0\}\to \{q_0\}$. The curve $\gamma$ might be stationary as well, but formally it is different from $\gamma_0$ since $\tau>0$. By Lemma~\ref{l:local_minimizers}(ii), $\gamma_0$ is the unique global minimizer of the restriction of the action functional $\SSS_k:\AC(q_0,q_0)\to\R$ to the open subset of curves contained in $U_{q_0}$. Therefore $\SSS_k(\gamma)>\SSS_k(\gamma_0)=0$.
\end{proof}

\subsection{Embedded global minimizers of the free-period action functional}\label{ss:embedded_global_min}

From now on, we will assume that $M$ is an orientable closed surface, that is,
\begin{align*}
\dim(M)=2. 
\end{align*}
We consider an energy level
\begin{align*}
k> e_0(L),
\end{align*}
and the functional $\SSS_k$ as defined on the space of absolutely continuous periodic curves on $M$ with arbitrary period. The latter space can be identified with the product 
$\AC(\R/\Z,M) \times (0,\infty)$,
so that any pair $(\Gamma,\tau)$ in this product defines the $\tau$-periodic curve $\gamma(t):=\Gamma(t/\tau)$. The functional $\SSS_k$ will be called the \textbf{free-period action}.

The proof of Theorem~\ref{t:local_minimizers} will follow Taimanov's ideas \cite{Taimanov:1991el, Taimanov:1992sm, Taimanov:1992fs} for the electromagnetic case: the free-period action functional $\SSS_k$ is bounded from below on the space of embedded periodic multicurves that are the oriented boundary of a compact region of $M$; this suggests to look for the periodic orbit of Theorem~\ref{t:local_minimizers} by  minimizing $\SSS_k$ there. The major difficulty is that such a space of multicurves is not compact, and so might be the minimizing sequences.

Let us introduce a suitable space over which we will perform the minimization. For each positive integer $m$, we first define $\ACMult(m)$ to be the space of multicurves $\ggamma=(\gamma_1,...,\gamma_m)$ such that
\begin{itemize}
\item[(C1)] for every $i\in\{1,...,m\}$, the curve $\gamma_i:\R/\tau_i\Z\to M$ is absolutely continuous and topologically embedded;

\item[(C2)] any two curves $\gamma_i$ and $\gamma_j$ do not have mutual intersections;

\item[(C3)] the multicurve $\ggamma$ is the oriented boundary of a possibly disconnected, embedded, oriented, compact surface $\Sigma\subset M$ whose orientation agrees with the one of $M$.
\end{itemize}
For all positive integers $m\leq n$, we then define $\ELMult_k(m,n)$ to be the space of multicurves $\ggamma=(\gamma_1,...,\gamma_m)$ such that:
\begin{itemize}
\item[(D1)] each $\gamma_i:\R/\tau_i\Z\to M$ is a continuous curve, and there exist 
$$0=\tau_{i,0}\leq \tau_{i,1}\leq ...\leq \tau_{i,n_i}=\tau_i$$
such that, for all $j=0,...,n_i-1$, 
$$\dist\big (\gamma_i(\tau_{i,j}),\gamma_i(\tau_{i,j+1})\big )\leq \rho_\inj,$$
and the restriction $\gamma_i|_{[\tau_{i,j},\tau_{i,j+1}]}$ is the unique free-time local minimizer with energy $k$ joining the endpoints (see the definition after Lemma~\ref{l:local_minimizers});

\item[(D2)] $n_1+...+n_m\leq n$.
\end{itemize}
We endow both $\ACMult(m)$ and $\ELMult_k(m,n)$ with the topologies induced by the inclusions into the space of absolutely continuous multicurves with $m$ connected components. 

\begin{lem}\label{l:ELMult_compact}
The space $\ELMult_k(m,n)$ is compact.
\end{lem}
\begin{proof}
Let $\{\ggamma_\alpha=(\gamma_{\alpha,1},...,\gamma_{\alpha,m})\ |\ \alpha\in\N \}\subseteq \ELMult_k(m,n)$ be a sequence of multicurves, and 
$0=\tau_{\alpha,i,0}\leq \tau_{\alpha,i,1}\leq ...\leq \tau_{\alpha,i,n_{\alpha,i}}=\tau_{\alpha,i}$
 the time-decomposition for the connected component $\gamma_{\alpha,i}$ given in (D1). Since $n_{\alpha,1}+...+n_{\alpha,m}\leq n$, by the pigeonhole principle we can extract a subsequence, which we still denote by $\{\ggamma_\alpha\ |\ \alpha\in\N \}$, such that, for some $n_1,...,n_m\in\N$, we have
\begin{align*}
n_{\alpha,i}=n_i,\qquad \forall \alpha\in\N,\ i=1,...,m.
\end{align*}
Since $M$ is compact, up to extracting a further subsequence we can assume that, for each $i\in \{1,...,m\}$ and  $j\in \{1,...,n_i\}$, the sequence 
$\{\gamma_{\alpha,i}(\tau_{\alpha,i,j})\,|\,\alpha\in \N\}$ converges to some point $q_{i,j}\in M$. For each $i=1,...,m$ we set $\gamma_i:\R/\tau_i\Z\to M$ to be the unique continuous curve such that, for some $0=\tau_{i,0}\leq \tau_{i,1}\leq ...\leq \tau_{i,n_{i}}=\tau_{i}$,
each portion $\gamma_i|_{[\tau_{i,j},\tau_{i,j+1}]}$ is the unique free-time local minimizer with energy $k$ joining $q_{i,j}$ and $q_{i,j+1}$. The multicurve $\ggamma:=(\gamma_1,...,\gamma_{m})$ clearly belongs to $\ELMult_k(m,n)$, and $\ggamma_\alpha\to\ggamma$ as $\alpha\to\infty$. 
\end{proof}

Finally, the space of multicurves that we will be interested in will be
\begin{align*}
\Mult_k(n):=\bigcup_{m=1}^{n} \ \Mult_k(m,n),
\end{align*}
where
\begin{align*}
\Mult_k(m,n):= \ELMult_k(m,n) \cap \overline{\ACMult(m)}.
\end{align*}
It readily follows from this definition and Lemma~\ref{l:ELMult_compact} that $\Mult_k(m,n)$ is compact, and so is $\Mult_k(n)$. With a slight abuse of notation, we will denote by $\SSS_k:\Mult_k(m,n)\to\R$ the free-period action functional with energy $k$ on multicurves, given by
\begin{align*}
\SSS_k(\ggamma)=
\sum_{i=1}^m \ \SSS_k(\gamma_i)=
\sum_{i=1}^m \ \left( \int_0^{\tau_i}  L(\gamma_i(t),\dot\gamma_i(t))\,\diff t + \tau_i k\right).
\end{align*}
Clearly, $\SSS_k$ is continuous, and therefore admits a minimum over the compact space $\Mult_k(n)$. The next two Lemmas give sufficient conditions for the connected components of a minimizer of $\SSS_k$ to be embedded periodic orbits of the Lagrangian system of $L$. We say that a periodic curve $\gamma:\R/\tau\Z\to M$ is \textbf{non-iterated} when $\tau$ is its minimal period.

\begin{lem}\label{l:components_are_periodic_orbits}
Assume that $\ggamma=(\gamma_1,...,\gamma_m)\in \Mult_k(n)$ is a multicurve that satisfies
\begin{align}\label{e:global_minimum_LEMMA}
\SSS_k(\ggamma) = \min_{\Mult_k(n)} \SSS_k
\end{align}
and has no connected component that is a collapsed curve. 
If in addition $\ggamma \in \Mult_k(n-1)$, then each of its connected components is a non-iterated periodic orbit of the Lagrangian system of $L$ with energy $k$.
\end{lem}

\begin{proof}
We first remark that $\ggamma$ does not have two distinct connected components with the same image in $M$. Indeed, if two such connected components exist, since $\ggamma\in\Mult_k(n)$ and none of its connected components is a collapsed curve, $\ggamma$ must have two connected components $\gamma_i$ and $\gamma_j$ that are the same geometric curve with opposite orientation. Corollary~\ref{c:small_loops} implies that $\SSS_k(\gamma_i)+\SSS_k(\gamma_j)>0$, for the multicurve $(\gamma_i,\gamma_j)$ can also be decomposed as the union of finitely many loops of length less than $\rho_\inj$. But in this case, if we remove from $\ggamma$ the connected components $\gamma_i$ and $\gamma_j$, the obtained multicurve $\ggamma'$ still belongs to $\Mult_k(n)$ and has action $\SSS_k(\ggamma')<\SSS_k(\ggamma)$, contradicting~\eqref{e:global_minimum_LEMMA}.

Every connected component of $\ggamma$ has the form $\gamma_i:\R/\tau_i\Z\to M$, and there exists a sequence of time parameters 
$0=\tau_{i,0}<\tau_{i,1}<...<\tau_{i,n_i}=\tau_i$
such that each restriction $\gamma_i|_{[\tau_{i,j},\tau_{i,j+1}]}$ is the unique free-time local minimizer with energy $k$ joining the endpoints, which are distinct and lie at distance less than or equal to $\rho_\inj$. Assume by contradiction that $\gamma_{i_1}$ is not $C^1$ at some time $\tau_{i_1,j_1}$. There may be other portions of the multicurve $\ggamma$ passing through $q:=\gamma_{i_1}(\tau_{i_1,j_1})$. Thus, there are finitely many more (possibly zero) indices $i_2,...,i_r\in\{1,...,m\}$ and times $\tau_{i_2,j_2},...,\tau_{i_r,j_r}$ such that the $(i_h,j_h)$'s are pairwise distinct and satisfy
\begin{align*}
\gamma_{i_1}(\tau_{i_1,j_1})=\gamma_{i_2}(\tau_{i_2,j_2})=...=\gamma_{i_r}(\tau_{i_r,j_r})=q.
\end{align*}

Up to replacing $(i_1,j_1)$ with some of the other $(i_h,j_h)$, we can assume that the curve $\gamma_{i_1}$ is an ``innermost'' one around its corner $\gamma_{i_1}(\tau_{i_1,j_1})$. More precisely, this means that for an arbitrarily small open disk $B\subset M$ containing the point $q$ the following holds. Let $(a,b)$ be the widest interval containing $\tau_{i_1,j_1}$ and such that $\gamma_{i_1}|_{(a,b)}$ is contained in $B$. We denote by $B''\subset B$ the connected component of $B\setminus\gamma_i|_{(a,b)}$ lying on the side of the angle larger than $\pi$ formed by $\gamma_i$ at $q$, and fix an arbitrary point $q''$ in the interior of $B''$. Consider a sequence of embedded multicurves $\ggamma_\alpha =(\gamma_{\alpha,1},...,\gamma_{\alpha,m})\in\ACMult(m)$ converging to $\ggamma$ as $\alpha\to\infty$. Let $(a_\alpha,b_\alpha)$ be the widest interval containing $\tau_{i_1,j_1}$ and such that $\gamma_{\alpha,i_1}|_{(a_\alpha,b_\alpha)}$ is contained in $B$. We denote by $B_\alpha''$ the connected component of $B\setminus\gamma_{\alpha,i_1}|_{(a_\alpha,b_\alpha)}$ containing the point $q''$. Since the connected components of $\ggamma$ have pairwise distinct image in $M$, for all $h=2,...,r$ either $\gamma_{\alpha,i_h}(\tau_{i_h,j_h})$ belongs to $B''_\alpha$ for all $\alpha$ large enough, or $\gamma_{\alpha,i_h}(\tau_{i_h,j_h})$ belongs to $B\setminus B''_\alpha$ for all $\alpha$ large enough. Then, the ``innermost'' condition for $\gamma_{i_1}$ is that, for all $\alpha$ large enough, all the points $\gamma_{\alpha,i_2}(\tau_{i_2,j_2}),...,\gamma_{\alpha,i_r}(\tau_{i_r,j_r})$ belong to $B''_\alpha$. The simple situation where the curves $\gamma_{i_1}, ...,\gamma_{i_r}$ have an isolated intersection at $q$ is depicted in Figure~\ref{f:innermost_curve}(a).

We set $(i,j):=(i_1,j_1)$ and $B':=B\setminus B''$. Notice that a priori $B'$ might have empty interior: this is the case when $\gamma_i$ reaches the point $q$ and goes back along the same path. Up to shrinking $B$ around $q$, we can assume that the multicurve $\ggamma$ does not intersect the interior of $B'$. Let $\epsilon_1>0$ be small enough so that the restriction $\gamma_i|_{[\tau_{i,j}-\epsilon_1,\tau_{i,j}]}$ is contained in $B'$ and has length less than the constant $\rho_\inj>0$ of Lemma~\ref{l:local_minimizers}. For $\epsilon_2>0$ sufficiently small, we remove the portion $\gamma_i|_{[\tau_{i,j}-\epsilon_1,\tau_{i,j}+\epsilon_2]}$ from $\gamma_i$ and glue in the unique free-time local minimizer with energy $k$ joining $\gamma_i(\tau_{i,j}-\epsilon_1)$ and $\gamma_i(\tau_{i,j}+\epsilon_2)$, see Figure~\ref{f:innermost_curve}(b). 
\begin{figure}
\begin{center}
\begin{small}
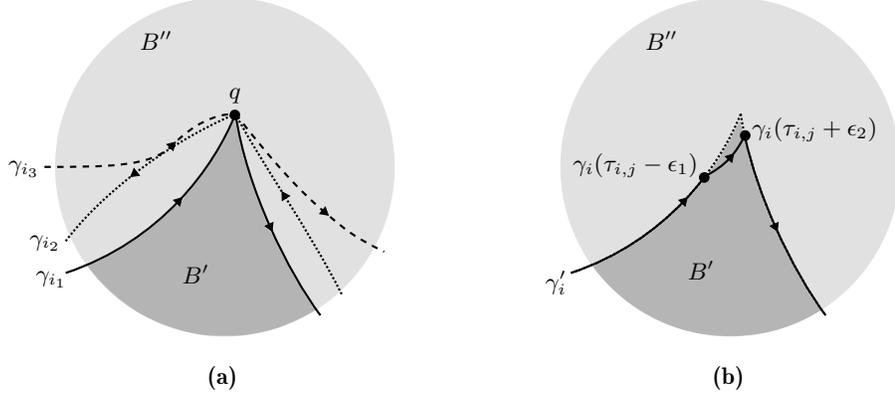 
\caption{\textbf{(a)} The connected components $\gamma_{i_h}$ intersecting at $q$.
\textbf{(b)} The curve $\gamma_i'$.}
\label{f:innermost_curve}
\end{small}
\end{center}
\end{figure}
Lemma~\ref{l:local_minimizers}(i) guarantees that the portion that we glued in does not have self-intersections nor intersections with other points of $\gamma_i$ except at the endpoints. We denote by $\gamma_i'$ the modified curve, and by $\ggamma'$ the multicurve obtained from $\ggamma$ by replacing $\gamma_i$ with $\gamma_i'$.

We claim that $\ggamma'$ belongs to $\Mult_k(n)$. Indeed, since $\ggamma\in\Mult_k(m,n-1)$, we have that $\ggamma'\in\ELMult_k(m,n)$.
We remove the portion $\gamma_{\alpha,i}|_{[\tau_{i,j}-\epsilon_1,\tau_{i,j}+\epsilon_2]}$ from the approximating embedded curve $\gamma_{\alpha,i}$ and glue in the unique free-time local minimizer with energy $k$ joining $\gamma_{\alpha,i}(\tau_{i,j}-\epsilon_1)$ and $\gamma_{\alpha,i}(\tau_{i,j}+\epsilon_2)$. We denote by $\gamma_{\alpha,i}'$ the obtained loop, and by $\ggamma_\alpha':=(\ggamma_\alpha\setminus\gamma_{\alpha,i})\cup\gamma_{\alpha,i}'$ the corresponding approximating multicurve. Clearly, $\ggamma_\alpha'\to\ggamma'$ as $\alpha\to\infty$.  The multicurve $\ggamma_{\alpha}'$ may not be embedded. If this happens, the portion of $\gamma_{\alpha,i}'$ joining $\gamma_{\alpha,i}(\tau_{i,j}-\epsilon_1)$ and $\gamma_{\alpha,i}(\tau_{i,j}+\epsilon_2)$ crosses $B_\alpha''$ (see Figure~\ref{f:approximating_curve}(b)). Thus, we slightly modify such a portion of $\gamma_{\alpha,i}'$ in order not to cross $B_\alpha''$. We denote by $\gamma_{\alpha,i}''$ the modified loop, and by $\ggamma_\alpha'':=(\ggamma_\alpha\setminus\gamma_{\alpha,i})\cup\gamma_{\alpha,i}''$ the corresponding approximating multicurve (see Figure~\ref{f:approximating_curve}(c)). This modification becomes smaller and smaller as $\alpha\to\infty$, and thus $\ggamma_\alpha''\to\ggamma'$ as $\alpha\to\infty$. This proves that \[\ggamma'\in \ELMult_k(m,n)\cap \overline{\ACMult(m)}=\Mult_k(m,n).\] However, $\SSS_k(\ggamma') < \SSS_k(\ggamma)$, since we replaced a non-smooth portion of $\gamma_i$ with a local action minimizer with energy $k$. This contradicts~\eqref{e:global_minimum_LEMMA}, and shows that every connected component of $\ggamma$ is $C^1$, thus $C^\infty$.

\begin{figure}
\begin{center}
\begin{small}
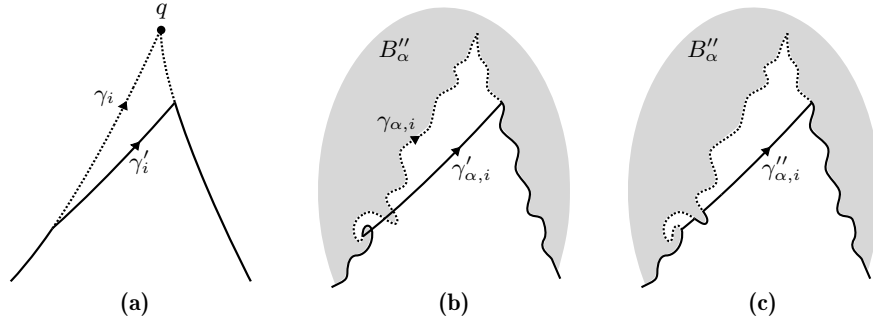 
\caption{\textbf{(a)} The connected component $\gamma_i'$ of the multicurve $\ggamma'$. \textbf{(b)} The approximating curve $\gamma_{\alpha,i}'$ intersecting $B_\alpha''$.  \textbf{(c)} The modified approximating curve $\gamma_{\alpha,i}''$ that does not enter $B_\alpha''$.}
\label{f:approximating_curve}
\end{small}
\end{center}
\end{figure}

Finally, let us show that each periodic orbit $\gamma_i:\R/\tau_i\Z\to M$ is non-iterated. Assume by contradiction that the minimal period of $\gamma_i$ is $\tau_i'$ such that the quotient $p:=\tau_i/\tau_i'$ is an integer larger than 1. Namely, $\gamma_i$ is the $p$-th iterate of a periodic orbit. But on an orientable surface, every periodic curve that is sufficiently $C^0$-close to a $p$-th iterate must have a self-intersection. Consider again a sequence $\ggamma_\alpha=(\gamma_{\alpha,1},...,\gamma_{\alpha,m})$ in $\ACMult(m)$ converging to $\ggamma$. For $\alpha$ large enough, the periodic curve $\gamma_{\alpha,i}$ has a self-intersection, which is impossible by the definition of $\ACMult(m)$.
\end{proof}

\begin{lem}\label{l:components_are_embedded}
Assume that $\ggamma=(\gamma_1,...,\gamma_m)\in \Mult_k (n)$ is a multicurve that satisfies~\eqref{e:global_minimum_LEMMA} and has no connected component that is a collapsed curve. If in addition $\ggamma \in \Mult_k (n-3)$, then $\ggamma$ is embedded in $M$.
\end{lem}

\begin{proof}
Lemma~\ref{l:components_are_periodic_orbits} tells us that the connected components of $\ggamma$ are non-iterated periodic orbits of $L$ with energy $k$. In particular they are immersed curves in $M$, for the energy hypersurface $E^{-1}(k)$ does not intersect the zero-section of $\Tan M$. Moreover, the only (self or mutual) intersections among connected components of $\ggamma$ are tangencies. Assume by contradiction that there is at least a tangency $q$ at some point along $\ggamma$, and consider all the branches of $\ggamma$ involved in this tangency: thus, there are only finitely many indices $i_1,...,i_r\in\{1,...,m\}$ and, for $h=1,...,r$, times $t_{h}\in\R/\tau_{i_h}\Z$ such that the $(i_h,t_h)$'s are pairwise distinct and satisfy
\begin{align*}
\gamma_{i_1}(t_1)=\gamma_{i_2}(t_2)=...=\gamma_{i_r}(t_r)=q.
\end{align*}
Consider a sequence of embedded multicurves $\ggamma_\alpha =(\gamma_{\alpha,1},...,\gamma_{\alpha,m})\in\ACMult(m)$ converging to $\ggamma$ as $\alpha\to\infty$. Up to permuting the indices $i_1,...,i_r$ and extracting a subsequence of $\{\ggamma_\alpha \, |\, \alpha\in\N\}$ we can assume that the curves $\gamma_{\alpha,i_1}$ and $\gamma_{\alpha,i_2}$ are adjacent near the point $q$. More precisely, this means that there exists an arbitrarily small open disk $B\subset M$ such that, for all $\alpha\in\N$, the points $\gamma_{\alpha,i_1}(t_1)$ and $\gamma_{\alpha,i_2}(t_2)$ belong to the closure of a connected component $B_\alpha$ of $B\setminus\ggamma_\alpha$, see Figure~\ref{f:strip}(a).

\begin{figure}
\begin{center}
\begin{small}
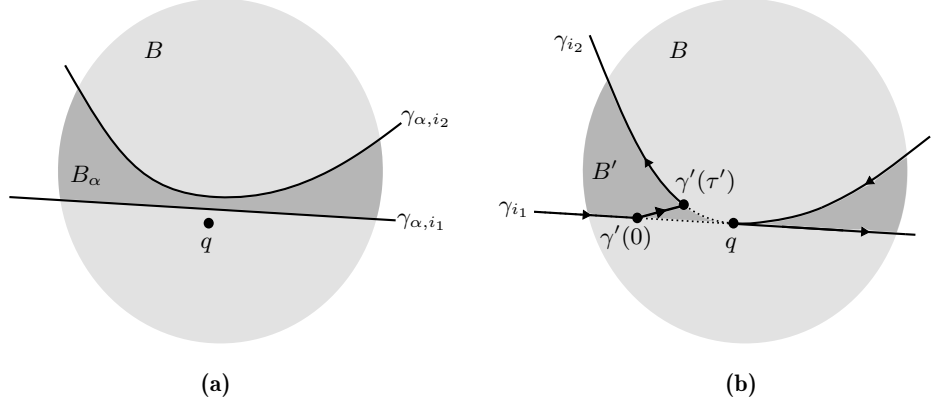 
\caption{\textbf{(a)} The disk $B$ and the strip $B_\alpha $.
\textbf{(b)} Outcome of the cut-and-paste procedure.}
\label{f:strip}
\end{small}
\end{center}
\end{figure}

The velocity vectors $\dot\gamma_{i_1}(t_1)$ and $\dot\gamma_{i_2}(t_2)$ are non-zero and parallel, i.e., $\dot\gamma_{i_1}(t_1)=\lambda\dot\gamma_{i_2}(t_2)$ for some real number $\lambda\neq0$. We claim that $\lambda<0$. Indeed, since the energy hypersurface $E^{-1}(k)$ intersects the fiber $\Tan_qM$ in a convex sphere, if $\lambda>0$ then $\lambda=1$, and therefore the periodic orbits $\gamma_{i_1}$ and $\gamma_{i_2}$ coincide. But then, for all $\alpha$ large enough, the curves $\gamma_{\alpha,i_1}$ and $\gamma_{\alpha,i_2}$ are the boundary of an open embedded annulus $A_\alpha \subset M$ containing the strip $B_\alpha $ and not containing any other connected component of $\ggamma_\alpha $. However, $\gamma_{\alpha,i_1}$ and $\gamma_{\alpha,i_2}$ are not the oriented boundary of $A_\alpha $, since they are both oriented clockwise or counterclockwise. This contradicts the fact that the multicurve $\ggamma_\alpha $ is the oriented boundary of a compact embedded surface (condition (C3) in the definition of $\ACMult(m)$). Now, consider the set
\begin{align*}
B':=\bigcap_{\beta\in\N} \ \overline{\bigcup_{\alpha>\beta} B_\alpha }.
\end{align*}
The topological boundary $\partial B'$ is contained in 
$ (\gamma_{i_1}\cup\gamma_{i_2})\cap  B.$ Fix $\epsilon>0$ small enough so that the curve $\gamma_{i_1}|_{[t_1-\epsilon,t_1]}$ has length less than $\rho_\inj$ and is contained in $B'$. 
For $\epsilon'>0$ small enough, consider the unique action minimizer $\gamma':[0,\tau']\to M$  with energy $k$ joining 
$\gamma'(0)=\gamma_{i_1}(t_1-\epsilon)$  with $\gamma'(\tau')=\gamma_{i_2}(t_2+\epsilon')$. 
By Lemma~\ref{l:local_minimizers}, $\gamma'$ is contained in $B'\setminus\{q\}$. We now modify the curves $\gamma_{i_1}$ and $\gamma_{i_2}$ as follows: along $\gamma_{i_1}$, once we reach $\gamma_{i_1}(t_1-\epsilon)$, we continue along $\gamma'|_{[0,\tau']}$ and then along $\gamma_{i_2}$; along $\gamma_{i_2}$, once we reach $\gamma_{i_2}(t_2)$, we continue along $\gamma_{i_1}$ (see Figure~\ref{f:strip}(b)). 

If $i=i_1=i_2$, this procedure replaces the connected component $\gamma_{i}$ with two curves $\gamma_{i}'$ and $\gamma_{i}''$ such that 
$\SSS_{k}(\gamma_{i}')+\SSS_{k}(\gamma_{i}'')<\SSS(\gamma_{i})$;
since $\ggamma\in\Mult_k(n-3)$, the multicurve $\ggamma'$ obtained from $\ggamma$ by replacing $\gamma_{i}$ with $\gamma_{i}'\cup\gamma_{i}''$ belongs to $\Mult_k(n)$ and satisfies $\SSS_k(\ggamma')<\SSS_k(\ggamma)$, which contradicts~\eqref{e:global_minimum_LEMMA}. 

If $i_1\neq i_2$, the procedure replaces the connected components $\gamma_{i_1}$ and $\gamma_{i_2}$ with a single curve $\gamma_{i_1i_2}'$ such that 
$\SSS_{k}(\gamma_{i_1i_2}')<\SSS_k(\gamma_{i_1})+\SSS_k(\gamma_{i_2})$;
 since $\ggamma\in\Mult_k(n-3)$, the multicurve $\ggamma'$ obtained from $\ggamma$ by replacing $\gamma_{i_1}\cup\gamma_{i_2}$ with $\gamma_{i_1i_2}'$ belongs to $\Mult_k(n)$ and satisfies $\SSS_k(\ggamma')<\SSS_k(\ggamma)$, which contradicts~\eqref{e:global_minimum_LEMMA}.
\end{proof}

\subsection{Compactness of minimizing sequences} 
\label{s:compactness}
In this subsection, we will complete the proof of Theorem~\ref{t:local_minimizers}. We begin with two preliminary lemmas.

\begin{lem}\label{l:bound_period_in_sublevels}
Let $\theta$ be the 1-form on $M$ defined by \[\theta_q(v):= L_v(q,0)v,\] where $L_v$ denotes the fiberwise derivative of $L$. Consider piecewise smooth curves $\gamma_i:\R/\tau_i\Z\to M$, for $i=1,...,m$, such that the multicurve $\ggamma=(\gamma_1,...,\gamma_m)$ belongs to $\overline{\ACMult(m)}$. If $k>e_0(L)$, then
\begin{align}\label{e:bound_period_in_sublevels}
\tau_1+...+\tau_m \leq \frac{\SSS_k(\ggamma)  + \int_M |\diff\theta|}{k-e_0(L)}.
\end{align}
\end{lem}

\begin{proof}
The function 
$v\mapsto L(q,v)-\theta_q(v)$
 has a global minimum at the origin, and therefore
\begin{align*}
L(q,v)-\theta_q(v)+e_0(L) \geq L(q,0)  + e_0(L) =  -E(q,0) + e_0(L) \geq 0.
\end{align*}
We infer
\begin{align}\label{e:positive_corrected_action}
\SSS_{e_0(L)}(\ggamma) -\int_{\ggamma} \theta 
=
\sum_{i=1}^{m} \int_0^{\tau_i}
\Big[ L(\gamma_i,\dot\gamma_i) - \theta_{\gamma_i}(\dot\gamma_i) + e_0(L)\Big]\,\diff t\geq 0.
\end{align}
Let $\{\ggamma_\alpha\ |\ \alpha\in\N\}$ be a sequence in $\ACMult(m)$ converging to $\ggamma$ as $\alpha\to\infty$. If $\Sigma_\alpha$ is the compact submanifold whose oriented boundary is $\ggamma_\alpha$, Stokes' Theorem implies
\begin{align}\label{e:Stokes}
\int_{\ggamma}\theta
=\lim_{\alpha\to\infty}\int_{\ggamma_\alpha}\theta 
=\lim_{\alpha\to\infty}\int_{\Sigma_\alpha} \diff\theta \leq \int_{M} |\diff\theta|.
\end{align}
We set $\tau:=\tau_1+...+\tau_m$. By~\eqref{e:positive_corrected_action} and~\eqref{e:Stokes}, we have
\begin{align*}
(k-e_0(L))\,\tau  &\leq  \SSS_{e_0(L)}(\ggamma)-\int_{\ggamma} \theta + (k-e_0(L))\,\tau\\
& =  \SSS_k(\ggamma) - \int_{\ggamma} \theta\\
& \leq  \SSS_k(\ggamma) - \int_{M} |\diff\theta|,
\end{align*}
which implies~\eqref{e:bound_period_in_sublevels}.
\end{proof}

\begin{lem}\label{l:upper_bound_length}
For each $k>e_0(L)$ and $s\in\R$, there exists $\ell_{\max}(k,s)>0$ such that every multicurve $\ggamma\in\cup_{n\in\N}\Mult_k(n)$ satisfying $\SSS_k(\ggamma)\leq s$ has total length less than or equal to $\ell_{\max}(k,s)$.
\end{lem}

\begin{proof}
Since the energy level $E^{-1}(k)$ is compact, there exists $r>0$ such that $g_q(v,v)\leq r^2$ for all $(q,v)\in E^{-1}(k)$. This, together with Lemma~\ref{l:bound_period_in_sublevels}, implies that the length of $\ggamma=(\gamma_1,...,\gamma_m)$ can be bounded from above as
\begin{align*}
\mathrm{length}(\ggamma) 
&= \sum_{i=1}^{m}\ \int_0^{\tau_i}  \sqrt{g_{\gamma_i(t)}(\dot\gamma_i(t),\dot\gamma_i(t))}\, \diff t\\
&\leq \big (\tau_1+...+\tau_m\big )\,  r  \\
& \leq  \frac{\big (\SSS_k(\ggamma) + \int_M |\diff\theta|\big )\, r}{k-e_0(L)}\\&\leq  \frac{\big (s + \int_M |\diff\theta|\big )\, r}{k-e_0(L)}\\
& =:  \ell_{\max}(k,s).
\qedhere
\end{align*}
\end{proof}

Let us consider the energy value $\cu(L)$, which is defined as the Ma\~n\'e critical value of the lift of $L$ to the universal cover of $M$, or equivalently as the minimum $k$ such that $\SSS_k(\gamma)<0$ for some absolutely continuous periodic curve $\gamma:\R/\tau\Z\to M$ that is contractible. It is easy to see that $e_0(L)\leq \cu(L)$. 

We have already seen in Corollary~\ref{c:small_loops} that every absolutely continuous periodic curve $\gamma:\R/\tau\Z\to M$ contained in a Riemannian ball of diameter $\rho_\inj$ is contractible and satisfies $\SSS_k(\gamma)\geq0$. On the other hand, if $k\in(e_0(L),\cu(L))$, up to passing to a finite cover of the configuration space $M$ we can always find embedded periodic curves with negative action.

\begin{lem}\label{l:minima_are_negative}
If $k\in(e_0(L),\cu(L))$, there exists a finite cover $M'\to M$ with the following property. If we lift the Tonelli Lagrangian $L$ to $\Tan M'$, and denote by $\Mult_k'(m,n)$ and $\SSS_k'$ the associated spaces of multicurves and free-period action functional, then for $n$ large enough there exists a periodic curve $\gamma\in\Mult_k'(1,n)$ with $\SSS_k'(\gamma) < 0$.
\end{lem}

\begin{proof}
Let $\widetilde M$ be the universal cover of our configuration space $M$. We lift the Tonelli Lagrangian $L$ to a function $\widetilde L:\Tan \widetilde M\to\R$, and we denote by $\widetilde{\SSS}_k$ the associated free-period action functional. Since $k<\cu(L)$, there exists an absolutely continuous loop $\gamma_0:\R/\tau\Z\to \widetilde M$ such that $\widetilde{\SSS}_k(\gamma_0)<0$. Let $n\in\N$ be large enough so that 
$\dist(\gamma_0(s),\gamma_0(t))<\rho_\inj$ whenever $|s-t|\leq 1/n$,
where $\rho_\inj>0$ is the constant of Lemma~\ref{l:local_minimizers}(ii). We replace each portion 
$\gamma_0|_{[i/n,(i+1)/n]}$
of our curve with the unique free-time local minimizer with energy $k$ joining the endpoints. The resulting curve, which we will still denote by $\gamma_0:\R/\tau\Z\to \widetilde M$ (possibly for a different period $\tau$ than before), is a piecewise solution of the Euler-Lagrange equation of $\widetilde L$, and in particular is piecewise smooth. Up to perturbing the vertices $\gamma_0(i/n)$, for $i=0,...,k-1$, we can assume that $\gamma_0$ has only finitely many self-intersections, all of which are double points.

We claim that we can find such a $\gamma_0$ without self-intersections. Suppose that our $\gamma_0$ has at least one double point (otherwise we are done). Let $t_0\in[0,\tau)$ be the smallest time such that $\gamma_0(t_0)$ is a double point of $\gamma_0$, and let $t_1>t_0$ be the smallest time such that $\gamma_0(t_1)=\gamma_0(t_0)$. We define the piecewise smooth curves 
$\gamma_1:\R/(t_1-t_0)\Z\to  \widetilde M$ and $\gamma_2:\R/(\tau-t_1+t_0)\Z\to  \widetilde M$
by
\begin{align*}
\gamma_1(t) &= \gamma_0(t_0+t),\qquad\forall t\in[0,t_1-t_0],\\
\gamma_2(t) &= \gamma_0(t_1+t),\qquad\forall t\in[0,\tau-t_1+t_0].
\end{align*}
Notice that 
\begin{align*}
\gamma_0 (\R/\tau\Z) &= \gamma_1(\R/(t_1-t_0)\Z) \cup \gamma_2(\R/(\tau-t_1+t_0)\Z),\\
\widetilde{\SSS}_k(\gamma_0) &= \widetilde{\SSS}_k(\gamma_1) + \widetilde{\SSS}_k(\gamma_2),
\end{align*}
and that both $\gamma_1$ and $\gamma_2$ have strictly less double points than $\gamma_0$. Since $\widetilde{\SSS}_k(\gamma_0)<0$, we must have $\widetilde{\SSS}_k(\gamma_1)<0$ or $\widetilde{\SSS}_k(\gamma_2)<0$, say $\widetilde{\SSS}_k(\gamma_1)<0$. If $\gamma_1$ has no double points, we are done. Otherwise, we repeat the whole procedure with the curve $\gamma_1$. After a finite numbers of iterations of this procedure we end up with a periodic curve without self-intersections, which we still denote by $\gamma_0:\R/\tau\Z\to\widetilde M$, such that $\widetilde{\SSS}_k(\gamma_0)<0$.

We denote by $\gamma:\R/\tau\Z\to M$ the projection of $\gamma_0$ to $M$. The periodic curve $\gamma$ is contractible, but may have finitely many self-intersections. We can get rid of the self-intersections by passing to a suitable finite cover of $M$ as follows. Let $0<b-a<\tau$ such that $\gamma(a)=\gamma(b)=:q$. Since $M$ is a closed surface, its fundamental group is residually finite \cite{Hempel:1972sp}. Therefore, there exists a normal subgroup of finite index $G\subset \pi_1(M,q)$ that does not contain $[\gamma|_{[a,b]}]$. We denote by $p:M''\to M$ the finite cover such that $p_*(\pi_1(M'',q''))=G$ for any $q''\in p^{-1}(q)$. We lift $\gamma$ to a contractible periodic curve $\gamma'':\R/\tau\Z\to M''$ such that $\gamma''(a)=q''$. Since $[\gamma|_{[a,b]}]\not\in G$, we have $\gamma''(a)\neq\gamma''(b)$. We now repeat this argument for the curve $\gamma''$. After finitely many iterations of this procedure, we obtain a finite cover $M'$ of $M$ such that $\gamma_0$ projects to a periodic curve without self-intersections $\gamma':\R/\tau\Z\to M'$ with $\SSS_k'(\gamma')<0$. Since $\gamma'$ is contractible, in particular $\gamma'\in\Mult_k'(1,n)$.
\end{proof}

From now on, we consider a fixed arbitrary energy value 
\begin{align*}
k\in(e_0(L),\cu(L)).
\end{align*}
We will need to apply Lemma~\ref{l:upper_bound_length} in the case where $s=0$, and thus we will simply write \[\ell_{\max}:=\ell_{\max}(k,0).\]
We denote by $\nneg$ the minimal positive integer so that the conclusion of Lemma~\ref{l:minima_are_negative} holds for all $n\geq \nneg$. In order to simplify the notation, we will replace the surface $M$ by its finite cover $M'$ given by Lemma~\ref{l:minima_are_negative}, so that $\SSS_k(\gamma)<0$ for some $\gamma\in\Mult_k(1,\nneg)$. By Lemma~\ref{l:local_minimizers}(ii), $\gamma$ cannot be contained in the open Riemannian ball $B_g(\gamma(0),\rho_\inj)$, and therefore its length must be at least $2\rho_\inj$. This implies 
\begin{align*}
 \ell_{\max}/\rho_{\inj} \geq 2.
\end{align*}

The only missing ingredient to complete the proof of Theorem~\ref{t:local_minimizers} is a compactness result: we wish to show that there exists $n_{\min}\in\N$ such that, for all $n\geq n_{\min}$, a minimizer of $\SSS_k$ over the space $\Mult_k(n)$ belongs to $\Mult_k(n_{\min})$, unless it contains  collapsed connected components that we can always throw away. The proof of such a statement will take most of this subsection.

Let $n\geq\nneg$. Since $\Mult_k(n)$ is compact and the free-period action functional $\SSS_k:\Mult_k(n)\to\R$ is continuous, there exists $\ggamma=(\gamma_1,...,\gamma_m)\in\Mult_k(n)$ such that
\begin{align}\label{e:global_minimum}
 \SSS_k(\ggamma) = \min_{\Mult_k(n)} \SSS_k<0.
\end{align}
Notice that some connected component of $\ggamma$ may be a collapsed curve $\gamma_i:\{0\}\to M$. However, in this case $\SSS_k(\gamma_i)=0$. Since $\SSS_k(\ggamma)<0$, there exists a connected component $\gamma_j$ of $\ggamma$ such that $\SSS_k(\gamma_j)<0$. Such a $\gamma_j$ is not a collapsed curve. Therefore, after removing all the collapsed connected components of $\ggamma$ we are left with a multicurve in $\Mult_k(n)$, which we still denote by $\ggamma$, that satisfies~\eqref{e:global_minimum}.

\begin{lem}\label{l:bound_connected_components}
Each connected component of the multicurve $\ggamma$ has length larger than or equal to $\rho_\inj$. In particular, there are at most 
$\ell_{\max}/\rho_\inj$ many such connected components.
\end{lem}

\begin{proof}
Assume that some $\gamma_i$ has length smaller than $\rho_\inj$, and in particular it is contained in a Riemannian ball $B\subset M$ of diameter $\rho_\inj$. Let $\gamma_{i_1},\gamma_{i_2},...,\gamma_{i_r}$ be all the connected components of $\ggamma$ that are entirely contained in $B$. By Corollary~\ref{c:small_loops}, each $\gamma_{i_j}$ is contractible and $\SSS_k(\gamma_{i_j})>0$. Therefore, the multicurve 
$\ggamma':= \ggamma\setminus\{\gamma_{i_1},\gamma_{i_2},...,\gamma_{i_r}\}$
still belongs to $\Mult_k(n)$ and satisfies $\SSS_k(\ggamma')<\SSS_k(\ggamma)$, which contradicts~\eqref{e:global_minimum}. This, together with the fact that the length of $\ggamma$ is at most $\ell_{\max}$, implies the lemma.
\end{proof}

\noindent Once $\ggamma$ is fixed, if needed, we reduce $n$ so that 
\begin{align}\label{e:minimal_number_vertices}
\ggamma\in \Mult_k(n)\setminus\Mult_k(n-1). 
\end{align}
We write $\ggamma=(\gamma_1,...,\gamma_m)$. We recall that, for each   component $\gamma_i:\R/\tau_i\Z\to M$, there exist  
$0=\tau_{i,0}<\tau_{i,1}<...<\tau_{i,n_i}=\tau_i$
such that
$\dist\big (\gamma_i(\tau_{i,j}),\gamma_i(\tau_{i,j+1})\big )\leq \rho_\inj$ and the restriction $\gamma_i|_{[\tau_{i,j},\tau_{i,j+1}]}$ is the unique free-time local minimizer with energy $k$ joining the endpoints. Condition~\eqref{e:minimal_number_vertices} implies
$n_1+...+n_m=n$.
We will call \textbf{vertices} the times $\tau_{i,j}$, and \textbf{segments} the portions $\gamma_i|_{[\tau_{i,j},\tau_{i,j+1}]}$. 
The decomposition of the multicurve $\ggamma$ in segments is clearly not unique: for instance, if a connected component $\gamma_i$ is smooth (and thus a periodic orbit of the Euler-Lagrange flow with energy $k$), we may be able to shift all the vertices around the curve. We say that a segment $\gamma_i|_{[\tau_{i,j},\tau_{i,j+1}]}$ is \textbf{short} if 
$$\dist\big (\gamma_i(\tau_{i,j}),\gamma_i(\tau_{i,j+1})\big )<\rho_\inj,$$
whereas we say that it is \textbf{long} if  
$$\dist\big (\gamma_i(\tau_{i,j}),\gamma_i(\tau_{i,j+1})\big )=\rho_\inj.$$ 
We choose a decomposition of $\ggamma$ in segments so that:
\begin{itemize}
\item[(S1)] each smooth connected component $\gamma_i$ contains at most one short segment;
\item[(S2)] on each connected component $\gamma_i$ that is not smooth, if $\gamma_i|_{[\tau_{i,j},\tau_{i,j+1}]}$ is a short segment then $\gamma_i$ is not $C^1$ at $\tau_{i,j+1}$.
\end{itemize}

\begin{lem}\label{l:short_segments}
On each connected component $\gamma_i$ that is not smooth, every short segment $\gamma_i|_{[\tau_{i,j},\tau_{i,j+1}]}$ contains a tangency of $\ggamma$, that is, there exists $h\in\{1,...,m\}$ and $l\in\{0,...,n_h\}$ such that $\gamma_i|_{[\tau_{i,j},\tau_{i,j+1}]}$ and $\gamma_h|_{[\tau_{h,l},\tau_{h,l+1}]}$ are distinct segments of $\ggamma$ and have a mutual intersection.
\end{lem}

\begin{proof}
Assume by contradiction that the short segment $\gamma_i|_{[\tau_{i,j},\tau_{i,j+1}]}$ does not intersect any other segment of $\ggamma$. There exists an open neighborhood $B\subset M$ of $\gamma_i([\tau_{i,j},\tau_{i,j+1}])$ such that no connected component of $\ggamma$ other than $\gamma_i$ intersects $B$, and if $\gamma_i(t)\in B$ then $t\in(\tau_{i,j-1},\tau_{i,j+2})$, see Figure~\ref{f:short}(a). Since, by assumption, 
$\dist\big (\gamma_i(\tau_{i,j}),\gamma_i(\tau_{i,j+1})\big )<\rho_\inj$,
 for all $\epsilon>0$ small enough we still have 
$\dist\big (\gamma_i(\tau_{i,j}),\gamma_i(\tau_{i,j+1}+\epsilon)\big )<\rho_\inj$. 
We denote by $\gamma_\epsilon':[0,\omega_\epsilon]\to M$ the unique free-time local minimizer with energy $k$ joining $\gamma_\epsilon'(0)=\gamma_i(\tau_{i,j})$ and $\gamma_\epsilon'(\omega_\epsilon)=\gamma_i(\tau_{i,j+1}+\epsilon)$. Up to further reducing $\epsilon$ such a curve is entirely contained in $B$ and, by Lemma~\ref{l:local_minimizers}, intersects $\gamma_i([\tau_{i,j},\tau_{i,j+2}])$ only at the endpoints $\gamma_\epsilon'(0)$ and $\gamma_\epsilon'(\omega_\epsilon)$. However, $\gamma_\epsilon'$ may intersect $\gamma_i([\tau_{i,j-1},\tau_{i,j}))$ (this happens for an arbitrarily small $\epsilon>0$ only if $\gamma_i$ has a cusp at $\tau_{i,j}$, that is, $\dot\gamma_i(\tau_{i,j}^-)$ and $\dot\gamma_i(\tau_{i,j}^+)$ are parallel and point in opposite directions). Consider the largest time $\alpha_\epsilon\in[0,\omega_\epsilon)$ such that $\gamma_\epsilon'(\alpha_\epsilon)$ lies on $\gamma_i((\tau_{i,j-1},\tau_{i,j}])$, and call $\sigma_\epsilon\in(\tau_{i,j-1},\tau_{i,j}]$ the unique time such that 
$q_0:= \gamma_\epsilon'(\alpha_\epsilon)=\gamma_i(\sigma_\epsilon)$.
The restricted curve $\gamma_\epsilon'|_{[\alpha_\epsilon,\omega_\epsilon]}$ intersects $\ggamma$ only at the endpoints $\gamma_\epsilon'(\alpha_\epsilon),\gamma_\epsilon'(\omega_\epsilon)$ and, by Lemma~\ref{l:local_minimizers}, is still a unique free-time local minimizer with energy $k$. If $\epsilon$ is small enough, both curves $\gamma_\epsilon'|_{[\alpha_\epsilon,\omega_\epsilon]}$ and $\gamma_i|_{[\sigma_\epsilon,\tau_{i,j+1}+\epsilon]}$ are contained in the open set $U_{q_0}$  given by Lemma~\ref{l:local_minimizers} and join $q_0$ with $\gamma_i(\tau_{i,j+1}+\epsilon)$. Therefore
\begin{align}\label{e:lowering_action}
\int_{\alpha_\epsilon}^{\omega_\epsilon} \Big [ L(\gamma_\epsilon'(t),\dot\gamma_\epsilon'(t)) + k \Big]\,\diff t
<
\int_{\sigma_\epsilon}^{\tau_{i,j+1}+\epsilon} \Big [ L(\gamma_i(t),\dot\gamma_i(t)) + k \Big]\,\diff t.
\end{align}
We remove the portion $\gamma_i|_{[\sigma_\epsilon,\tau_{i,j+1}+\epsilon]}$ from $\gamma_i$ and glue in the unique free-time local minimizer $\gamma_\epsilon'|_{[\alpha_\epsilon,\omega_\epsilon]}$. We denote by $\gamma_i'$ the modified curve, and by $\ggamma'$ the multicurve obtained from $\ggamma$ by replacing $\gamma_i$ with $\gamma_i'$, see Figure~\ref{f:short}(b). 

Clearly, $\ggamma'$ belongs to $\Mult_k(n)$. Furthermore, the inequality in~\eqref{e:lowering_action} implies that $\SSS_k(\ggamma')<\SSS_k(\ggamma)$, which contradicts~\eqref{e:global_minimum}.
\end{proof}

\begin{figure}
\begin{center}
\begin{small}
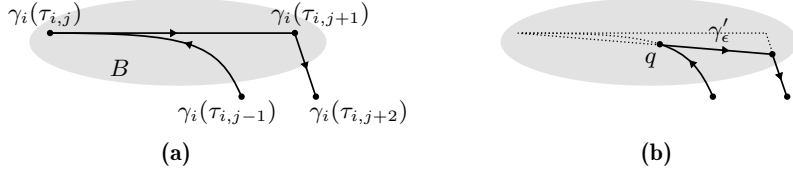 
\caption{\textbf{(a)} The neighborhood $B$ of the segment $\gamma_i|_{[\tau_{i,j},\tau_{i,j+1}]}$.
\textbf{(b)} The modified curve $\gamma_i'$, with the portion $\gamma_\epsilon'|_{[\alpha_\epsilon,\omega_\epsilon]}$ glued in.}
\label{f:short}
\end{small}
\end{center}
\end{figure}

We fix, once for all, a sequence $\{\ggamma_\alpha=(\gamma_{\alpha,1},...,\gamma_{\alpha,m})\ |\ \alpha\in\N\}\subset\ACMult(m)$ such that $\ggamma_\alpha\to \ggamma$ as $\alpha\to\infty$. Each multicurve $\ggamma_\alpha$ is without self-intersections, but a priori that is not the case for the limit curve $\ggamma$. Since more than two branches of $\ggamma$ may intersect (tangentially) at a same point, we need the following definition. 
We say that $(i_1,t_1),(i_2,t_2)$ is an \textbf{adjacent tangency} of $\ggamma$ at $q\in M$ when $(i_1,t_1)\neq(i_2,t_2)$, $q=\gamma_{i_1}(t_1)=\gamma_{i_2}(t_2)$, and for all sufficiently small neighborhoods $B\subset M$ of $q$ and for all $\alpha\in\N$ large enough the points $\gamma_{\alpha,i_1}(t_1)$ and $\gamma_{\alpha,i_2}(t_2)$ belong to the closure of a same connected component of $B\setminus\ggamma_\alpha$. 
We will say that two segments $\gamma_{i}|_{[\tau_{i,j},\tau_{i,j+1}]}$ and $\gamma_{h}|_{[\tau_{h,l},\tau_{h,l+1}]}$ contain a \textbf{mutual adjacent tangency}, and we will write 
\[\gamma_{i}|_{[\tau_{i,j},\tau_{i,j+1}]}\  \asymp\ \gamma_{h}|_{[\tau_{h,l},\tau_{h,l+1}]},\] 
if there exist $t_1\in[\tau_{i,j},\tau_{i,j+1}]$ and $t_2\in[\tau_{h,l},\tau_{h,l+1}]$ such that $(i,t_1),(h,t_2)$ is an adjacent tangency of $\ggamma$. 

For all $i,h\in\{1,...,m\}$ we introduce the set
\begin{align*}
A_{i,h}:= 
\big\{
(j,l)\ \big|\ 
\gamma_{i}|_{[\tau_{i,j},\tau_{i,j+1}]}\asymp\gamma_{h}|_{[\tau_{h,l},\tau_{h,l+1}]}
\big\}.
\end{align*}
We denote by $\# A_{i,h}$ the cardinality of this set. 

\begin{lem}\label{l:bound_A_i_h}
The total number $n$ of segments of the multicurve $\ggamma$ is bounded as
\begin{align*}
n\leq 2\frac{\ell_{\max}}{\rho_\inj} + 2 \sum_{i\leq h}  \# A_{i,h}.
\end{align*}
\end{lem}

\begin{proof}
We denote by $n_{\mathrm{long}}$ and $n_{\mathrm{short}}$ the number of long segments and the number of short segments of $\ggamma$ respectively, so that $n=n_{\mathrm{long}}+n_{\mathrm{short}}$. The first one can be  bounded as
\begin{align*}
 n_{\mathrm{long}} \leq \ell_{\max}/\rho_\inj.
\end{align*}
Let $n_{\mathrm{short}}'$ be the number of short segments in the smooth connected components of the multicurve $\ggamma$, and $n_{\mathrm{short}}'':=n_{\mathrm{short}}-n_{\mathrm{short}}'$ be the number of short segments in the non-smooth connected components. By our choice of the decomposition of $\ggamma$ in segments and by Lemma~\ref{l:bound_connected_components}, we have
\begin{align*}
n_{\mathrm{short}}'\leq m \leq \ell_{\max}/\rho_\inj.
\end{align*}
Finally, Lemma~\ref{l:short_segments} implies that
\begin{equation*}
n_{\mathrm{short}}'' \leq 2 \sum_{i\leq h}  \# A_{i,h}. \qedhere
\end{equation*}
\end{proof}

In order to provide an upper bound for $n$, we are left to bound the cardinality of the set $A_{i,h}$. We first provide a uniform bound of the number of tangencies at an arbitrary given point $q\in M$. For $i=1,...,m$, we define
\begin{align*}
T_i(q):= \big \{ t\in \R/\tau_i\Z\ \big |\ \gamma_i(t)=q \big \}.
\end{align*}
 
\begin{lem}\label{l:bound_on_valence}
If $T_i(q)$ is not empty, let $t_0<t_1<...<t_{u-1}$ be its ordered elements, with $0\leq t_0<t_{u-1}<\tau_i$. For each $j\in\Z/u\Z$, we see $[t_j,t_{j+1}]$ as a compact subset of the circle $\R/\tau_i\Z$, which is an interval if $u>1$ or is the entire circle if $u=1$. Each closed curve $\gamma_i|_{[t_j,t_{j+1}]}$ is not contained in the Riemannian ball $B_g(q,\rho_\inj/2)$, and thus has length larger than or equal to $\rho_\inj$. In particular, 
$$\# T_i(q)\leq \ell_{\max}/\rho_\inj.$$
\end{lem}

\begin{proof}
Assume by contradiction that, for some $j$, the curve $\gamma_i|_{[t_j,t_{j+1}]}$ is contained in $B_g(q,\rho_\inj/2)$. By Lemma~\ref{l:local_minimizers}, for every $h=1,...,m$, if any connected component $\gamma_h$ restricts to a loop $\gamma_h|_{[a,b]}$, for some $a<b$, that is  contained in the Riemannian ball $B_g(q,\rho_\inj/2)$ and satisfies $\gamma_h(a)=\gamma_h(b)=q$, then $\gamma_h$ must possess a vertex $t\in[a,b]$; moreover, by Corollary~\ref{c:small_loops}, $\gamma_h|_{[a,b]}$ is contractible and satisfies $\SSS_k(\gamma_h|_{[a,b]})>0$; in this case, we cut the portion $\gamma_h|_{[a,b]}$ from $\gamma_h$ and replace it by a vertex. 
We repeat this procedure iteratively as many times as possible, and we produce a multicurve $\ggamma'$ that still belongs to $\Mult_k(n)$ but satisfies $\SSS_k(\ggamma')<\SSS_k(\ggamma)$, contradicting~\eqref{e:global_minimum}. This argument, together with the fact that $\gamma_i$ has length at most $\ell_{\max}$, provides the 
desired upper bound for $\# T_i(q)$.
\end{proof}

Now, for all pairs of distinct segments $\gamma_{i}|_{[\tau_{i,j},\tau_{i,j+1}]}\asymp\gamma_{h}|_{[\tau_{h,l},\tau_{h,l+1}]}$, we fix a point $q_{(i,j),(h,l)}\in M$ at which the two segments have an adjacent tangency (if there are several such points, we choose one of them arbitrarily), thus forming the set
\begin{align*}
B_{i,h}:= \big\{ q_{(i,j),(h,l)}\in M\ \big|\ (j,l)\in A_{i,h} \big\}.
\end{align*}
Clearly, $\# B_{i,h} \leq \# A_{i,h}$. The following lemma states that also the opposite inequality holds, up to a multiplicative constant. 

\begin{lem}\label{l:A_i_h_B_i_h}  
$ \# A_{i,h}\leq \# B_{i,h} \cdot  (\ell_{\max}/\rho_\inj)^2$.
\end{lem}

\begin{proof}
Consider the map 
$Q:A_{i,h}\to   B_{i,h}$ given by $Q(j,l)=q_{(i,j),(h,l)}$,
and an arbitrary $q\in B_{i,h}$. We denote by $\pi_1(j,l)=j$ and $\pi_2(j,l)=l$ the projections onto the first and second factors respectively. Notice that
\begin{align*}
\#(Q^{-1}(q)) \leq \#(\pi_1(Q^{-1}(q)))\cdot\#(\pi_2(Q^{-1}(q))).
\end{align*}
The set $\pi_1(Q^{-1}(q))$ contains precisely all those $j\in\{1,...,n_i\}$ such that the segment $\gamma_i|_{[\tau_{i,j},\tau_{i,j+1}]}$ has an adjacent tangency at $q$ with some segment of $\gamma_h$. Therefore, by Lemma~\ref{l:bound_on_valence}, we have
\begin{align*}
\#(\pi_1(Q^{-1}(q)))  \leq  \#(T_i(q)) \leq \ell_{\max}/\rho_\inj.
\end{align*}
Analogously
\begin{align*}
\#(\pi_2(Q^{-1}(q)))\leq \#(T_h(q)) \leq \ell_{\max}/\rho_\inj.
\end{align*}
We conclude that
\begin{align*}
\# A_{i,h}  &\leq  \# B_{i,h} \cdot \max_{q\in B_{i,h}} \Big( \#(\pi_1(Q^{-1}(q))) \cdot \#(\pi_2(Q^{-1}(q))) \Big)\\
& \leq  \# B_{i,h} \cdot  (\ell_{\max}/\rho_\inj )^2.
\qedhere
\end{align*}
\end{proof}

\begin{lem}\label{l:estimate_B_i_h}
For all distinct $i,h\in\{1,...,m\}$, we have
$
\# B_{i,h}
\leq
\big\lfloor
1+2\ell_{\max}/\rho_\inj
\big\rfloor
!
$
\end{lem}

\begin{proof}
Let $0\leq t_0<...<t_{u-1}<\tau_i$ and $s_0,...,s_{u-1}\in[0,\tau_h)$ be such that:
\begin{itemize}
\item $B_{i,h}=\{\gamma_i(t_0),..., \gamma_i(t_{u-1})\}=\{\gamma_h(s_0),..., \gamma_h(s_{u-1})\}$;
\item $\gamma_i(t_j)\neq\gamma_i(t_{l})$ if $j\neq l$;
\item $(i,t_j),(h,s_j)$ is an adjacent tangency of $\ggamma$ for each $j=0,...,u-1$.
\end{itemize}
At this point we employ a combinatorial statement due to Taimanov \cite[Proposition~1]{Taimanov:1992sm}: for every finite set $W\subset\R$ with cardinality $\# W\geq w!$ and every injective map $f:W\to\R$ there exists a subset $W'\subset W$ of cardinality $\# W'\geq w$ such that $f|_{W'}$ is monotone. Let $w\in\N$ be the integer such that 
\begin{align*}
w! \leq  \# B_{i,h}  <(w+1)!
\end{align*}
The combinatorial statement implies that there exist 
$0 \leq j_0<...<j_{w-1}\leq  \# B_{i,h}$
such that either 
\begin{align}\label{e:preserved_order}
s_{j_0}<s_{j_1}<...<s_{j_{w-1}}
\end{align}
or
\begin{align}\label{e:reversed_order}
s_{j_0}>s_{j_1}>... >s_{j_{w-1}}.
\end{align}

Assume that~\eqref{e:preserved_order} holds. For each $v\in\Z/w\Z$ we consider the closed curve 
\begin{align*}
\zeta_v:= \gamma_i |_{[t_{j_v},t_{j_{v+1}}]}  *  \overline{\gamma_h |_{[s_{j_v},s_{j_{v+1}}]}},
\end{align*}
where $*$ denotes the concatenation of paths and the overline bar changes the orientation of a loop. We claim that, if $w>2$, the loop $\zeta_v$ is not contractible. Indeed, assume that $\zeta_v$ is the boundary of a contractible compact subset $K\subset M$. Since $\gamma_i$ and $\gamma_h$ have adjacent tangencies at the points 
$q_v:= \gamma_i(t_{j_v})=\gamma_h(s_{j_v})$ and $q_{v+1}:= \gamma_i(t_{j_{v+1}})=\gamma_h(s_{j_{v+1}})$, 
the point $\gamma_h(s_{j_{v+2}})$ is forced to lie inside $K$. But then $\gamma_h(s_{j_{v+2}})$ must coincide with $\gamma_h(s_{j_v})$, as otherwise $\gamma_i$ and $\gamma_h$ could not have an adjacent tangency at the point $\gamma_i(t_{j_{v+2}})=\gamma_h(s_{j_{v+2}})$, see Figure~\ref{f:trapped_point}. Therefore $v=v+2$ and $w=2$. By Corollary~\ref{c:small_loops}, each loop $\zeta_v$ must have length larger than or equal to $\rho_\inj$. Since 
\begin{align*}
\sum_{v\in\Z/w\Z} \mathrm{length}(\zeta_v)=\mathrm{length}(\gamma_i)+\mathrm{length}(\gamma_h)\leq \ell_{\max},
\end{align*}
we conclude $w \leq \max\{2,\ell_{\max}/\rho_\inj\} = \ell_{\max}/\rho_{\inj}$, and therefore
\begin{align*}
\# B_{i,h}
\leq
\big\lfloor
1+\ell_{\max}/\rho_{\inj}
\big\rfloor !
\end{align*}

\begin{figure}
\begin{center}
\begin{small}
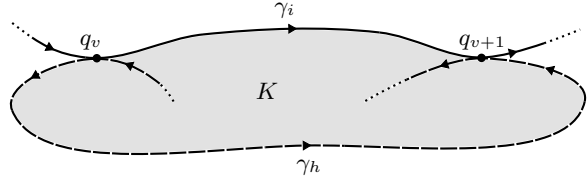 
\caption{Situation under~\eqref{e:preserved_order} with $\zeta_v$ contractible. The connected components $\gamma_i$ and $\gamma_h$ are the solid and dashed ones respectively. A priori, the compact set $K$ may even have empty interior, but still the portion of $\gamma_h$ that goes from $q_{v+1}$ to $q_{v}$ must be contained in $K$.}
\label{f:trapped_point}
\end{small}
\end{center}
\end{figure}

Assume that~\eqref{e:reversed_order} holds instead. For each $v\in\Z/w\Z$ we consider the closed curve 
\begin{align*}
\xi_v:= 
\gamma_i|_{[t_{j_v},t_{j_{v+1}}]}*\gamma_h|_{[s_{j_v},s_{j_{v+1}}]}.
\end{align*}
We denote by $w_{\mathrm{short}}$ the number of closed curves $\xi_v$, for $v\in\Z/w\Z$, whose length is smaller than $\rho_\inj$, and we set $w_{\mathrm{long}}:=w-w_{\mathrm{short}}$. Since the total length of the multicurve $\ggamma$ is at most $\ell_{\max}$, we have 
$w_{\mathrm{long}}\leq \ell_{\max}/\rho_\inj$.
We claim that
\begin{align*}
\# B_{i,h}
\leq
\big\lfloor
1+2\ell_{\max}/\rho_{\inj}
\big\rfloor !
\end{align*}
Indeed, assume by contradiction that this inequality does not hold, so that $w>2\ell_{\max}/\rho_\inj$.  This implies that $w_{\mathrm{short}}>w_{\mathrm{long}}$, and therefore there exists $v\in\Z/w\Z$ such that both $\xi_v$ and $\xi_{v+1}$ have length smaller than $\rho_\inj$. We set, for $l=0,1,2$,
$$q_{v+l} :=\gamma_i(t_{j_{v+l}})=\gamma_h(s_{j_{v+l}}) \in M.$$
By Corollary~\ref{c:small_loops}, the loops $\xi_v$ and $\xi_{v+1}$ are the boundary of some contractible compact subsets $K_v\subseteq B_g(q_v,\rho_\inj/2)$ and $K_{v+1} \subseteq  B_g(q_{v+1},\rho_\inj/2)$ respectively, and satisfy $\SSS_k(\xi_v)>0$ and $\SSS_k(\xi_{v+1})>0$, see Figure~\ref{f:consecutive_short_loops}(a). Corollary~\ref{c:small_loops} further implies that no connected component of the multicurve $\ggamma$ is entirely contained in $K_v$ or in $K_{v+1}$, as otherwise by removing from $\ggamma$ all these connected components we would obtain a new multicurve $\ggamma'$  still belonging to $\Mult_k(n)$ and satisfying $\SSS_k(\ggamma')<\SSS_k(\ggamma)$, which would contradict~\eqref{e:global_minimum}. Notice that, by the very definition of the $t_j$'s, there is at least one vertex of $\ggamma$ in 
$\gamma_i |_{[t_{j_v},t_{j_{v+1}}]}$ or in $\gamma_h |_{[s_{j_{v+1}},s_{j_{v}}]}$,
and analogously there is at least one vertex of $\ggamma$ in 
$\gamma_i |_{[t_{j_{v+1}},t_{j_{v+2}}]}$ or  in $\gamma_h |_{[s_{j_{v+2}},s_{j_{v+1}}]}$.
We now modify the connected components $\gamma_i$ and $\gamma_j$ as follows: along $\gamma_i$, once we reach $q_v$, we continue along $\gamma_h$; along $\gamma_h$, once we reach $q_{v+2}$, we continue along $\gamma_i$; finally we remove the portions $\gamma_i|_{[t_{j_{v}},t_{j_{v+2}}]}$ and $\gamma_h|_{[s_{j_{v+2}},s_{j_{v}}]}$.  
\begin{figure}
\begin{center}
\begin{small}
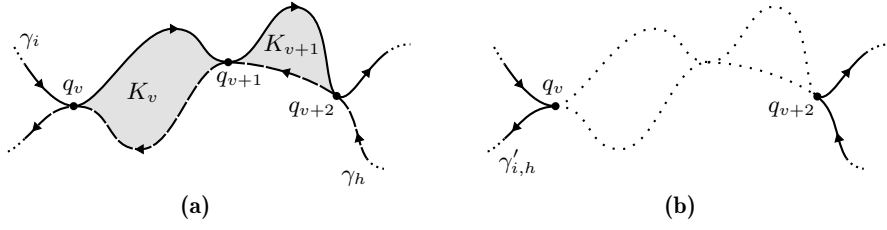 
\caption{\textbf{(a)} The connected components $\gamma_i$ and $\gamma_h$ of the original multi\-curve~$\ggamma$. \textbf{(b)}~The connected component $\gamma_{i,h}'$ of the new  multicurve $\ggamma'$.}
\label{f:consecutive_short_loops}
\end{small}
\end{center}
\end{figure}
This procedure replaces the connected components $\gamma_i$ and $\gamma_h$ with a single connected component $\gamma_{i,h}'$, see Figure~\ref{f:consecutive_short_loops}(b). 
The multicurve 
$\ggamma':= (\ggamma\setminus\{\gamma_i,\gamma_h\})\cup\gamma_{i,h}$ satisfies $\SSS_k(\ggamma')<\SSS_k(\gamma)$, and does not have more vertices than the original multicurve $\ggamma$. Moreover, since $K_v$ and $K_{v+1}$ do not contain entire connected components of $\ggamma$ and since each of the sets $\{(i,t_{j_v}),(h,s_{j_v})\}$, $\{(i,t_{j_{v+1}}),(h,s_{j_{v+1}})\}$, and $\{(i,t_{j_{v+2}}),(h,s_{j_{v+2}})\}$ is an adjacent tangency of $\ggamma$, the multicurve $\ggamma'$ belongs to $\overline{\ACMult(m-1)}$, and therefore to $\Mult_k(n)$. This contradicts~\eqref{e:global_minimum}.
\end{proof}

Finding an upper bound for the cardinalities $\#B_{i,i}$ is a more difficult task, which requires some preliminaries. Given an absolutely continuous loop $\zeta:\R/\sigma\Z\to M$ such that $\zeta(t_0)=\zeta(t_1)$ for some $t_0,t_1\in\R$ with $0< t_1-t_0<\sigma$, we say that the \textbf{scission} of $\zeta$ at $t_0,t_1$ is the operation that produces the loops 
$\zeta': \R/(t_1-t_0)\Z\to M$ and $\zeta'':\R/(\sigma+t_0-t_1)\Z \to M$ 
given by
\begin{align*}
\zeta'(t)&=\zeta(t_0+t),\qquad\forall  t\in[0,t_1-t_0],\\
\zeta''(t)&=\zeta(t_1+t),\qquad\forall t\in[0,\sigma+t_0-t_1],
\end{align*}
see Figure~\ref{f:scission}. Notice that
$\SSS_k(\zeta)=\SSS_k(\zeta')+\SSS_k(\zeta'')$. Moreover, if $\zeta$ is the connected component of a multicurve $\zzeta\in\Mult_k(n')$ and $t_0$ is an adjacent tangency of $\zeta$ in the multicurve $\zzeta$, then $(\zzeta\setminus\zeta)\cup\zeta'\cup\zeta''\in\Mult_k(n'+2)$.

We say that a periodic curve $\zeta_0:\R/\sigma\Z\to M$, with $\sigma>0$, is a \textbf{tree of small loops} when, for increasing values of the integer $j$ going from $1$ to some $h\in\N$, a scission of $\zeta_{j-1}$ produces the loops $\omega_{j}$ and $\zeta_{j}$ such that, if we further set $\omega_{h+1}=\zeta_{h}$, each $\omega_{j}$ is entirely contained in the Riemannian ball $B_g(\omega_j(0),\rho_\inj/2)$. Therefore, a sequence of scissions decomposes $\zeta_0$ in finitely many loops $\omega_1,\omega_2,...,\omega_{h},\omega_{h+1}$, and each of these loops is contained in a Riemannian ball of diameter $\rho_\inj$. By Corollary~\ref{c:small_loops}, we have
\begin{align*}
\SSS_k(\zeta_0)=\SSS_k(\omega_1)+\SSS_k(\omega_2) +...+\SSS_k(\omega_h)+\SSS_k(\omega_{h+1})>0.
\end{align*}

\begin{figure}
\begin{center}
\begin{small}
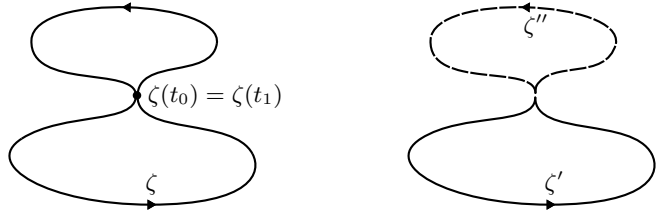 
\caption{The curve $\zeta$, and the result of its scission at $t_0,t_1$.}
\label{f:scission}
\end{small}
\end{center}
\end{figure}

Given a multicurve $\ggamma'=(\gamma_1',...,\gamma_{m'}')\in\Mult_k(n')$ with a connected component of the form $\gamma_i':\R/\tau_i'\Z\to M$, a finite subset $B'\subset M$, and a pair $t_0,t_1\in\R$ with $0<t_1-t_0<\tau_i'$, we say that the restriction $\gamma_i'|_{[t_0,t_1]}$ is \textbf{$(\ggamma',B')$-minimal} if  $(i,t_0),(i,t_1)$ is an adjacent tangency of $\ggamma'$ at some point of $B'$, and for all proper subintervals
$[t_0',t_1']\subsetneqq[t_0,t_1]$ we have that $(i,t_0'),(i,t_1')$ is not an adjacent tangency of $\ggamma'$ at some point of $B'$.

We say that an absolutely continuous loop $\zeta:\R/\sigma\Z\to M$ is \textbf{without small subloops} when, for all intervals $[a,b]\subset\R$ such that $\zeta(a)=\zeta(b)$, the loop $\zeta|_{[a,b]}$ is not contained in the Riemannian ball $B_g(\zeta(a),\rho_\inj/2)$. In particular, $\zeta|_{[a,b]}$ has length larger than or equal to $\rho_\inj$.

Now, consider a connected component $\gamma_i:\R/\tau_i\Z\to M$ of our multicurve $\ggamma$ satisfying~\eqref{e:global_minimum} and~\eqref{e:minimal_number_vertices}. By Lemma~\ref{l:bound_on_valence},
$\gamma_i$ is without small subloops. We set $\mu_0:=\gamma_i$, $\mmu_0:=\ggamma$, and we apply the following steps iteratively, starting from $j=0$, to the connected component $\mu_j$ without small subloops of the multicurve $\mmu_j$.
\vspace{10pt} 

\noindent\textbf{Step 1.} Assume that there exists a restriction $\mu_j|_{[t_0,t_1]}$ that is a $(\mmu_j,B_{i,i})$-minimal loop (if this is not possible, we set $\zeta_{j+1}:=\mu_j$ and we terminate the iterative procedure). Notice that $\mu_j|_{[t_0,t_1]}$ has length larger than or equal to $\rho_\inj$, since $\mu_j$ is without small subloops. We perform a scission of $\mu_j$ at $t_0,t_1$, thus obtaining two loops $\zeta_{j+1}$ and $\zeta_{j+1}'$, the first one corresponding to $\mu_j|_{[t_0,t_1]}$.  If $\mmu_j$ belongs to $\Mult_k(n_j)$, the multicurve $ \mmu_j':=(\mmu_j \setminus \mu_j )\cup \zeta_{j+1}\cup \zeta_{j+1}'$ belongs to $\Mult_k(n_j+2)$. Moreover
$\SSS_k(\mmu_j')=\SSS_k(\mmu_j)$.
The time $0$ is the vertex of the loop $\zeta_{j+1}'$ corresponding to the scission point 
$\zeta_{j+1}'(0)=\zeta_{j+1}(0)\in B_{i,i}$.
\hfill\qed
\vspace{10pt} 

\noindent\textbf{Step 2.} 
If no restriction of $\zeta_{j+1}'$ to some interval $[s_0,s_1]$ is a tree of small loops, we set $\mu_{j+1}:=\zeta_{j+1}'$ and $\mmu_{j+1}:=\mmu_{j}'$. Notice that $\mu_{j+1}$ is without small subloops. We now apply again Step~1 to the connected component $\mu_{j+1}$ of the multicurve $\mmu_{j+1}$.

If the whole $\zeta_{j+1}'$ is a tree of small loops, we set $\omega_{j+1}:=\zeta_{j+1}'$ and we terminate the iterative procedure.

Otherwise, we consider a maximal interval $[s_0,s_1]$ such that $\zeta_{j+1}'|_{[s_0,s_1]}$ is a tree of small loops. Here, maximal means that $[s_0,s_1]$ is not strictly contained in a larger interval $[r_0,r_1]$ with the same property. Since the loop $\mu_j$ was without small subloops, we have $0\in[s_0,s_1]$, and the connected component $\zeta_{j+1}'$ of $\mmu_j'\in\Mult_k(n_j+2)$ has an adjacent self-tangency at times $s_0,s_1$. We perform a scission of $\zeta_{j+1}'$ at times $s_0,s_1$, which produces a tree of small loops $\omega_{j+1}$ and a loop $\mu_{j+1}$ without small subloops. We set $\mmu_{j+1}:=(\mmu_{j}'\setminus\zeta_j')\cup\omega_{j+1}\cup\mu_{j+1}\in\Mult_k(n_j+4)$, and notice that $\SSS_k(\mmu_{j+1})=\SSS_k(\mmu_{j})$. We continue by applying Step~1 to the connected component $\mu_{j+1}$ of the multicurve $\mmu_{j+1}$.
\hfill\qed
\vspace{10pt}

The loops $\zeta_{j+1}$ produced in Step 1 have length larger than or equal to $\rho_\inj$. Therefore, the iterations of Steps 1-2 will eventually terminate after $a\in\N$ iterations. The outcome of this procedure is that we have replaced the connected component $\gamma_i$ of the original multicurve $\ggamma$ with the loops 
$\zeta_1, ..., \zeta_a, \omega_1, ..., \omega_b$, for some $0\leq  b\leq  a$. Since the original multicurve $\ggamma$ has length less than or equal to $\ell_{\max}$, we have $a\leq\ell_{\max}/\rho_\inj$. Moreover, the multicurve
\begin{align*}
\zzeta:= (\ggamma \setminus \gamma_i)\cup  \zeta_1\cup ...\cup  \zeta_a\cup  \omega_{1}\cup ...\cup  \omega_{b}
\end{align*}
belongs to $\Mult_k(n+2a+2b)\subset\Mult_k(n+4\ell_{\max}/\rho_\inj)$ and satisfies $\SSS_k(\zzeta)=\SSS_k(\ggamma)$. In order to simplify the notation, we write 
\begin{align*}
\zzeta= (\zeta_1, ..., \zeta_a, \zeta_{a+1}, ..., \zeta_{a+b}, \zeta_{a+b+1}, ..., \zeta_{a+b+n-1} ),
\end{align*}
where
\begin{align*}
(\zeta_{a+1}, ..., \zeta_{a+b}) &:= (\omega_1, ..., \omega_b ),\\
 (\zeta_{a+b+1},...,\zeta_{a+b+n-1} ) & := (\gamma_1, ..., \gamma_{i-1}, \gamma_{i+1}, ..., \gamma_{n} ).
\end{align*}
For each $h=1,...,a+b$, we consider a decomposition in segments of the connected component $\zeta_h:\R/\sigma_h\Z\to M$ that satisfies the conditions as in (S1-S2): we introduce the time-decomposition 
$0=\sigma_{h,0}\leq \sigma_{h,1} \leq ... \leq \sigma_{h,m_h}=\sigma_h$,
so that the restrictions $\zeta_h|_{[\sigma_{h,l},\sigma_{h,l+1}]}$ are the segments of $\zeta_h$. For all pairs of distinct segments 
$\zeta_h |_{[\sigma_{h,l},\sigma_{h,l+1}]}\asymp \zeta_x |_{[\sigma_{x,y},\sigma_{x,y+1}]}$
we fix a point $p_{(h,l),(x,y)}\in M$ at which the two segments have an adjacent tangency; if possible, we fix such a point so that it belongs to $B_{i,i}$. We define the finite set 
\begin{align*}
C_{h,x}:=
\Big\{
p_{(h,l),(x,y)}\
\Big|\
1\leq l\leq m_h,\, 1\leq y\leq m_x,\, \zeta_h|_{[\sigma_{h,l},\sigma_{h,l+1}]}\asymp\zeta_x|_{[\sigma_{x,y},\sigma_{x,y+1}]}
\Big\}
\end{align*}
Notice that
\begin{align*}
B_{i,i} \subseteq \bigcup_{1\leq h\leq x\leq a+b} \!\!\!\!\!\!\! C_{h,x}.
\end{align*}
Actually, since for all $j=1,...,a$ the loop $\zeta_j$ is $(\mmu_j,B_{i,i})$-minimal, we even have
\begin{align}\label{e:B_i_i_C_h_x}
B_{i,i}
\subseteq
\left(\bigcup_{h=1,...,b} \!\!\! C_{a+h,a+h}\right) 
\cup
\left(\bigcup_{1\leq h< x\leq a+b} \!\!\!\!\!\!\! C_{h,x}\right).
\end{align}
Let us find upper bounds for the cardinality of the sets $C_{a+h,a+h}$ and $C_{h,x}$ in this last expression.

\begin{lem}\label{l:bound_C_h_h}
For all $h=1,...,b$ we have
$\# C_{a+h,a+h}< (4\ell_{\max}/\rho_\inj)^2.$
\end{lem}

\begin{proof}
We recall that we have denoted by $m_{a+h}$ the number of segments of $\zeta_{a+h}$, and therefore $\# C_{a+h,a+h}\leq m_{a+h}^2$. Assume by contradiction that 
$\# C_{a+h,a+h}\geq  (4\ell_{\max}/\rho_\inj)^2$. In particular, $\zeta_{a+h}$ contains at least $4\ell_{\max}/\rho_\inj$ segments. Since $\zeta_{a+h}$ is a tree of short loops, a sequence of scissions decomposes it into finitely many loops $\eta_{1},...,\eta_{s}$, each one having length smaller than $\rho_\inj$. In particular, each $\eta_j$ is the topological boundary of a contractible compact set $K_j\subset B_g(\eta_j(0),\rho_\inj/2)$. We remove from the multicurve $\zzeta$ all the connected components that are entirely contained in one of the $K_j$'s, and we further remove the connected component $\zeta_{a+h}$. We denote the obtained multicurve by $\zzeta'$. Since we have removed at least $4\ell_{\max}/\rho_\inj$ segments, $\zzeta'$ belongs to $\Mult_k(n)$ and satisfies
\begin{align*}
\SSS_k(\zzeta')\leq \SSS_k(\zzeta)-\SSS_k(\zeta_{a+h})<\SSS_k(\zzeta)=\SSS_k(\ggamma),
\end{align*}
which contradicts~\eqref{e:global_minimum}.
\end{proof}

\begin{lem}\label{l:bound_C_h_x}
For all distinct $h,x\in\{1,...,a+b\}$, we have
\begin{align*}
\#C_{h,x}
\leq  
\big\lfloor 
1 + 4(\ell_{\max}/\rho_\inj)^2 + (\ell_{\max}/\rho_\inj) 
\big\rfloor!
=: c.
\end{align*}
\end{lem}

\begin{proof}
The proof is analogous to the one of Lemma~\ref{l:estimate_B_i_h}, but we provide full details for the reader's convenience. Let $0\leq t_0<...<t_{u-1}<\sigma_h$ and $s_0,...,s_{u-1}\in[0,\sigma_x)$ be such that:
\begin{itemize}
\item $C_{h,x}=\big \{\zeta_h(t_0),..., \zeta_h(t_{u-1})\big \}=\big \{\zeta_x(s_0),..., \zeta_x(s_{u-1})\big \}$;

\item $\zeta_h(t_l)\neq\zeta_h(t_y)$ if $l\neq y$;

\item $(h,t_j),(x,s_j)$ is an adjacent tangency of $\zzeta$ for each $j=0,...,u-1$.
\end{itemize}
Let $w\in\N$ be the integer such that
$w! \leq   \# C_{h,x}< (w+1)!$.
The combinatorial statement \cite[Proposition~1]{Taimanov:1992sm} implies that there exist $0\leq j_0<...<j_{w-1}\leq \# C_{h,x}$ such that either 
\begin{align}\label{e:preserved_order_Chx}
s_{j_0}<s_{j_1}<...<s_{j_{w-1}}
\end{align}
or
\begin{align}\label{e:reversed_order_Chx}
s_{j_0}>s_{j_1}>...>s_{j_{w-1}}.
\end{align}

Assume that~\eqref{e:preserved_order_Chx} holds. For each $v\in\Z/w\Z$ we consider the closed curve 
\begin{align*}
\xi_v:= \zeta_h |_{[t_{j_v},t_{j_{v+1}}]}  *  \overline{\zeta_x |_{[s_{j_v},s_{j_{v+1}}]}}.
\end{align*}
We claim that, if $w>2$, the loop $\xi_v$ is not contractible. Indeed, assume that $\xi_v$ is the boundary of a contractible compact subset $K\subset M$. Since $\zeta_h$ and $\zeta_x$ have adjacent tangencies at the points 
$q_v:= \zeta_h(t_{j_v})=\zeta_x(s_{j_v})$ and $q_{v+1}:= \zeta_h(t_{j_{v+1}})=\zeta_x(s_{j_{v+1}})$,
the point $\zeta_x(s_{j_{v+2}})$ is forced to lie inside $K$. But then $\zeta_x(s_{j_{v+2}})$ must coincide with $\zeta_x(s_{j_v})$, as otherwise $\zeta_h$ and $\zeta_x$ could not have an adjacent tangency at the point $\zeta_h(t_{j_{v+2}})=\zeta_x(s_{j_{v+2}})$. Therefore $v=v+2$ and $w=2$. By Corollary~\ref{c:small_loops}, each loop $\xi_v$ must have length larger than or equal to $\rho_\inj$. Since
\begin{align*}
\sum_{v\in\Z/w\Z}  \mathrm{length}(\xi_v)=\mathrm{length}(\gamma_h)+\mathrm{length}(\gamma_x)\leq \ell_{\max},
\end{align*}
we conclude
$w\leq \max \{ 2, \ell_{\max}/\rho_\inj\}=\ell_{\max}/\rho_\inj$, and therefore
\begin{align*}
\# C_{h,x}
\leq 
\big\lfloor 
1 + \ell_{\max}/\rho_\inj
\big\rfloor!
\end{align*}

Assume that~\eqref{e:reversed_order_Chx} holds instead. For each $v\in\Z/w\Z$ we consider the closed curve 
\begin{align*}
\xi_v:= \zeta_h |_{[t_{j_v},t_{j_{v+1}}]} *  \zeta_x |_{[s_{j_v},s_{j_{v+1}}]}.
\end{align*}
We denote by $w_{\mathrm{short}}$ the number of closed curves $\xi_v$, for $v\in\Z/w\Z$, whose length is smaller than $\rho_\inj$, and we set $w_{\mathrm{long}}:=w-w_{\mathrm{short}}$. Since the total length of the multicurve $\zzeta$ is at most $\ell_{\max}$, we have $w_{\mathrm{long}}\leq\ell_{\max}/\rho_\inj$. We claim that
\begin{align*}
\# C_{h,x} 
\leq  
\big\lfloor
1 + 4(\ell_{\max}/\rho_\inj )^2 + (\ell_{\max}/\rho_\inj)
\big\rfloor!
\end{align*}
Indeed, assume by contradiction that the above inequality does not hold. This implies 
\begin{align*}
w_{\mathrm{short}} 
> 4 \left (\frac{\ell_{\max}}{\rho_\inj}\right )^2 + \frac{\ell_{\max}}{\rho_\inj}-w_{\mathrm{long}} 
\geq 4 \left (\frac{\ell_{\max}}{\rho_\inj}\right )^2
\geq 4 \left (\frac{\ell_{\max}}{\rho_\inj}\right ) w_{\mathrm{long}}.
\end{align*}
Therefore, there exists $v\in\Z/w\Z$ such that the loops $\xi_v,\xi_{v+1},...,\xi_{z}$, for $z:=v-1+4\ell_{\max}/\rho_\inj$, are distinct and each one has length smaller than $\rho_\inj$. By Corollary~\ref{c:small_loops}, each of these loops $\xi_y$ is the boundary of a contractible compact subset $K_y\subset B_g(\xi_y(0),\rho_\inj/2)$ and satisfies $\SSS_k(\xi_y)>0$. Notice that, by the very definition of the $t_j$'s, each loop $\xi_{y}$ contains at least one vertex of $\zzeta$. We now modify the connected components $\zeta_h$ and $\zeta_x$ as follows: along $\zeta_h$, once we reach $\zeta_h(t_{j_v})$, we continue along $\zeta_x$; along $\zeta_x$, once we reach $\zeta_x(s_{j_{z}})$, we continue along $\zeta_h$; finally we throw away the portions $\zeta_h|_{[t_{j_{v}},t_{j_{z}}]}$ and $\zeta_x|_{[s_{j_{z}},s_{j_{v}}]}$. This procedure replaces $\zeta_h$ and $\zeta_x$ with a single connected component $\zeta_{h,x}'$. We define the multicurve
\begin{align*}
\zzeta':= \big (\zzeta \setminus (\zeta_h\cup\zeta_x)\big )\cup\zeta_{h,x}'.
\end{align*}
We remove from $\zzeta'$ all the connected components that are entirely contained in a compact set among $K_v,K_{v+1},...,K_z$, and denote the resulting multicurve by $\zzeta''$. Notice that $\zzeta''$ belongs to $\Mult_k(n)$ and satisfies
\begin{align*}
\SSS_k(\zzeta'')\leq \SSS_k(\zzeta)-\SSS_k(\xi_v) -...-\SSS_k(\xi_z)<\SSS_k(\zzeta)=\SSS_k(\ggamma),
\end{align*}
which contradicts~\eqref{e:global_minimum}.
\end{proof}

\begin{lem}\label{l:estimate_B_i_i} 
$\# B_{i,i}\leq 4(\ell_{\max}/\rho_\inj)^2c +  16(\ell_{\max}/\rho_\inj)^3$, where $c$ is the constant given by Lemma~\ref{l:bound_C_h_x}.
\end{lem}

\begin{proof}
By~\eqref{e:B_i_i_C_h_x},  Lemma~\ref{l:bound_C_h_h}, and Lemma~\ref{l:bound_C_h_x}, we have
\begin{align*}
\# B_{i,i}  & \leq \sum_{1\leq h< x\leq a+b} \!\!\!\!\!\!\! \#C_{h,x} + \sum_{h=1}^b\# C_{a+h,a+h}\\ 
& \leq 
 (a+b)^2c + (4\ell_{\max}/\rho_\inj)^2b\\
& \leq  4(\ell_{\max}/\rho_\inj)^2c + 16 (\ell_{\max}/\rho_\inj)^3. 
\qedhere
\end{align*}
\end{proof}

Summing up, Lemmas~\ref{l:bound_A_i_h}, \ref{l:A_i_h_B_i_h}, \ref{l:estimate_B_i_h}, and \ref{l:estimate_B_i_i} provide the following compactness result.

\begin{prop}\label{p:compactness}
There exists $n_{\min}\geq n_{\mathrm{neg}}$ such that, for all integers $n\geq n_{\min}$, a minimizer of $\SSS_k$ over the space $\Mult_k(n)$ that does not have collapsed connected components belongs to $\Mult_k(n_{\min})$.
\hfill\qed
\end{prop}

\begin{proof}[Proof of Theorem~\ref{t:local_minimizers}]
By Lemma~\ref{l:minima_are_negative}, up to replacing $M$ by a finite covering space and lifting the Lagrangian $L$ to the tangent bundle of such covering space, for all integers $n\geq n_{\mathrm{neg}}$ the action functional $\SSS_k:\Mult_k(n)\to\R$ attains negative values. Let $\ggamma\in\Mult_k(n_{\min})$ be a multicurve without collapsed connected components such that
\begin{align*}
\SSS_k(\ggamma)=\min_{\Mult_k(n_{\min})} \SSS_k<0.
\end{align*}
Proposition~\ref{p:compactness} implies that 
\begin{align}\label{e:global_minimizer_in_proof}
\SSS_k(\ggamma)=\min_{\Mult_k(n)} \SSS_k,\qquad\forall n\geq n_{\min}.
\end{align}
Therefore, Lemmas~\ref{l:components_are_periodic_orbits} and~\ref{l:components_are_embedded} imply that the connected components of $\ggamma$ are embedded periodic orbits of the Lagrangian system of $L$ with energy $k$. Since $\SSS_k(\ggamma)<0$, at least one connected component  of $\ggamma$, say $\gamma_1$, satisfies $\SSS_k(\gamma_1)<0$.

It only remains to show that the periodic orbit $\gamma_1:\R/\tau_1\Z\to M$ is a local minimizer of the free-period action functional $\SSS_k$ over the space of absolutely continuous periodic curves. We prove this by contradiction, assuming that there exists a sequence of absolutely continuous periodic curves $\gamma_{\alpha,1}:\R/\tau_{\alpha,1}\Z\to M$ such that 
\begin{align*}
&\lim_{\alpha\to\infty} \tau_{\alpha,1} = \tau_1,\\
&\lim_{\alpha\to\infty} \max \big\{ \dist (\gamma_{\alpha,1}(t/\tau_{\alpha,1}),\gamma_{1}(t/\tau_{1}) )\  \big|\ t\in[0,1] \big\}=0,\\
&\SSS_k(\gamma_{\alpha,1})<\SSS_k(\gamma_1), \quad \forall\alpha\in\N.
\end{align*}
We fix an integer $h\in\N$ large enough so that, for all $\alpha\in\N$ sufficiently large and for all $t_0,t_1\in\R$ with $t_1-t_0<1/h$, we have $\dist(\gamma_{\alpha,1}(t_0),\gamma_{\alpha,1}(t_1))<\rho_\inj$. For all $j=0,...,h-1$, we set 
$t_{\alpha,j}:=\tau_{\alpha,1} j/h$, 
$q_{\alpha,j}:=\gamma_{\alpha,1}(t_{\alpha,j})$, 
$t_j :=\tau_1 j/h$, and 
$q_j :=\gamma_1(t_j)$.
For each $j\in\Z/h\Z$, we remove from $\gamma_{\alpha,1}$ each portion $\gamma_{\alpha,1}|_{[t_{\alpha,j},t_{\alpha,j+1}]}$, and we glue in the unique free-time local minimizer with energy $k$ joining $q_{\alpha,j}$ and $q_{\alpha,j+1}$. We still denote the resulting curve by $\gamma_{\alpha,1}:\R/\tau_{\alpha,1}\Z\to M$. 
Since we replaced portions of the original curve with unique free-time local minimizers with energy $k$, we still have 
$\SSS_k(\gamma_{\alpha,1})<\SSS_k(\gamma_{1})$ for all $\alpha\in\N$.
Moreover, since $\gamma_{\alpha,1}$ is now a piecewise solution of the Euler-Lagrange equation of $L$, and since its vertices satisfy $q_{\alpha,j}\to q_j$ as $\alpha\to\infty$, for all $t\in[0,1]$ we have 
$\gamma_{\alpha,1}(\tau_{\alpha,1}t)\to  \gamma_{1}(\tau_{1}t)$ as $\alpha\to\infty$, 
and for all but finitely many $t\in[0,1]$ we have
$\dot\gamma_{\alpha,1}(\tau_{\alpha,1}t)\to  \dot\gamma_{1}(\tau_{1}t)$ as $\alpha\to\infty$. 
Since $\gamma_1$ is a smoothly embedded curve, and since the property of being smoothly embedded is open in the $C^1$-topology, for all $\alpha$ large enough the curve $\gamma_{\alpha,1}$ is topologically embedded. Therefore, the multicurve $\ggamma_\alpha$ obtained by replacing $\gamma_1$ with $\gamma_{\alpha,1}$ in $\ggamma$, belongs to $\Mult_k(n)$ for some $n\in\N$, and satisfies 
\[\SSS_k(\ggamma_\alpha)=\SSS_k(\ggamma)-\SSS_k(\gamma_1)+\SSS_k(\gamma_{\alpha,1})<\SSS_k(\ggamma).\] 
This contradicts~\eqref{e:global_minimizer_in_proof}.
\end{proof}

\section{Free-period discretizations}\label{s:discretizations}

\subsection{Unique action minimizers}

We work under the assumptions of Section~\ref{s:unique_free_time_minimizers}: 
we consider a closed manifold $M$ of arbitrary dimension, and a Tonelli Lagrangian $L:\Tan M\to\R$. We fix an energy value $k\in\R$. Since we are only interested in the Euler-Lagrange dynamics on  $E^{-1}(k)$, up to modifying $L$ far from $E^{-1}(-\infty,k]$ we can always assume that $L$ is quadratic at infinity, that is, 
$L(q,v)=\tfrac12 g_q(v,v)$
for all $(q,v)\in\Tan M$ outside a compact set of $\Tan M$. Here, $g$ is some fixed Riemannian metric on $M$. We will denote by $\dist:M\times M\to[0,\infty)$ the Riemannian distance associated to $g$.

We consider the free-period action functional $\SSS_k$ defined on the space of $W^{1,2}$ periodic curves on $M$ of any period. Formally, such a functional has the form 
\begin{gather*}
\SSS_k: W^{1,2}(\R/\Z,M)\times(0,\infty)\to\R,\\
\SSS_{k}(\Gamma,\tau)=\tau\int_0^1 \Big [L(\Gamma(s),\Gamma'(s)/\tau)+ k\Big ]\, \diff s 
=\int_0^{\tau} L(\gamma(t),\dot\gamma(t))\,\diff t +  \tau k,
\end{gather*}
where we have identified $(\Gamma,\tau)$ with the $\tau$-periodic curve $\gamma(t)=\Gamma(t/\tau)$. In this setting, $\SSS_{k}$ is $C^{1,1}$, but in general not $C^2$ unless the restriction of the  Lagrangian $L$ to any fiber of $\Tan M$ is a polinomial of degree two (see~\cite[Proposition~3.2]{Abbondandolo:2009gg}). The aim of this section is to overcome this lack of regularity by restricting $\SSS_{k}$ to a suitable finite dimensional submanifold, in the spirit of Morse's broken geodesics approximation of path spaces \cite[Section~16]{Milnor:1963rf}. In the simpler setting of the fixed-period action functional, this type of analysis has been carried over by the second author in \cite[Section~3]{Mazzucchelli:2011cd} and \cite[Chapter~4]{Mazzucchelli:2012nm}.

In Section~\ref{s:unique_free_time_minimizers}, we investigated the properties of curves that are unique free-time local minimizers. In this section, we will need the following slightly different notion. 
An absolutely continuous curve $\gamma:[0,\tau]\to M$ is called a \textbf{unique action minimizer} when, for every other absolutely continuous curve $\zeta:[0,\tau]\to M$ with $\zeta(0)=\gamma(0)$ and $\zeta(\tau)=\gamma(\tau)$, we have
\begin{align*}
\int_0^\tau L(\gamma(t),\dot\gamma(t))\,\diff t < \int_0^\tau L(\zeta(t),\dot\zeta(t))\,\diff t.
\end{align*}
It is well known that every unique action minimizer $\gamma$ is a smooth solution of the Euler-Lagrange equation~\eqref{e:Euler_Lagrange}. A \textbf{Jacobi vector field} along such a curve $\gamma$ is a solution $\theta:[0,\tau]\to\gamma^*(\Tan M)$ of the linearized Euler-Lagrange equation $\Jac_\gamma\theta=0$, where $\Jac_\gamma$ is the linear second-order differential operator
\begin{align*}
\Jac_\gamma\theta:= \tfrac{\diff}{\diff t} \Big[ L_{qv}(\gamma,\dot\gamma)\,\theta +L_{vv}(\gamma,\dot\gamma)\,\dot\theta \Big]
-L_{qq}(\gamma,\dot\gamma)\,\theta-L_{vq}(\gamma,\dot\gamma)\,\dot\theta.
\end{align*}
A unique action minimizer $\gamma:[0,\tau]\to M$ is said to be \textbf{non-degenerate} when there is no non-zero Jacobi vector field $\theta:[0,\tau]\to\gamma^*(\Tan M)$ vanishing at both endpoints $\theta(0)$ and $\theta(\tau)$. 

As before, we denote by $\phi_L^t:\Tan M\to\Tan M$ the Euler-Lagrange flow of $L$, and by $\pi:\Tan M\to M$ the projection map of the tangent bundle. For $\rho>0$, we define the open neighborhood of the diagonal in $M\times M$
\begin{align*}
\Delta_\rho:= \big\{(q_0,q_1)\in M\times M\ \big|\ \dist(q_0,q_1)<\rho\big\}. 
\end{align*}
The following proposition follows from the proofs of \cite[Theorems~4.1.1 and 4.1.2]{Mazzucchelli:2012nm}.

\begin{prop}\label{p:unique_action_minimizers}
There exist $\rho=\rho(L)>0$, $\epsilon=\epsilon(L)>0$, and smooth maps 
\[\nu^\pm:\Delta_\rho\times(0,\epsilon)\to  \Tan M\]
such that, for all $(q_0,q_1,\tau)\in\Delta_\rho\times(0,\epsilon)$, we have 
\begin{align*}
\nu^-(q_0,q_1,\tau) & \in\Tan_{q_0}M,\\
\nu^+(q_0,q_1,\tau) & \in\Tan_{q_1}M,\\
\phi_L^\tau\big (q_0,\nu^-(q_0,q_1,\tau)\big )&=\big (q_1,\nu^+(q_0,q_1,\tau)\big ).
\end{align*}
The curve $\gamma_{q_0,q_1,\tau}:[0,\tau]\to M$ given by
\begin{align*}
\gamma_{q_0,q_1,\tau}(t):= \pi\circ\phi_L^t\big (q_0,\nu^-(q_0,q_1,\tau)\big )=\pi\circ\phi_L^{t-\tau}\big (q_1,\nu^+(q_0,q_1,\tau)\big )
\end{align*}
is a non-degenerate unique action minimizer.
\hfill\qed
\end{prop}

We need to consider the space of smooth paths defined on any compact interval of the form $[0,\tau]$. We can see this space as the product $C^\infty([0,1];M)\times(0,\infty)$ by identifying any pair $(\Gamma,\tau)\in C^\infty([0,1];M)\times(0,\infty)$ with the smooth curve $\gamma:[0,\tau]\to M$ given by $\gamma(t)=\Gamma(t/\tau)$. We define the smooth map 
\begin{gather*}
\iota:\Delta_\rho\times(0,\epsilon)\to C^\infty([0,1];M)\times(0,\epsilon),\\ 
\iota(q_0,q_1,\tau):=\gamma_{q_0,q_1,\tau}=(\Gamma_{q_0,q_1,\tau},\tau). 
\end{gather*}
Proposition~\ref{p:unique_action_minimizers} implies that this map is an injective immersion. Indeed, for all $(v_0,v_1)\in\Tan_{q_0}M\times\Tan_{q_1}M$, consider 
\begin{align*}
(\Theta_{v_0,v_1},0):= \diff\iota (q_0,q_1,\tau)\big [(v_0,v_1,0)\big ].
\end{align*}
The curve $\theta_{v_0,v_1}(t):=\Theta_{v_0,v_1}(t/\tau)$  is a Jacobi vector field along $\gamma_{q_0,q_1,\tau}$ satisfying
$\theta_{v_0,v_1}(0)=v_0$ and $\theta_{v_0,v_1}(\tau)=v_1$.  
Notice that such a Jacobi vector field is unique: if $\theta'$ were another Jacobi vector field along $\gamma_{q_0,q_1,\tau}$ with $\theta'(0)=v_0$ and $\theta'(\tau)=v_1$, the difference $\theta_{v_0,v_1}-\theta'$ would be a Jacobi vector field that vanishes at the endpoints, and therefore vanishes identically due to the non-degeneracy of $\gamma_{q_0,q_1,\tau}$.

Consider now the Jacobi vector field $\partial_\tau\gamma_{q_0,q_1,\tau}(t)$ along $\gamma_{q_0,q_1,\tau}$, which satisfies 
\begin{align*}
\partial_\tau\gamma_{q_0,q_1,\tau}(0) & =0,\\
\partial_\tau\gamma_{q_0,q_1,\tau}(\tau) & =-\dot\gamma_{q_0,q_1,\tau}(\tau)=-\nu^+(q_0,q_1,\tau).
\end{align*}
Such vector field enters the expression of the derivative of the map $\iota$ with respect to $\tau$. Indeed, consider
\begin{align*}
 (\Psi_{q_0,q_1,\tau},1 ):= \diff\iota (q_0,q_1,\tau)\big [(0,0,1)\big ].
\end{align*}
The vector field $\Psi_{q_0,q_1,\tau}$ along $\Gamma_{q_0,q_1,\tau}$ is given by
$$\Psi_{q_0,q_1,\tau}(t)= \partial_\tau\Gamma_{q_0,q_1,\tau}(t)= (\partial_\tau\gamma_{q_0,q_1,\tau} )(\tau t) + \dot\gamma_{q_0,q_1,\tau}(\tau t)\,t.$$
Hence, the associated vector field $\psi_{q_0,q_1,\tau}$ along $\gamma_{q_0,q_1,\tau}$ is given by
\begin{equation}
\psi_{q_0,q_1,\tau}(t)= \Psi_{q_0,q_1,\tau}(t/\tau)=\partial_\tau\gamma_{q_0,q_1,\tau}(t) +\dot\gamma_{q_0,q_1,\tau}(t)\,t/\tau
\label{psiq0q1tau}
\end{equation}
and satisfies \[\psi_{q_0,q_1,\tau}(0)=\psi_{q_0,q_1,\tau}(\tau)=0.\]  Notice that $\psi_{q_0,q_1,\tau}$ is not a Jacobi vector field. Indeed, it satisfies
\begin{equation}\label{e:Jacobi_psi}
\begin{split}
\Jac_{\gamma_{q_0,q_1,\tau}}(\psi_{q_0,q_1,\tau})  & = 
\Jac_{\gamma_{q_0,q_1,\tau}} (\dot\gamma_{q_0,q_1,\tau} \,t/\tau )\\
& = \big( \tfrac{\diff}{\diff t} E_v(\gamma_{q_0,q_1,\tau},\dot\gamma_{q_0,q_1,\tau}) - E_q(\gamma_{q_0,q_1,\tau},\dot\gamma_{q_0,q_1,\tau}) \big)/\tau.
\end{split} 
\end{equation}

\vspace{3mm}

\subsection{The discrete free-period action functional}
For $h\in\N$, which will be chosen large enough later, consider the open neighborhood of the $h$-fold diagonal 
\begin{align*}
\Delta_{h,\rho}:= \big\{\qq=(q_0,q_1,...,q_{h-1})\in M^{\times h}\ \big|\ \dist(q_{i},q_{i+1})<\rho,\ \forall  i\in \Z/h\Z\big\}.
\end{align*}
We equip $\Delta_{h,\rho}\times(0,\epsilon)$ with the Riemannian metric $\llangle\cdot,\cdot\rrangle_\cdot$ given by
\begin{align}\label{e:discrete_Riemannian_metric}
\llangle (\vv,\sigma),(\ww,\mu)\rrangle_{(\qq,\tau)} =  h\sigma\mu  +\!\!\! \sum_{i\in\Z/h\Z }  \langle v_i,w_i\rangle_{q_i},
\end{align}
where  $\langle\cdot,\cdot\rangle_\cdot$ denotes the Riemannian metric $g$ on the manifold $M$. We define the injective smooth map 
\begin{equation}
\label{e:iota_h}
\begin{gathered}
\iota_h:\Delta_{h,\rho}\times(0,\epsilon) \to W^{1,2}(\T,M)\times(0,\infty),\\ 
\iota_h(\qq,\tau):=\gamma_{\qq,\tau}=(\Gamma_{\qq,\tau},\tau h),
\end{gathered}
\end{equation}
where 
$$\gamma_{\qq,\tau}(i\tau + t):= \gamma_{q_i,q_{i+1},\tau}(t),\qquad \forall i\in \Z/h\Z,\ \forall t\in[0,\tau].$$ 
For later purposes, let us compute the derivative of $\iota_h$. For all $(\vv,\sigma)\in \Tan_{\qq}\Delta_{h,\rho}\times\R$, we have
\begin{equation}
\diff\iota_h(\qq,\tau)\big [(\vv,\sigma)\big ] = \big (\Theta_{\vv}+\sigma\,\Psi_{\qq,\tau},h\sigma\big ),
\label{diffiotah}
\end{equation}
where $\Theta_{\vv}$ and $\Psi_{\qq,\tau}$ are vector fields along $\Gamma_{\qq,\tau}$ defined as follows: if $\theta_{\vv}(t)=\Theta_{\vv}(t/(\tau h))$ and $\psi_{\qq,\tau}(t)=\Psi_{\qq,\tau}(t/(\tau h))$, then
\begin{align*}
\theta_{\vv}(i\tau+t) & =\theta_{v_i,v_{i+1}}(t),\\ 
\psi_{\qq,\tau}(i\tau+t) & =\psi_{q_i,q_{i+1},\tau}(t).
\end{align*}
Notice that $\theta_{\vv}$ is the unique continuous vector field along $\gamma_{\qq,\tau}$ whose restriction to any interval of the form $[i\tau,(i+1)\tau]$ is the Jacobi vector field with boundary conditions $\theta_{\vv}(i\tau)=v_i$ and $\theta_{\vv}((i+1)\tau)=v_{i+1}$. This implies that $\iota_h$ is an immersion. We define the \textbf{discrete free-period action functional} as the composition
\begin{equation}
S_k:= \SSS_k\circ \iota_h : \Delta_{h,\rho}\times(0,\epsilon)\to  \R.
\label{discreteactionfunctional}
\end{equation}
More explicitely, $S_k$ is given by
\begin{align*}
S_k(\qq,\tau)  & = \tau h k  + \int_0^{\tau h}  L(\gamma_{\qq,\tau}(t),\dot\gamma_{\qq,\tau}(t))\,\diff t\\
& = \tau h k  + \!\!\sum_{i\in\Z/h\Z }\int_0^{\tau} L(\gamma_{q_i,q_{i+1},\tau}(t),\dot\gamma_{q_i,q_{i+1},\tau}(t))\,\diff t\\
 & = \tau h k  + \!\!\sum_{i\in\Z/h\Z }\int_0^{\tau} L\circ\phi_L^t(q_i,\nu^-(q_i,q_{i+1},\tau))\,\diff t.
\end{align*}
The last expression readily implies that $S_k$ is a $C^\infty$ function.

\subsection{The discrete least action principle}
The differential of the free-period action is given by
\begin{equation}
\label{e:differential_infinite_dim_S_k}
\begin{split}
\diff\SSS_{k}(\Gamma,\tau)\big[(\Theta,\sigma)\big]
 = &
\int_0^{\tau} \Big[ L_q(\gamma(t),\dot\gamma(t))\,\theta(t)+L_v(\gamma(t),\dot\gamma(t))\,\dot\theta(t) \Big ]\,\diff t\\
& + \frac{\sigma}{\tau} \int_0^{\tau} \Big[ k-E(\gamma(t),\dot\gamma(t)) \Big ]\,\diff t, 
\end{split}
\end{equation}
where $\theta(t):=\Theta(t/\tau)$. This, together with an integration by parts and a bootstrapping argument, readily implies that the critical points of $\SSS_k$ are precisely the periodic orbits of the Lagrangian system of $L$ with energy $k$, see, e.g., \cite[Section~3.2]{Abbondandolo:2013is}. Any given critical point of $\SSS_k$ is contained in the image of the map $\iota_h$ for a large enough integer $h$. In particular, its preimage under $\iota_h$ is a critical point of $S_k$. Let us now verify that all critical points of $S_k$ are mapped under $\iota_h$ to critical points of $\SSS_k$. By~\eqref{diffiotah} and~\eqref{e:differential_infinite_dim_S_k}, we compute
\begin{align*}
\diff S_k(\qq,\tau)\big[(\vv,\sigma)\big]
 = \, &
\diff \SSS_k(\Gamma_{\qq,\tau},\tau h)\big[(\Theta_{\vv}+\sigma\,\Psi_{\qq,\tau},h\sigma)\big]\\
 = \, &
\int_0^{\tau h}
\Big[
L_q(\gamma_{\qq,\tau},\dot\gamma_{\qq,\tau})\theta_{\vv} + L_v(\gamma_{\qq,\tau},\dot\gamma_{\qq,\tau})\dot\theta_{\vv}
\Big]\,\diff t\\
\, & + \sigma
\int_0^{\tau h}
\Big[
L_q(\gamma_{\qq,\tau},\dot\gamma_{\qq,\tau})\psi_{\qq,\tau} + L_v(\gamma_{\qq,\tau},\dot\gamma_{\qq,\tau})\dot\psi_{\qq,\tau}
\Big]\,\diff t\\
\, & + \frac{\sigma}{\tau}
\int_0^{\tau h}
\Big[
k - E(\gamma_{\qq,\tau},\dot\gamma_{\qq,\tau})\Big]\,\diff t.
\end{align*}

We break the integrals in the three summands as the sum of integrals over intervals of the form $[i\tau,(i+1)\tau]$. The first summand becomes 
\begin{align*}
&\sum_{i\in\Z/h\Z }
\int_{0}^{\tau}
\Big[
L_q(\gamma_{q_i,q_{i+1},\tau},\dot\gamma_{q_i,q_{i+1},\tau})\theta_{v_i,v_{i+1}} + L_v(\gamma_{q_i,q_{i+1},\tau},\dot\gamma_{q_i,q_{i+1},\tau})\dot\theta_{v_i,v_{i+1}}
\Big]\,\diff t\\
&\qquad=
\sum_{i\in\Z/h\Z } L_v(\gamma_{q_i,q_{i+1},\tau}(t),\dot\gamma_{q_i,q_{i+1},\tau}(t))\theta_{v_i,v_{i+1}}(t) \Big|_{t=0}^{t=\tau}\\
&\qquad=
\sum_{i\in\Z/h\Z } \Big[L_v(q_i,\nu^+(q_{i-1},q_i,\tau)) - L_v(q_i,\nu^-(q_i,q_{i+1},\tau)) \Big]v_i.
\end{align*}
The second summand vanishes. Indeed, since $\psi_{q_i,q_{i+1},\tau}(0)=\psi_{q_i,q_{i+1},\tau}(\tau)=0$, we have
\begin{align*}
& 
\int_{0}^{\tau}
\Big[
L_q(\gamma_{q_i,q_{i+1},\tau},\dot\gamma_{q_i,q_{i+1},\tau})\psi_{q_i,q_{i+1},\tau} + L_v(\gamma_{q_i,q_{i+1},\tau},\dot\gamma_{q_i,q_{i+1},\tau})\dot\psi_{q_i,q_{i+1},\tau}
\Big]\,\diff t\\
&\qquad=
\sum_{i\in\Z/h\Z } L_v(\gamma_{q_i,q_{i+1},\tau}(t),\dot\gamma_{q_i,q_{i+1},\tau}(t))\,\psi_{q_i,q_{i+1},\tau}(t) \Big|_{t=0}^{t=\tau}\\
&\qquad =0.
\end{align*}
Since $E(\gamma_{\qq,\tau},\dot\gamma_{\qq,\tau})$ is constant on each interval of the form $[i\tau,(i+1)\tau]$, the third summand becomes
\begin{align*}
\sigma \sum_{i\in\Z/h\Z } \Big[k - E(\gamma_{q_i,q_{i+1},\tau},\dot\gamma_{q_i,q_{i+1},\tau})\Big]
=
\sigma \sum_{i\in\Z/h\Z } \Big[k - E(q_i,\nu^-(q_i,q_{i+1},\tau)) \Big].
\end{align*}
Summing up, we have obtained
\begin{equation}\label{e:differential_S_k}
\begin{split}
\diff S_k(\qq,\tau)\big[(\vv,\sigma)\big]
 = & 
\sum_{i\in\Z/h\Z }
\Big[
L_v(q_i,\nu^+(q_{i-1},q_i,\tau)) - L_v(q_i,\nu^-(q_i,q_{i+1},\tau))
\Big]
v_i\\
& + \sigma \sum_{i\in\Z/h\Z }\Big[k - E(q_{i},\nu^-(q_i,q_{i+1},\tau)) \Big].
\end{split} 
\end{equation}
In particular, the critical points of $S_k$ are precisely those points $(\qq,\tau)$ such that $\nu^+(q_{i-1},q_i,\tau)=\nu^-(q_i,q_{i+1},\tau)$ and $E(q_{i+1},\nu^+(q_i,q_{i+1},\tau))=k$ for all $i\in\Z/h\Z $, that is, such that $\gamma_{\qq,\tau}$ is a $\tau h$-periodic orbit with energy $k$.

\subsection{The Hessian of the free-period action functional} We denote by $H_{\qq,\tau}$ the Hessian of the free-period action functional $S_k$ at a critical point $(\qq,\tau)$. By differentiating equation~\eqref{e:differential_S_k} we find that the expression of $H_{\qq,\tau}$ is 
\begin{align*}
&H_{\qq,\tau} ((\vv,\sigma),(\ww,\mu) )\\
&\qquad= 
\sum_{i\in\Z/h\Z }
 \langle L_{vv}\,v_i,\diff\nu^+(q_{i-1},q_i,\tau)(w_{i-1},w_i,0)-\diff\nu^-(q_{i},q_{i+1},\tau)(w_i,w_{i+1},0) \rangle\\
&\qquad\quad 
+\sigma\sum_{i\in\Z/h\Z }  \langle L_{vv}\,w_i,\partial_\tau\nu^+(q_{i-1},q_i,\tau)-\partial_\tau\nu^-(q_{i},q_{i+1},\tau) \rangle\\
&\qquad\quad 
+\mu\sum_{i\in\Z/h\Z }  \langle L_{vv}\,v_i,\partial_\tau\nu^+(q_{i-1},q_i,\tau)-\partial_\tau\nu^-(q_{i},q_{i+1},\tau) \rangle\\
&\qquad\quad 
-\sigma\mu \sum_{i\in\Z/h\Z } 
\partial_\tau \big(E(q_{i},\nu^-(q_{i},q_{i+1},\tau))\big)
.
\end{align*}
We denote by $h_{\qq,\tau}$ the restriction of the Hessian $H_{\qq,\tau}$ to the codimension-one vector subspace $\Tan_{\qq}\Delta_{h,\rho}\times\{0\}$, which reads
\begin{equation*}
\begin{split}
h_{\qq,\tau}(\vv,\ww)
 :=&\ H_{\qq,\tau}\big ((\vv,0),(\ww,0)\big ) \\
 =& \sum_{i\in\Z/h\Z }
 \langle L_{vv}\,v_i,\diff\nu^+(q_{i-1},q_i,\tau)(w_{i-1},w_i,0)\\
&\qquad\quad -\diff\nu^-(q_{i},q_{i+1},\tau)(w_i,w_{i+1},0) \rangle\\
 =&\sum_{i\in\Z/h\Z }
 \langle L_{vv}\,v_i, \dot\theta_{\ww}(i\tau^-)-\dot\theta_{\ww}(i\tau^+) \rangle.
\end{split} 
\end{equation*}
We recall that the \textbf{nullity} $\nul(h)$ of a symmetric bilinear form $h$ is the dimension of its kernel, whereas the \textbf{index} $\ind(h)$ is the maximal dimension of a vector subspace of its domain over which $h$ is negative definite. In the usual terminology of Morse theory, $\ind(H_{\qq,\tau})$ and $\nul(H_{\qq,\tau})$ are the Morse index and the nullity of the functional $S_k$ at the critical point $(\qq,\tau)$.

A vector $\vv\in\Tan_{\qq}\Delta_{h,\rho}$ belongs to $\ker(h_{\qq,\tau})$ if and only if $\dot\theta_{\vv}(i\tau^-)=\dot\theta_{\vv}(i\tau^+)$ for all $i\in\Z/h\Z $, that is, if and only if $\theta_{\vv}$ is a smooth $h\tau$-periodic Jacobi vector field, thus satisfying
\begin{align*}
\diff\phi_L^t(\gamma_{\qq,\tau}(0),\dot\gamma_{\qq,\tau}(0)) \big [(\theta_{\vv}(0),\dot\theta_{\vv}(0))\big ] & =(\theta_{\vv}(t),\dot\theta_{\vv}(t)),\qquad\forall t\in\R,\\
(\theta_{\vv}(0),\dot\theta_{\vv}(0)) & =(\theta_{\vv}(h\tau),\dot\theta_{\vv}(h\tau)).
\end{align*}
This shows that the linear map $\vv\mapsto(\theta_{\vv}(0),\dot\theta_{\vv}(0))$ gives an isomorphism between $\ker(h_{\qq,\tau})$ and the eigenspace of the linear map $P:=\diff\phi_L^{h\tau}(\gamma_{\qq,\tau}(0),\dot\gamma_{\qq,\tau}(0))$ corresponding to the eigenvalue 1. In particular, 
\begin{align}\label{e:nul_proof_1}
\nul(h_{\qq,\tau})=\dim\ker(P-I).
\end{align}

A vector $(\vv,\sigma)$ belongs to $\ker(H_{\qq,\tau})$ if and only if
\begin{align}\label{e:ker_H_1}
\diff\nu^+(q_{i-1},q_i,\tau)\big [(v_{i-1},v_i,\sigma)\big ]=\diff\nu^-(q_{i},q_{i+1},\tau)\big [(v_i,v_{i+1},\sigma)\big ]
\end{align}
and
\begin{equation}
\begin{split}\label{e:ker_H_2}
&\sum_{i\in\Z/h\Z }\big  \langle L_{vv}\,v_i,\partial_\tau\nu^+(q_{i-1},q_i,\tau)-\partial_\tau\nu^-(q_{i},q_{i+1},\tau)\big \rangle\\
&\qquad\qquad\qquad = \sigma \sum_{i\in\Z/h\Z } \partial_\tau \big(E(q_{i},\nu^-(q_{i},q_{i+1},\tau))\big).
\end{split} 
\end{equation}
Conditions \eqref{e:ker_H_1} and \eqref{e:ker_H_2} can be conveniently rephrased by employing the vector fields $\theta_{\vv}$ and $\psi_{\qq,\tau}$.
\begin{lem}\label{l:kernel_H}
The kernel $\ker(H_{\qq,\tau})$ is isomorphic to the vector space of pairs $(\xi,\sigma)$, where $\xi$ is a smooth $h\tau$-periodic vector field along $\gamma_{\qq,\tau}$ such that
\begin{align}\label{e:Jacobi_xi}
\Jac_{\gamma_{\qq,\tau}}\xi= \big( \tfrac{\diff}{\diff t} E_v(\gamma_{\qq,\tau},\dot\gamma_{\qq,\tau}) -E_q(\gamma_{\qq,\tau},\dot\gamma_{\qq,\tau}) \big)\sigma/\tau
,
\end{align}
and $\sigma\in\R$ satisfies
\begin{align}\label{e:sigma}
\sigma
\int_0^{h\tau} \langle L_{vv}(\gamma_{\qq,\tau},\dot\gamma_{\qq,\tau}) \dot\gamma_{\qq,\tau},\dot\gamma_{\qq,\tau}\rangle\,\diff t
= 
\tau 
\int_0^{h\tau} \diff E(\gamma_{\qq,\tau},\dot\gamma_{\qq,\tau})\big[(\xi,\dot\xi)\big]\,\diff t.
\end{align}
The isomorphism is given by 
$(\vv,\sigma)\mapsto (\xi:=\theta_{\vv}+\sigma\psi_{\qq,\tau},\sigma)$.
\end{lem}

\begin{rem}
This lemma, together with the results in \cite[Appendix~A]{Abbondandolo:2015lt}, shows that the map $\diff\iota_h(\qq,\tau)$ restricts to an isomorphism from $\ker(H_{\qq,\tau})$ to the kernel of the Hessian of the free-period action functional $\SSS_k$ at $(\gamma_{\qq,\tau},h\tau)$. Notice that there is a unique $\sigma$ satisfying~\eqref{e:sigma} unless $\gamma_{\qq,\tau}$ is a stationary curve.
\hfill\qed
\end{rem}

\begin{proof}[Proof of Lemma~\ref{l:kernel_H}]
By the definition of the map $\nu^+$, we have
\begin{align*}
\diff\nu^+(q_{i-1},q_i,\tau)(v_{i-1},v_i,\sigma)
&=
\dot\theta_{\vv}(i\tau^-) + \sigma\,\big( \partial_\tau \dot\gamma_{\qq,\tau}(i\tau^-) + \ddot\gamma_{\qq,\tau}(i\tau^-)\big)\\
&=
\dot\theta_{\vv}(i\tau^-) + \sigma\, \big( \dot\psi_{\qq,\tau}(i\tau^-) - \dot\gamma_{\qq,\tau}(i\tau)/\tau \big),
\end{align*}
and analogously
\begin{align*}
\diff\nu^-(q_{i},q_{i+1},\tau)(v_{i},v_{i+1},\sigma)
&=
\dot\theta_{\vv}(i\tau^+) + \sigma \partial_\tau \dot\gamma_{\qq,\tau}(i\tau^+)\\
&=
\dot\theta_{\vv}(i\tau^+) + \sigma\,\big( \dot\psi_{\qq,\tau}(i\tau^+) - \dot\gamma_{\qq,\tau}(i\tau)/\tau  \big).
\end{align*}
We employ these two equations to rewrite~\eqref{e:ker_H_1} as
\begin{align*}
\dot\theta_{\vv}(i\tau^+) + \sigma \dot\psi_{\qq,\tau}(i\tau^+)=\dot\theta_{\vv}(i\tau^-) + \sigma \dot\psi_{\qq,\tau}(i\tau^-).
\end{align*}
Namely, for each $(\vv,\sigma)\in\ker(H_{\qq,\tau})$, the $h\tau$-periodic vector field $\xi:=\theta_{\vv} + \sigma\psi_{\qq,\tau}$ is $C^1$ and, by~\eqref{e:Jacobi_psi}, satisfies
\begin{align*}
\Jac_{\gamma_{\qq,\tau}} (\xi)=
\underbrace{\Jac_{\gamma_{\qq,\tau}} (\theta_{\vv})}_{=0} + \sigma\Jac_{\gamma_{\qq,\tau}}(\psi_{\qq,\tau})=
 \big( \tfrac{\diff}{\diff t} E_v(\gamma_{\qq,\tau},\dot\gamma_{\qq,\tau}) - E_q(\gamma_{\qq,\tau},\dot\gamma_{\qq,\tau}) \big)\sigma/\tau.
\end{align*}
In particular, $\xi$ is $C^\infty$.

Notice that, for all $\vv\in\Tan_{\qq}\Delta_{h,\rho}$, we have
\begin{align*}
\int_0^{h\tau}
&\diff E(\gamma_{\qq,\tau},\dot\gamma_{\qq,\tau})\big[(\theta_{\vv},\dot\theta_{\vv})\big]\,\diff t\\
&\qquad\qquad=
\tau \!\! \sum_{i\in\Z/h\Z } \Big[
E_q (q_{i},\nu^-(q_i,q_{i+1},\tau))v_i\\
&\qquad\qquad\qquad\qquad+
E_v (q_{i},\nu^-(q_i,q_{i+1},\tau)) \diff\nu^-(q_i,q_{i+1},\tau)\big [(v_i,v_{i+1},0) \big ]\Big]\\
&\qquad\qquad=
-\tau\!\!\sum_{i\in\Z/h\Z }  \langle L_{vv}\,v_i,\partial_\tau\nu^+(q_{i-1},q_i,\tau)-\partial_\tau\nu^-(q_{i},q_{i+1},\tau) \rangle.
\end{align*}
Moreover
\begingroup
\allowdisplaybreaks[1]
\begin{align*}
&\int_0^{h\tau}
\diff E(\gamma_{\qq,\tau},\dot\gamma_{\qq,\tau})\big [(\psi_{\qq,\tau},\dot\psi_{\qq,\tau})\big ]\,\diff t\\
&\qquad=
\sum_{i\in\Z/h\Z } \int_0^\tau
\Big(
\diff E(\gamma_{q_i,q_{i+1},\tau},\dot\gamma_{q_i,q_{i+1},\tau})\big [(\partial_\tau\gamma_{q_i,q_{i+1},\tau},\partial_\tau\dot\gamma_{q_i,q_{i+1},\tau})\big ]\\
&
\qquad\qquad\qquad\qquad
+
\underbrace{\diff E(\gamma_{q_i,q_{i+1},\tau},\dot\gamma_{q_i,q_{i+1},\tau})\big [(\dot\gamma_{q_i,q_{i+1},\tau}t/\tau,\ddot\gamma_{q_i,q_{i+1},\tau}t/\tau)\big ]}_{= (t/\tau) \frac{\diff}{\diff t} E(\gamma_{q_i,q_{i+1},\tau},\dot\gamma_{q_i,q_{i+1},\tau}) = 0}\\
&
\qquad\qquad\qquad\qquad
+
\diff E(\gamma_{q_i,q_{i+1},\tau},\dot\gamma_{q_i,q_{i+1},\tau})\big [(0,\dot\gamma_{q_i,q_{i+1},\tau} /\tau)\big ]\,\diff t
\Big)\\
&\qquad=
\sum_{i\in\Z/h\Z } \int_0^\tau
\partial_\tau\big( E(\gamma_{q_i,q_{i+1},\tau},\dot\gamma_{q_i,q_{i+1},\tau})\big)\,\diff t\\
&\qquad\qquad+
\int_0^{h\tau}
\diff E(\gamma_{\qq,\tau},\dot\gamma_{\qq,\tau})\big [(0,\dot\gamma_{\qq,\tau})\big ]
\,\diff t\\
&\qquad=
\sum_{i\in\Z/h\Z } \left(  \partial_\tau \int_0^\tau
 E(\gamma_{q_i,q_{i+1},\tau},\dot\gamma_{q_i,q_{i+1},\tau})\,\diff t -k\right)\\
&\qquad\qquad+
\int_0^{h\tau}
\diff E(\gamma_{\qq,\tau},\dot\gamma_{\qq,\tau})\big [(0,\dot\gamma_{\qq,\tau})\big ]
\,\diff t\\
&\qquad=
\tau \!\!\sum_{i\in\Z/h\Z } 
\partial_\tau\big(E (q_{i},\nu^-(q_i,q_{i+1},\tau))\big) +
\int_0^{h\tau}
 \langle L_{vv}\dot\gamma_{\qq,\tau},\dot\gamma_{\qq,\tau} \rangle\,\diff t.
\end{align*}
\endgroup

We plug these last two computations in equation~\eqref{e:ker_H_2}, and infer that, for all $(\vv,\sigma)\in\ker(H_{\qq,\tau})$,
\begin{align*}
-
\int_0^{h\tau}
\diff E(\gamma_{\qq,\tau},\dot\gamma_{\qq,\tau})\big[(\theta_{\vv},\dot\theta_{\vv})\big]\,\diff t
= \,& \sigma \bigg(
\int_0^{h\tau}
\diff E(\gamma_{\qq,\tau},\dot\gamma_{\qq,\tau})\big[(\psi_{\qq,\tau},\dot\psi_{\qq,\tau})\big]\,\diff t\\
& -
\int_0^{h\tau}
\langle L_{vv}\dot\gamma_{\qq,\tau},\dot\gamma_{\qq,\tau}\rangle\,\diff t
\bigg).
\end{align*}
Namely, $\sigma$ satisfies~\eqref{e:sigma} for $\xi:=\theta_{\vv}+\sigma\psi_{\qq,\tau}$.

On the other hand, if $\xi$ is an $h\tau$-periodic vector field along $\gamma_{\qq,\tau}$ that satisfies equation~\eqref{e:Jacobi_xi} and $\sigma\in\R$ satisfies~\eqref{e:sigma}, the difference $\xi-\sigma\psi_{\qq,\tau}$ is an $h\tau$-periodic continuous and piecewise-smooth Jacobi vector field, smooth on each interval of the form $[i\tau,(i+1)\tau]$. Therefore, there exists $\vv\in\Tan_{\qq}\Delta_{h,\rho}$ such that $\theta_{\vv}=\xi-\sigma\psi_{\qq,\tau}$.
\end{proof}

Assume that the critical point $(\qq,\tau)$ of $S_k$ is such that 
\begin{align*}
\dist(\gamma_{\qq,\tau}(t_0),\gamma_{\qq,\tau}(t_1))<\rho,\qquad
\forall t_0,t_1\in\R\mbox{ with }|t_1-t_0|\leq\tau.
\end{align*}
In particular, $(\qq,\tau)$ belongs to a circle of critical points $(\qq(s),\tau)\in\Delta_{h,\rho}$, for $s\in\R/h\tau\Z$, associated to the curves $\gamma_{\qq(s),\tau}:=\gamma_{\qq,\tau}(s+\cdot\,)$.
\begin{lem}\label{l:indices_critical_circles}
The indices $\ind(h_{\qq(s),\tau})$, $\nul(h_{\qq(s),\tau})$, $\ind(H_{\qq(s),\tau})$, and $\nul(H_{\qq(s),\tau})$ are independent of $s\in\R/h\tau\Z$.
\end{lem}

\begin{proof}
For each $s\in\R/h\tau\Z$ we set $P_s:=\diff\phi_L^{h\tau}(\gamma_{\qq,\tau}(s),\dot\gamma_{\qq,\tau}(s))$, so that
\begin{align*}
\nul(h_{\qq(s),\tau})=\dim\ker(P_s).
\end{align*}
The linear map $(\theta_{\vv}(0),\dot\theta_{\vv}(0))\mapsto(\theta_{\vv}(s),\dot\theta_{\vv}(s))$ is an isomorphism from $\ker(P_0)$ to $\ker(P_s)$, and in particular the function $s\mapsto\nul(h_{\qq(s),\tau})$ is constant on $\R/h\tau\Z$. 

The fact that the function $s\mapsto\nul(H_{\qq(s),\tau})$ is constant on $\R/h\tau\Z$ is implied by Lemma~\ref{l:kernel_H}. Indeed, assume that $\xi$ and $\sigma$ satisfy~\eqref{e:Jacobi_xi} and~\eqref{e:sigma}. For each $s\in\R/h\tau\Z$, the shifted vector field $\xi_s:=\xi(s+\cdot)$ satisfies
\begin{gather*}
\Jac_{\gamma_{\qq(s),\tau}}\xi_s= \big( \tfrac{\diff}{\diff t} E_v(\gamma_{\qq(s),\tau},\dot\gamma_{\qq(s),\tau}) - E_q(\gamma_{\qq(s),\tau},\dot\gamma_{\qq(s),\tau}) \big)\sigma/\tau
,\\
\sigma\!
\int_0^{h\tau} \langle L_{vv}(\gamma_{\qq(s),\tau},\dot\gamma_{\qq(s),\tau}) \dot\gamma_{\qq(s),\tau},\dot\gamma_{\qq(s),\tau}\rangle\,\diff t
= 
\tau\! 
\int_0^{h\tau} \diff E(\gamma_{\qq(s),\tau},\dot\gamma_{\qq(s),\tau})\big[(\xi_s,\dot\xi_s)\big]\,\diff t.
\end{gather*}

Let $A_s$ be the self-adjoint operator on $\Tan_{\qq(s)}\Delta_{h,\rho}\times\R$ associated to $H_{\qq(s),\tau}$, i.e.,
\begin{align*}
H_{\qq(s),\tau}((\vv,\sigma),(\ww,\mu))=\llangle A_s(\vv,\sigma),(\ww,\mu)\rrangle_{(\qq(s),\tau)}.
\end{align*}
Since the eigenvalues depend continuously on the operator, there exist continuous functions $\lambda_i:\R/h\tau\Z\to\R$, for $i=1,...,h\dim(M)$, such that the eigenvalues of each $A_s$ are precisely the numbers $\lambda_i(s)$, repeated with their multiplicity. Since the nullity
$
\nul(H_{\qq(s),\tau})=\#\{ i \ |\ \lambda_i(s)=0\}
$
is constant in $s$, we conclude that the index 
$
\ind(H_{\qq(s),\tau})=\#\{ i \ |\ \lambda_i(s)<0\}
$
is constant in $s$ as well. An entirely analogous argument shows that $\ind(h_{\qq(s),\tau})$ is constant is $s$ as well.
\end{proof}

\subsection{The iteration map}
The classical $m$-th \textbf{iteration map} on the free loop space sends any loop to its $m$-fold cover. In our finite dimensional setting, the analogous map $\psi^m:\Delta_{h,\rho}\times(0,\epsilon)\hookrightarrow\Delta_{mh,\rho}\times(0,\epsilon)$ is given by
\begin{align*}
\psi^m(\qq,\tau)=(\qq^m,\tau),
\end{align*}
where $\qq^m=(\qq,\qq,...,\qq)$ denotes the $m$-fold diagonal vector. This is essentially a ``linear'' embedding (it would be a linear map in the formal sense if $M$ were a Euclidean space instead of a closed manifold), and its differential is given by
\begin{align*}
\diff\psi^m(\qq,\tau)\big [(\vv,\sigma)\big ]=(\vv^m,\sigma).
\end{align*}
With a slight abuse of notation, we will use the symbol $S_k$ to denote the free-period action functional on both spaces $\Delta_{h,\rho}\times(0,\epsilon)$ and $\Delta_{mh,\rho}\times(0,\epsilon)$. For all $(\qq,\tau)\in \Delta_{h,\rho}\times(0,\epsilon)$ the gradient $(\ww,\mu):=\nabla S_k(\qq,\tau)$ with respect to the Riemannian metric~\eqref{e:discrete_Riemannian_metric} is given by
\begin{align*}
w_i & = L_v(q_i,\nu^+(q_{i-1},q_i,\tau)) - L_v(q_i,\nu^-(q_i,q_{i+1},\tau)),\qquad\forall i \in\Z/h\Z ,\\
\mu & = \frac{1}{h}\sum_{i\in\Z/h\Z }\Big[ k - E(q_{i},\nu^-(q_i,q_{i+1},\tau)) \Big]
\end{align*}
Analogously, the gradient $(\ww',\mu'):=\nabla S_k(\qq^m,\tau)$ is given by
\begin{align*}
w_i' & = w_{i\ \mathrm{mod}\ h},\qquad\forall i \in\Z/mh\Z ,\\
\mu' & = 
\frac{1}{mh}\sum_{i\in\Z/mh\Z }\Big[ k - E(q_{i},\nu^-(q_i,q_{i+1},\tau)) \Big]
=
\mu.
\end{align*}
This proves the following.

\begin{lem}\label{l:iterated_gradient}
For all $(\qq,\tau)\in \Delta_{h,\rho}\times(0,\epsilon)$ we have 
\begin{equation*}
\tag*{\qed}
\diff\psi^m(\qq,\tau)\nabla S_k(\qq,\tau)=\nabla S_k(\psi^m(\qq,\tau)).
\end{equation*}
\end{lem}

In order to simplify the notation, for all $m\in\N$ we abbreviate the Hessian bilinear forms as
\begin{align}\label{e:Hessians}
 h_m:=h_{\qq^m,\tau},\qquad H_m:=H_{\qq^m,\tau}.
\end{align}

\begin{lem}\label{l:equal_indices}
\hfill
\begin{itemize}
\item[(i)] The function $m\mapsto \nul(H_m)-\nul(h_m)$ is constant on $\N$, and takes values in $\{-1,0,1\}$;
\item[(ii)] The function $m\mapsto \ind(H_m)-\ind(h_m)$ is constant on 
\[\M:=\{m\in\N\ |\ \nul(h_1)=\nul(h_m)\}.\]
\item[(iii)] The function $m\mapsto\ind(h_m)$ grows linearly, and is bounded if and only if it vanishes identically.
\end{itemize}
\end{lem}

\begin{proof} 
We set 
\begin{align*}
\V_m & := \Tan_{\qq^m}\Delta_{mh,\rho}\times\{0\},\\
\K_m & :=  \ker(H_m)\cap\V_m= 
\ker(H_m)\cap\ker(h_m).
\end{align*}
Notice that $\K_m$ is a vector subspace of $\ker(H_m)$ of codimension at most one. Since $\K_m=\ker(F_m)$, where
\begin{align*}
F_m:\ker(h_m)\to\R,
\qquad
F_m(v)=H_m((v,0),(0,1)),
\end{align*}
$\K_m$ is also a vector subspace of $\ker(h_m)$ of codimension at most one. Therefore 
\[\nul(H_m)-\nul(h_m)\in\{-1,0,1\}.\]
We define the linear map 
\begin{gather*}
\omega^m:\Tan_{\qq^m}(\Delta_{mh,\rho})\times\R\to\Tan_{\qq}(\Delta_{h,\rho})\times\R,\\
\omega^m(\ww,\mu):=(\ww',m\mu),
\end{gather*}
where $w_i'=w_i+w_{i+h}+w_{i+2h}+...+w_{i+(m-1)h}$ for all $i\in\Z/h\Z $. This map is the adjoint of $\psi^m_*:=\diff\psi^m(\qq,\tau)$ with respect to the Hessian of the free-period action functionals, in the sense that
\begin{equation*}
H_m(\psi^m_*(\vv,\sigma),(\ww,\mu))
=
H_1((\vv,\sigma),\omega^m(\ww,\mu)).
\end{equation*}
This equation readily implies that
\begin{align*}
\psi^m_* \ker(H_1) & \subseteq \ker(H_m),\\
\psi^m_* \ker(h_1) & \subseteq \ker(h_m),\\
\psi^m_* \K_1 & \subseteq \K_m.
\end{align*}
We recall that $\psi^m_*$ is injective, since the iteration map $\psi^m$ is an embedding. 
If $(\vv,0)\in\ker(h_1)\setminus\K_1$, then $H_1((\vv,0),(0,1))\neq 0$ and 
\[H_m(\psi^m_*(\vv,0),(0,1))=H_1((\vv,0),\omega^m(0,1))=H_1((\vv,0),(0,m))\neq 0,\]
hence $\psi^m_*(\vv,0)\in\ker(h_m)\setminus\K_m$. Therefore, we have
\begin{align}
\label{e:same_nullity_1}
\nul(h_m)-\dim(\K_m)=\nul(h_1)-\dim(\K_1),\qquad\forall m\in\N.
\end{align}
If $(\vv,\sigma)\in\ker(H_1)\setminus\K_1$, then $\sigma\neq 0$ and $\psi^m_*(\vv,\sigma)=(\vv^m,\sigma)\in\ker(H_m)\setminus\K_m$. Therefore, we have
\begin{align}
\label{e:same_nullity_2}
\nul(H_m)-\dim(\K_m)=\nul(H_1)-\dim(\K_1),\qquad\forall m\in\N.
\end{align}
Equations~\eqref{e:same_nullity_1} and~\eqref{e:same_nullity_2} imply point~(i) of the lemma.

As for point~(ii), consider again the vector space $\V_m$ defined above, and its $H_m$-orthogonal
\begin{align*}
\V_m^\bot := \Big\{ (\vv,\sigma)\in\Tan_{\qq^m}\Delta_{mh,\rho}\times\R\ \Big|\ H_m((\vv,\sigma),\cdot)|_{\V_m}=0  \Big\}.
\end{align*}
An elementary formula from linear algebra (see, e.g., \cite[Prop.~A.3]{Mazzucchelli:2015zc}) allows us to relate the index of any symmetric  bilinear form to the index of its restriction to a vector subspace. For the Hessian $H_m$ and the subspace $\V_m$, such a formula reads
\begin{align*}
\ind(H_m)
=
\ind(h_m) + \ind(H_m|_{\V_m^\bot\times\V_m^\bot}) + \dim(\V_m\cap\V_m^\bot) - \dim(\K_m).
\end{align*}
Notice that 
$$\V_m\cap\V_m^\bot = \Big \{(\vv,0)\in \V_m \ \Big |\ H_m((\vv,0),\cdot\,))|_{\V_m}=0 \Big \} = \ker(h_m)$$
and therefore, by~\eqref{e:same_nullity_1}, the difference $\dim(\V_m\cap\V_m^\bot) - \dim(\K_m)$ is independent of $m\in\N$. In order to prove point~(ii) of the lemma, we are only left to show that the index $\ind(H_m|_{\V_m^\bot\times\V_m^\bot})$ is independent of $m\in\M$. For all $(\vv,\sigma)\in\V_1^\bot$ and $(\ww,0)\in\V_m$, we have
\begin{align*}
H_m(\psi^m_*(\vv,\sigma),(\ww,0))=H_1((\vv,\sigma),\omega^m(\ww,0))=H_1((\vv,\sigma),(\ww',0))=0.
\end{align*}
On the other hand, for all $(\vv,\sigma)\in\V_m^\bot$ and $(\ww,0)\in\V_1$, we have
\begin{align*}
H_1(\omega^m(\vv,\sigma),(\ww,0))=H_m((\vv,\sigma),\psi^m_*(\ww,0))=H_m((\vv,\sigma),(\ww^m,0))=0.
\end{align*}
Namely,
\begin{align*}
\psi^m_* \V_1^\bot & \subseteq \V_m^\bot,\\
\omega^m \V_m^\bot & \subseteq \V_1^\bot.
\end{align*}
The kernel $\ker(h_m)$ is contained in the vector space $\V_m^\bot$ with codimension at most one. For all $m\in\M$, we have $\psi^m_*\ker(h_1)=\ker(h_m)$. If there exists $(\vv,\sigma)\in\V_1^\bot\setminus\ker(h_1)$, then $\psi^m_*(\vv,\sigma)\not\in\V_m$, and therefore $\psi^m_*(\vv,\sigma)\in\V_m^\bot\setminus\ker(h_m)$. Analogously, if there exists $(\vv,\sigma)\in\V_m^\bot\setminus\ker(h_m)$, then $\omega^m(\vv,\sigma)\not\in\V_1$, and therefore $\omega^m(\vv,\sigma)\in\V_1^\bot\setminus\ker(h_1)$. This shows that $\psi^m_*\V_1^\bot=\V_m^\bot$. Since, for all $(\vv,\sigma)\in\V_1^\bot$, we have
\begin{align*}
H_m(\psi^m_*(\vv,\sigma),\psi^m_*(\vv,\sigma))
=
m H_1((\vv,\sigma),(\vv,\sigma)),
\end{align*}
we conclude that $\ind(H_m|_{\V_m^\bot\times\V_m^\bot})=\ind(H_1|_{\V_1^\bot\times\V_1^\bot})$ for all $m\in\M$.

Now, assume that the function $m\mapsto\ind(h_m)$ does not vanish identically, and consider $m_0\in\N$ such that $\ind(h_{m_0})>0$. In particular, there exists a non-zero vector $\vv\in\Tan_{\qq^{m_0}}\Delta_{m_0h,\rho}$ such that 
$\delta_1:=h_{m_0}(\vv,\vv)<0$.
We set
\begin{align*}
\delta_2  :=\, &  \langle L_{vv}\,v_{0},\diff\nu^+(q_{h-1},q_{0},\tau)\big [(0,v_{0},0)\big ] \rangle\\
& -
 \langle L_{vv}\,v_{m_0h-1},\diff\nu^-(q_{h-1},q_{0},\tau)\big [(v_{m_0h-1},0,0)\big ] \rangle,\\
m_1  :=\, & \big\lceil|\delta_2/\delta_1| \big\rceil.
\end{align*}
We consider an integer $m\geq m_0m_1+1$, and define the vectors
\begin{align*}
\ww_i:=(\bm0^{i(m_0m_1+1)},\vv^{m_1},\bm0^{m-i(m_0m_1+1)-m_0m_1})\in\Tan_{\qq^{m}}\Delta_{mh,\rho},\\
\forall i=0,...,\big\lfloor\tfrac{m}{m_0 m_1 +1}\big\rfloor -1,
\end{align*}
where $\bm0=(0,...,0)$ is the origin in $\Tan_{\qq}\Delta_{h,\rho}$. Notice that, for all $i\neq j$, we have
\begin{align*}
h_{m}(\ww_i,\ww_i) & = m_1\delta_1 + \delta_2<0,\\
h_{m}(\ww_i,\ww_j) & = 0.
\end{align*}
Therefore, the $\ww_i$'s span an $\big\lfloor\tfrac{m}{m_0 m_1 +1}\big\rfloor$-dimensional vector subspace over which the Hessian $h_{m}$ is negative definite. In particular 
\[
\ind(h_{m})\geq \big\lfloor\tfrac{m}{m_0 m_1 +1}\big\rfloor,\qquad \forall m\in\N.
\qedhere
\]
\end{proof}

The following lemma is the analog of a closed geodesics result due to Gromoll and Meyer \cite[Lemma~1.2]{Gromoll:1969gh} in our finite dimensional setting for the free-period action functional. In the infinite dimensional setting, the result was established in \cite[Lemma~1.2]{Abbondandolo:2014rb}.

\begin{lem}\label{l:partition_nullity}
Let $(\qq,\tau)$ be a critical point of the discrete free-period action functional $S_k$. The set of positive integers admits a partition $\N_1\cup...\cup\N_r$, integers $m_1\in\N_1$, ..., $m_r\in\N_r$, and $\nu_1,...,\nu_r\in\{0,1,...,2\dim(M)+1\}$ such that $m_j$ divides all the integers in $\N_j$, and $\nul(H_m)=\nu_j$ for all $m\in\N_j$.
\end{lem}

\begin{proof}
Thanks to Lemma~\ref{l:equal_indices}(i), it is enough to provide a proof for the nullity of $h_m$ instead of the nullity of $H_m$. The argument leading to~\eqref{e:nul_proof_1} shows that
\begin{align}\label{e:nul_proof_1bis}
\nul(h_m)=\dim\ker(P^m-I),\qquad\forall m\in\N.
\end{align}
The geometric multiplicity of the eigenvalue 1 of a power matrix varies as
\begin{align}\label{e:nul_proof_2}
\dim\ker(P^m-I)=\sum_{\lambda\in\sqrt[m]{1}} \dim_\C\ker_\C (P-\lambda I),
\end{align}
see, e.g., \cite[Prop.~A.1]{Mazzucchelli:2015zc}. The remaining of the proof is a standard arithmetic argument. We denote by $\sigma_m(P)$ the set of those eigenvalues of $P$ on the complex unit circle that are $m$-th roots of unity, and we introduce the following equivalence relation on the natural numbers: $m\sim m'$ if and only if $\sigma_m(P)=\sigma_{m'}(P)$. The subsets $\N_1,...,\N_r$ of the lemma will be the equivalence classes of this relation. Let $m_j$ be the minimum of $\N_j$, and $\sigma_{m_j}(P)=\{\exp(i2\pi p_1/q_1),...,\exp(i2\pi p_s/q_s)\}$, where $p_i$ and $q_i$ are relatively prime for all $i=1,...,s$. The integers is $\N_j$ are common multiples of $q_1,...,q_s$, and $m_j$ is the least common multiple of $q_1,...,q_s$. In particular, every $m\in\N_j$ is divisible by $m_j$ and, by~\eqref{e:nul_proof_1bis} and~\eqref{e:nul_proof_2}, we conclude
\begin{align*}
\nul(h_m)
=\!\!\sum_{\lambda\in\sigma_m(P)}\!\! \dim_\C\ker_\C (P-\lambda I)
=\!\!\sum_{\lambda\in\sigma_{m_j}(P)}\!\! \dim_\C\ker_\C (P-\lambda I)=\nul(h_{m_j}).
\end{align*}
\end{proof}

\section{Multiplicity of low energy periodic orbits}
\label{s:multiplicity}

This section is devoted to the proof of Theorem~\ref{t:multiplicity}, which will follow closely the one for the electromagnetic case in \cite{Abbondandolo:2014rb}. The essential difference in our general Tonelli case is that we need to employ the finite dimensional functional setting of Section~\ref{s:discretizations} in all the arguments involving the Hessian of the free-period action functional.

\subsection{A property of high iterates of periodic orbits}\label{s:iterated_mountain_passes}

Let $M$ be a closed manifold of arbitrary positive dimension, and $L:\Tan M\to\R$ a Tonelli Lagrangian. We fix an energy value $k\in\R$. Up to modifying $L$ outside an open neighborhood of $E^{-1}(-\infty,k]$, we can assume that $L(q,v)=\tfrac12g_q(v,v)$ outside a compact set of $\Tan M$, where $g$ is some Riemannian metric on $M$. We  will adopt the notation of Section~\ref{s:discretizations}, and in particular we consider the positive constants $\rho=\rho(L)$ and $\epsilon=\epsilon(L)$ given by Proposition~\ref{p:unique_action_minimizers}. Let $\gamma:\R/\sigma\Z\to M$ be a periodic orbit of the Lagrangian system of $L$ with energy $k$. We choose an integer $h\geq1/\epsilon$ large enough so that
\begin{align}\label{e:discretization_step_big}
\dist(\gamma(t_0),\gamma(t_1))<\rho,\qquad\forall t_0,t_1\in\R\mbox{ with }|t_0-t_1|\leq\sigma/h.
\end{align}
Therefore, there exists a critical point $(\qq,\tau)$ of the discrete free-period action functional $S_k:\Delta_{h,\rho}\times(0,\epsilon)\to \R$ such that $\sigma=h\tau$ and $\gamma=\gamma_{\qq,\tau}$. We denote by 
\begin{align*}
K=\bigcup_{s\in\R/h\tau\Z} \{(\bm q(s),\tau)\},
\end{align*}
the critical circle containing $(\qq,\tau)$, where $(\bm q(s),\tau)$ is the critical point associated to the periodic orbit 
\[\gamma_{\bm q(s),\tau}:=\gamma_{\bm q,\tau}(s+\cdot).\] 
Let $c:=S_k(\qq,\tau)$. We assume that $K$ is isolated in $\mathrm{crit}(S_k)$. Since $S_k$ is a smooth function, $K$ has an arbitrarily small connected open neighborhood $U\subset\Delta_{h,\rho}\times(0,\epsilon)$ such that the intersection $\{S_k<c\}\cap U$ has only finitely many connected components $U^-_1,...,U^-_l$ (this follows from the more general statement that the local homology of the isolated critical circle $K$ has finite rank). We can assume that
\begin{align}\label{e:K_intersects_boundary_nbhd}
\partial U^-_i\cap K\neq\varnothing,\qquad\forall i=1,...,l.
\end{align}
Indeed, if~\eqref{e:K_intersects_boundary_nbhd} is not verified, we remove from $U$ the closure of each $U^-_i$ whose boundary does not intersect $K$. This leaves us with a smaller connected open neighborhood of $K$ that satisfies~\eqref{e:K_intersects_boundary_nbhd}.

\begin{lem}\label{l:K_contained_in_bdary}
$\displaystyle K\subset  \partial U^-_1 \cap ... \cap \partial U^-_l$.
\end{lem}

\begin{proof}
For each $(\qq',\tau')\in\Delta_{h,\rho}$ and $s\in\R$, we set
\begin{align*}
\bm q'(s):=\big(\gamma_{\qq',\tau'}(s),\gamma_{\qq',\tau'}(s+\tau'),...,\gamma_{\qq',\tau'}(s+(h-1)\tau')\big).
\end{align*}
By~\eqref{e:discretization_step_big}, if $(\qq',\tau')$ is sufficiently close to $K$, all points $\qq'(s)$, for $s\in\R$, belong to $\Delta_{h,\rho}$, and actually all points $(\qq'(s),\tau')$ belong to the neighborhood $U$. Moreover, $S_k(\qq'(s),\tau')\leq S_k(\qq',\tau')<c$, that is, 
\begin{align}\label{e:rotation_preserves_sublevel}
(\qq'(s),\tau')\in\{S_k<c\}.
\end{align}
Fix an arbitrary $i\in\{1,...,l\}$. By~\eqref{e:K_intersects_boundary_nbhd}, there exists $(\qq(s_0),\tau)\in \partial U_i^-\cap K$. If we take a sequence of points $(\qq',\tau')\in U_i^-$ such that $(\qq',\tau')\to(\qq(s_0),\tau)$, we have $(\qq'(s),\tau')\to(\qq(s_0+s),\tau)$ for all $s\in(0,h\tau)$. This, together with~\eqref{e:rotation_preserves_sublevel}, implies that $K\subset\partial U_i^-$.
\end{proof}

\begin{lem}\label{l:Bangert_trick}
For all $m\in\N$ large enough, the image under the iteration map $\psi^m(\{S_k<c\}\cap U)$ is contained in one connected component of the sublevel set $\{S_k<mc\}$. 
\end{lem}

\begin{proof}
Consider two distinct connected components $U_1^-$ and $U_2^-$ of $\{S_k<c\}\cap U$ (if this latter intersection is connected we are already done). We denote by $B_r\subset \Delta_{h,\rho}\times(0,\epsilon)$ an open ball of radius $r>0$ centered at $(\qq,\tau)\in K$. By Lemma~\ref{l:K_contained_in_bdary}, the critical point $(\qq,\tau)$ belongs to the intersection $\partial U_1^-\cap\partial U_2^-$. Therefore, there exists a continuous path $(\qq',\tau'):[0,1]\to B_r$ such that $(\qq'(0),\tau'(0))\in U_1^-$ and $(\qq'(1),\tau'(1))\in U_2^-$. Consider the embedding $\iota_h$ defined in~\eqref{e:iota_h}, and the continuous path 
\begin{gather*}
\iota_h\circ(\qq',\tau'):[0,1]\to W^{1,2}(\R/\Z,M)\times(0,\infty),\\
\iota_h\circ(\qq'(u),\tau'(u))=(\Gamma_{\qq'(u),\tau'(u)},h\tau'(u)).
\end{gather*}
By Bangert's trick of ``pulling one loop at the time'' (see \cite[pages~86-87]{Bangert:1980ho}, \cite[page~421]{Abbondandolo:2013is}, or \cite[Section~3.2]{Mazzucchelli:2014ys}), for all integers $m$ large enough there exists a continuous homotopy 
\[F_s:[0,1]\to W^{1,2}(\R/\Z,M)\times(0,mh\epsilon),\qquad s\in[0,1],\] 
such that $F_0=\iota_{mh}\circ\psi^m\circ(\qq',\tau')$, $F_s(0)=F_0(0)$ and $F_s(1)=F_0(1)$ for all $s\in[0,1]$, and $F_1([0,1])\subset\{\SSS_k<mc\}$. Moreover, up to reducing the radius $r$ of the original ball (and consequently increasing the integer $m$ necessary for Bangert's trick), we can choose such $F_s$ so that it takes values in an arbitrarily small $C^0$-neighborhood of the $m$-th iterate of $\gamma_{\qq,\tau}=(\Gamma_{\qq,\tau},h\tau)$. We write the homotopy  as $F_s(u)=(\Gamma''_s(u),mh\tau''_s(u))$. We define another continuous homotopy 
\begin{gather*}
f_s:[-1,1]\to \Delta_{mh,\rho}\times(0,\epsilon),\\
f_s(u):=(\qq''(s,u),\tau''_s(u)),
\end{gather*}
where
\begin{align*}
\qq''(s,u)
:=
\big(\Gamma''_s(u)(0),\Gamma''_s(u)(\tfrac1{mh}),...,\Gamma''_s(u)(\tfrac{mh-1}{mh})\big).
\end{align*}
The fact that $\qq''(s,u)$ is contained in $\Delta_{mh,\rho}$ follows from~\eqref{e:discretization_step_big} together with the fact that $F_s$ takes values inside a small $C^0$-neighborhood of the $m$-th iterate of $\gamma_{\qq,\tau}=(\Gamma_{\qq,\tau},h\tau)$. Notice that $f_0(u)=(\qq'(u),\tau'(u))$, $f_s(-1)=f_0(-1)$ and $f_s(1)=f_0(1)$ for all $s\in[0,1]$, and
$S_k(f_1(u)) \leq \SSS_k(F_1(u)) < mc$. The existence of the continuous path $f_1$ implies the lemma.
\end{proof}

One of the main ingredients for the proof of Theorem~\ref{t:multiplicity} is the following statement about high iterates of periodic orbits. In the special case of Tonelli Lagrangians that restrict to polynomials of degree 2 in any fiber of the tangent bundle, an infinite dimensional version of this theorem was established in \cite[Theorem~2.6]{Abbondandolo:2014rb}. A similar statement for closed Riemannian geodesics on surfaces was originally established by Bangert in \cite[Theorem~2]{Bangert:1980ho}.

\begin{lem}\label{l:high_iterates}
Assume that, for all $m\in\N$, the point $(\qq^m,\tau)$ belongs to an isolated critical circle of $S_k:\Delta_{mh,\rho}\times(0,\epsilon)\to \R$. Then, if $m$ is large enough, there exists a connected open neighborhood $W\subset\Delta_{mh,\rho}\times(0,\epsilon)$ of the critical circle of $(\qq^m,\tau)$ such that the inclusion map induces the injective map among connected components
\begin{align}\label{e:injective_pi_0}
\pi_0(\{S_k<mc\})\hookrightarrow\pi_0(\{S_k<mc\}\cup W).
\end{align}
\end{lem}

\begin{rem}
This lemma can be phrased by saying that sufficiently high iterates of periodic orbits are not 1-dimensional mountain passes for the discrete free-period action functional. By applying the same techniques as in \cite[Section~4]{Macarini:2015rm}, one can prove that, for any given periodic orbit and degree $d\geq1$, sufficiently high iterates of the periodic orbit are not $d$-dimensional mountain passes for the discrete free-period action functional. Since the case $d=1$ suffices for the proof of Theorem~\ref{t:multiplicity}, we will only present this case.
\hfill\qed
\end{rem}

\begin{proof}[Proof of Lemma~\ref{l:high_iterates}]
As in~\eqref{e:Hessians}, we denote by $H_m$ and $h_m$ the Hessians of the discrete free-period and fixed-period action functionals at $(\qq^m,\tau)$. By Lemma~\ref{l:equal_indices}(iii), either $\ind(h_m)\to\infty$ as $m\to\infty$ or $\ind(h_m)=0$ for all $m\in\N$. 

In the former case, by Lemma~\ref{l:equal_indices}(i), for all $m$ large enough we have $\ind(H_m)\geq2$. By Lemma~\ref{l:indices_critical_circles}, the discrete free-period action functional $S_k$ have the same Morse index (larger than or equal to $2$) and nullity at all the critical points belonging to the critical circle of $(\qq^m,\tau)$. In particular, there exist arbitrarily small connected open neighborhoods $W\subset\Delta_{mh,\rho}\times(0,\epsilon)$ of the critical circle of $(\qq^m,\tau)$ such that the intersection $\{S_k<mc\}\cap W$ is connected, and~\eqref{e:injective_pi_0} follows.

Let us now deal with the second case in which $\ind(h_m)=0$ for all $m\in\N$. We denote by $K_m$ the critical circle of $S_k$ containing $(\qq^m,\tau)$. We consider the partition $\N=\N_1\cup...\cup\N_r$ given by Lemma~\ref{l:partition_nullity}, and the integers $m_i:=\min\N_i$. By Lemmas~\ref{l:equal_indices}(ii) and~\ref{l:partition_nullity}, we have $\ind(H_m)=\ind(H_{m_i})$ and $\nul(H_m)=\nul(H_{m_i})$ for all $m\in\N_i$. By Lemma~\ref{l:Bangert_trick} and the prior discussion, for each $i\in\{1,...,r\}$ there exists a connected open neighborhood $U$ of $K_{m_i}$ such that, for all $m\in\N_i$ large enough, $\psi^{m/m_i}(\{S_k<m_ic\}\cap U)$ is contained in one connected component of the sublevel set $\{S_k<mc\}$.

Let $U'\subset U$ be a sufficiently small connected neighborhood of $K_{m_i}$. Since the iteration map $\psi^{m/m_i}:\Delta_{m_i h,\rho}\times(0,\epsilon)\hookrightarrow\Delta_{mh,\rho}\times(0,\epsilon)$ is an embedding, $\psi^{m/m_i}(U')$ admits a tubular neighborhood $W\subset \Delta_{mh,\rho}\times(0,\epsilon)$, which we identify with an open neighborhood of the zero-section of the normal bundle of $\psi^{m/m_i}(U')$ with projection $\pi: W\to\psi^{m/m_i}(U')$. We denote the points of $W$ by $(x,v)$, where $x\in\psi^{m/m_i}(U')$ is a point in the base and $v\in\pi^{-1}(x)$ is a point in the corresponding fiber. By Lemma~\ref{l:iterated_gradient}, the gradient of $S_k$ is tangent to the zero-section of $W$ at all points $(x,0)\in W$. Therefore, the restriction of $S_k$ to any fiber $\pi^{-1}(x)$ has a critical point at the origin. We introduce the radial deformation $r_t:W\to W$, for $t\in[0,1]$, given by $r_t(x,v)=(x,(1-t)v)$. This is a deformation retraction of the neighborhood $W$ onto the base $\psi^{m/m_i}(U')$. By Lemma~\ref{l:indices_critical_circles}, every critical point of $S_k$ on the critical circle $\psi^{m/m_i}(K)$ has Morse index $\ind(H_m)$ and nullity $\nul(H_m)$. Since $\ind(H_m)=\ind(H_{m_i})$ and $\nul(H_m)=\nul(H_{m_i})$, up to shrinking $U'$ we have that the restriction of the discrete free-period action functional $S_k:W\to\R$ to any fiber $\pi^{-1}(x)$ has a non-degenerate local minimum at the origin; up to shrinking $W$, such a local minimum is a global minimum and we have $\tfrac{\diff}{\diff t}S_k\circ r_t\leq 0$. In particular, the deformation $r_t$ preserves the intersection $\{S_k<mc\}\cap W$. Since $r_1(\{S_k<mc\}\cap W)=\psi^{m/m_i}(U')$, we conclude that $\{S_k<mc\}\cap W$ is contained in one connected component of $\{S_k<mc\}$, which implies~\eqref{e:injective_pi_0}.
\end{proof}

\subsection{The minimax scheme}\label{ss:minimax}

From now on, we will further assume that $M$ is a closed surface, i.e.,
\begin{align*}
\dim(M)=2. 
\end{align*}
Up to lifting $L$ to the tangent bundle of the orientation double cover of $M$, we can assume that $M$ is orientable. We recall that the Ma\~n\'e critical value $\cu(L)$ is defined as the minimum $k$ such that the free-period action functional $\SSS_k$ is non-negative on the connected component of contractible periodic curves, or equivalently as the usual Ma\~n\'e critical value of the lift of $L$ to the tangent bundle of the universal cover of $M$. It is easy to see that $e_0(L)\leq \cu(L)$.

In order to prove Theorem~\ref{t:multiplicity}, we employ the periodic orbit provided by Theorem~\ref{t:local_minimizers}: since such a periodic orbit is a local minimizer of the free-period action functional $\SSS_k:W^{1,2}(\R/\Z,M)\times(0,\infty)\to\R$, and since this functional is unbounded from below on every connected component when $k<\cu(L)$, we can find a second periodic orbit by performing a suitable 1-dimensional minimax. The framework of such a minimax is not entirely standard, since $\SSS_k$ might not satisfy the Palais-Smale condition. The argument for the case of fiberwise quadratic Lagrangians provided in \cite[Sections~3.1-3.2]{Abbondandolo:2014rb} actually works for general Tonelli Lagrangians; indeed, such an argument only requires that $\SSS_k$ is $C^1$. We will thus present the construction and the properties of the minimax scheme (Lemma~\ref{l:minimax} below) without proofs,  which the reader can find in \cite[Sections~3.1-3.2]{Abbondandolo:2014rb}.

We fix, once for all, an energy value $k_0\in(e_0(L),\cu(L))$. Let $\gamma_{k_0}=(\Gamma_{k_0},\tau_{k_0})\in W^{1,2}(\R/\Z,M)\times(0,\infty)$ be a local minimizer  of $\SSS_{k_0}$ with action $\SSS_{k_0}(\gamma_{k_0})<0$, whose existence is guaranteed by Theorem~\ref{t:local_minimizers}. Since the configuration space $M$ is assumed to be an orientable surface, every iterate of $\gamma_{k_0}$ is a local minimizer of $\SSS_{k_0}$ as well, see \cite[Lemma~4.1]{Abbondandolo:2015lt}.  The free-period action functional satisfies the Palais-Smale condition on subsets of the form $W^{1,2}(\R/\Z,M)\times[\epsilon_0,\epsilon_1]$, where $0<\epsilon_0\leq\epsilon_1<\infty$. Therefore, we can find $\epsilon_0,\epsilon_1$ and a sufficiently small open neighborhood $\UU\subset W^{1,2}(\R/\Z,M)\times[\epsilon_0,\epsilon_1]$ of the critical circle $C$ of $\gamma_{k_0}$ such that
\begin{gather*}
\UU\cap\mathrm{crit}(\SSS_{k_0})=C,\\
\inf_{\partial \UU} \SSS_{k_0}>\SSS_{k_0}(\gamma_{k_0}).
\end{gather*}
The free-period action functional $\SSS_{k_0}$ is unbounded from below in every connected component, for $k_0<\cu(L)$. In particular, there exists $\zeta\in W^{1,2}(\R/\Z,M)\times(0,\infty)$ in the same connected component as $\gamma_{k_0}$ such that $\SSS_{k_0}(\zeta)<\SSS_{k_0}(\gamma_{k_0})$. For each $k\in\R$, we denote by $\MM_{k}\subset W^{1,2}(\R/\Z,M)\times[\epsilon_0,\epsilon_1]$ the (possibly empty) closure of the set of local minimizers of $\SSS_{k}|_{\UU}$. Notice that $\MM_{k_0}=C$. As customary, we denote by $\gamma^n:\R/n\tau\Z\to M$ the $n$-th iterate of a closed curve $\gamma:\R/\tau\Z\to M$. For all $n\in\N$ and $k\in \R$, we denote by $\PP(n,k)$ the set of continuous paths 
\[p:[0,1]\to W^{1,2}(\R/\Z,M)\times(0,\infty)\]
such that $p(0)=\zeta^n$ and $p(1)=\mu^n$ for some $\mu\in\MM_{k}$. The following statement is quoted from \cite[Lemma~3.5]{Abbondandolo:2014rb}.

\begin{lem}\label{l:minimax}
There exists a full-measure subset $J$ of a closed neighborhood of $k_0$ such that, for all $n\in\N$ and $k\in J$, the real number
\begin{align*}
c(n,k):=\inf_{p\in\PP(n,k)} \max_{s\in[0,1]} \SSS_k(p(s))
\end{align*}
is a critical value of the free-period action functional $\SSS_k$, and 
\begin{align*}
\lim_{n\to\infty} c(n,k)=-\infty.
\end{align*}
For every neighborhood $\VV$ of $\mathrm{crit}(\SSS_{k})\cap \SSS_{k}^{-1}(c(n,k))$ there exists $p\in\PP(n,k)$ such that 
$p([0,1])\subset \{\SSS_{k}<c(n,k)\}\cup\VV$.
\hfill\qed
\end{lem}

\begin{proof}[Proof of Theorem~\ref{t:multiplicity}]
We prove the theorem by contradiction. We fix, once for all, an energy value $k_0\in(e_0(L),\cu(L))$ and the corresponding set $J\subset(e_0(L),\cu(L))$ given by Lemma~\ref{l:minimax}. It is enough to prove that, for each $k\in J$, the family of critical points of $\SSS_k$ with critical values in $\{c(n,k)\,|\,n\in\N\}$ contains infinitely many (geometrically distinct) periodic orbits. Assume by contradiction that this does not hold. Therefore, there exist finitely many critical points $\zeta_1,...,\zeta_r$ of $\SSS_k$ such that, for all $n\in\N$, all the critical points of $\SSS_k$ with critical values $c(n,k)$ are iterates of some periodic orbits among $\zeta_1,...,\zeta_r$.

We consider the finite dimensional setting of Section~\ref{s:discretizations}. For an integer $h\in\N$, we introduce the open subset $\WW\subset W^{1,2}(\R/\Z,M)\times(0,\infty)$ given by those periodic curves $\gamma=(\Gamma,\tau)\in W^{1,2}(\R/\Z,M)\times(0,\infty)$ such that 
\begin{align*}
\dist(\gamma(t_0),\gamma(t_1))<\rho,\qquad\forall t_0,t_1\in\R\mbox{ with }|t_0-t_1|\leq\tau/h. 
\end{align*}
We define a homotopy $h_s:\WW\to W^{1,2}(\R/\Z,M)\times(0,\infty)$, for $s\in[0,1]$, as follows. For each $s\in[0,1]$ and $\gamma=(\Gamma,\tau)\in\WW$, we set
\begin{align*}
q_i:=\gamma\big(\tfrac {i}h \tau\big), 
\qquad q_i':=\gamma\big(\tfrac{(i+s)}{h} \tau\big),
\qquad\forall i=0,...,h-1,
\end{align*}
and, for all $i=0,...,h-1$ and $t\in[0,\tau/h]$,
\begin{align*}
h_s(\gamma)\big(\tfrac {i}h\tau + t\big)
=
\left\{
  \begin{array}{lll}
    \gamma_{q_i,q_i',s\tau/h}(t), & & \mbox{if }t\in[0,s\tau/h], \vspace{5pt}\\ 
    \gamma(i\tau + t) &  & \mbox{if }t\in[s\tau/h,\tau/h]. \\ 
  \end{array}
\right.
\end{align*}
Here, we have adopted the notation of Section~\ref{s:discretizations}, denoting by $\gamma_{q_i,q_i',s\tau/h}$ the non-degenerate unique action minimizer given by Proposition~\ref{p:unique_action_minimizers}. Notice that 
$h_0$ is the identity, and the image of $h_1$ is contained in the image of the embedding $\iota_h$ defined in~\eqref{e:iota_h}. Moreover,
\begin{align*}
\tfrac{\diff}{\diff s} \SSS_k(h_s(\gamma))\leq0, \qquad\forall\gamma\in\WW,
\end{align*}
that is, the homotopy $h_s$ preserves the sublevels of the free-period action functional.

We choose $h$ large enough so that $\zeta_1,...,\zeta_r\in\WW$. In particular, there exist points $(\qq_i,\tau_i)\in\Delta_{h,\rho}\times(0,\epsilon)$ such that  
\begin{align*}
\iota_h(\qq_i,\tau_i)=\zeta_i,\qquad\forall i=1,...,r.
\end{align*}
We set $c_i:=\SSS_k(\zeta_i)=S_k(\qq_i,\tau_i)$. Let us apply Lemma~\ref{l:high_iterates}, which gives an integer $m_0\in\N$ with the following property. For each $i\in\{1,...,r\}$ and $m\geq m_0$, there exists a neighborhood $U_{i,m}$ of the critical circle of $(\qq_i^m,\tau_i)$ such that the inclusion induces the injective map among connected components
\begin{align}\label{e:injective_pi_0_final_proof}
\pi_0(\{S_k < mc_i\})
\hookrightarrow
\pi_0(\{S_k < mc_i\}\cup U_{i,m}).
\end{align}

Now, by Lemma~\ref{l:minimax}, $c(n,k)\to-\infty$ as $n\to\infty$. Therefore, if we choose $n\in\N$ large enough, we have that
\begin{align*}
\mathrm{crit}(\SSS_k)\cap \SSS_k^{-1}(c(n,k)) 
= 
\big\{ \zeta_{i_1}^{m_1},...,\zeta_{i_u}^{m_u}\big\}
\end{align*}
for some $i_1,...,i_u\in\{1,...,r\}$ and $m_1,...,m_u\in\N$ such that 
\begin{align*}
\min\{m_1,...,m_u\}\geq m_0.
\end{align*}
For each $v=1,...,u$, we fix a small enough open neighborhood $\VV_{v}'\subset \WW$ of the critical circle of $\zeta_{i_v}^{m_v}$ such that $h_1(\VV_v')\subset \iota_{m_{i_v}h}(U_{i_v,m_{i_v}})$. We further choose a smaller open neighborhood $\VV_{v}$ of the critical circle of $\zeta_{i_v}^{m_v}$ whose closure is contained in $\VV_{v}'$. By Lemma~\ref{l:minimax}, there exists a path $p\in\PP(n,k)$ such that 
\begin{align*}
p([0,1])\subset \{\SSS_k < c(n,k)\}\cup\VV_1\cup...\cup\VV_u.
\end{align*}
Let $\chi:W^{1,2}(\R/\Z,M)\times(0,\infty)\to[0,1]$ be a continuous bump function that is identically equal to 1 on $\VV_1\cup...\cup\VV_u$ and is supported inside $\VV_1'\cup...\cup\VV_u'$. We define the path $p'\in\PP(n,k)$ by 
\begin{align*}
p'(s)
:=
h_{\chi(p(s))}(p(s)). 
\end{align*}
Notice that 
\begin{align*}
p'([0,1])\subset \{\SSS_k<c(n,k)\}\cup \iota_{m_1h}(U_{i_1,m_1})\cup...\cup \iota_{m_uh}(U_{i_u,m_u}).
\end{align*}
This, together with equation~\eqref{e:injective_pi_0_final_proof}, implies that there exists a path $p''\in\PP(n,k)$ such that $p''([0,1])\subset \{\SSS_k<c(n,k)\}$, which contradicts the definition of the minimax value $c(n,k)$. This completes the proof.
\end{proof}

\section{A Tonelli Lagrangian with few periodic orbits on low energy levels}
\label{s:counterexample}

If $L:\Tan M\to\R$ is a Tonelli Lagrangian with associated energy $E:\Tan M\to\R$, the Euler-Lagrange flow on any energy hypersurface $E^{-1}(k)$ is conjugate to the Hamiltonian flow of the dual Tonelli Hamiltonian $H:\Tan^*M\to\R$, which is given by~\eqref{e:Hamiltonian}, on the energy hypersurface $H^{-1}(k)$. We recall that the Hamiltonian vector field $X_H$ generating the Hamiltonian flow is defined by $\diff q\wedge\diff p (X_H,\cdot\,)=\diff H$. Moreover, $H(q,p)=E(q,H_p(q,p))$ for all $(q,p)\in\Tan^*M$, and therefore
\begin{align*}
\min E & =\min H,\\
e_0(L) & = e_0(H):= \min\big\{ k\in\R\ \big |\ \pi(H^{-1}(k))=M \big \},
\end{align*}
where $\pi:\Tan^*M\to M$ denotes the projection onto the base of the cotangent bundle.

We want to show that the assertions of Theorems~\ref{t:local_minimizers} and~\ref{t:multiplicity} do not necessarily hold on the energy range $[\min E,e_0(L)]$. We will provide a counterexample for these statements in the Hamiltonian formulation. We fix two positive real numbers $r_1<r_2$ such that their quotient $r_1/r_2$ is irrational, a real number $R>r_2$,  and a smooth monotone function $\chi:[0,\infty)\to[0,\infty)$ such that   $\chi(x)=x$ for all $x\in[0,r_2]$, and  $\chi(x)=R$ for all $x\geq R$.  We define a Tonelli Hamiltonian $H:\Tan^*\R^2\to\R$ by
\begin{align*}
H(q_1,q_2,p_1,p_2) = \frac12 \left( \frac{\chi(|q_1|^2)+|p_1|^2}{r_1} + \frac{\chi(|q_2|^2)+|p_2|^2}{r_2} \right).
\end{align*}
Notice that
\begin{align*}
0= \min H < \frac12 < e_0(H) = \frac12 \left( \frac{R}{r_1} + \frac{R}{r_2} \right).
\end{align*}
If we identify $\Tan^*\R^2$ with $\C^2$ by means of the map $(q_1,q_2,p_1,p_2)\mapsto(x_1+iy_1,x_2+iy_2)$, the Hamiltonian flow of $H$ on the open polydisk $B(r_1)\times B(r_2)\subset\C^2$ is given by the unitary linear map
\begin{align*}
\phi_H^t
=
\left(
  \begin{array}{@{}cc@{}}
    e^{-it/r_1} & 0 \vspace{5pt}\\ 
    0 & e^{-it/r_2} 
  \end{array}
\right).
\end{align*}
Since $r_1/r_2$ is irrational, for each $k\in(0,1/2)$ there exists only two periodic orbits of the Hamiltonian flow $\phi_H^t$ with energy $k$. These orbits are given by
\begin{align*}
\Gamma_k  :  \R/2\pi r_1\Z\to\C^2, &\qquad \Gamma_k(t)=\phi_H^t\big(\sqrt{2 r_1 k} ,0\big),\\
\Psi_k  :  \R/2\pi r_2\Z\to\C^2, &\qquad \Psi_k(t)=\phi_H^t\big(0,\sqrt{2 r_2 k}\big).
\end{align*}
Their Maslov index, which can be computed as in \cite[example~3.5]{Mazzucchelli:2015zc}, is given by 
\begin{align*}
\textrm{mas}(\Gamma_k)= 2\big( \lfloor r_1/r_2\rfloor +1\big),\\ 
\textrm{mas}(\Psi_k)= 2\big( \lfloor r_2/r_1\rfloor +1\big).
\end{align*}
Let $\pi:\Tan^*\R^2\to\R$ be the base projection. The curves $\gamma_k:=\pi\circ\Gamma_k$ and $\psi_k:=\pi\circ\Psi_k$ are the only periodic orbits of the Lagrangian system of $L$ with energy $k$. The Morse index Theorem for Tonelli Lagrangians (see, e.g., \cite[Theorem~4.1]{Mazzucchelli:2015zc}) implies that the Morse index of the fixed-period Lagrangian action functional at a periodic orbit is equal to the Maslov index of its corresponding Hamiltonian periodic orbit. Since the Morse index of the free-period action functional at a periodic orbit is larger than or equal to the corresponding Morse index of the fixed-period action functional, and since $\textrm{mas}(\Gamma_k)>0$ and $\textrm{mas}(\Psi_k)>0$, we infer that $\gamma_k$ and $\psi_k$ are not local minima of the free-period action functional of $L$ at energy $k$. Finally, notice that $H$ does not depend on the base variables $(q_1,q_2)$ in the region where $\min\{|q_1|^2,|q_2|^2\}>R$. Therefore, we can replace the configuration space with the 2-torus 
$\T^2=([-R,R]/\{-R,R\})\times([-R,R]/\{-R,R\})$, 
and see $H$ as a Tonelli Hamiltonian of the form $H:\Tan^*\T^2\to\R$. Equivalently we can see $L$ as a Tonelli Lagrangian of the form $L:\Tan \T^2\to\R$. Such a Tonelli Lagrangian violates the assertions of Theorems~\ref{t:local_minimizers} and~\ref{t:multiplicity} in the energy range $(0,1/2)\subset[\min E,e_0(L)]$.

\section{Tonelli setting with a general magnetic form}
\label{s:non_exact}

Let $M$ be a closed manifold, $L:\Tan M\to\R$ a Tonelli Lagrangian, and $\sigma$ a closed 2-form on $M$. The pair $(L,\sigma)$ defines a flow on the tangent bundle $\Tan M$, which we call the \textbf{Euler-Lagrange flow} of $(L,\sigma)$, as follows: for each open subset $U\subset M$ such that $\sigma|_U$ admits a primitive $\theta$, the Euler-Lagrange flow of $(L,\sigma)$ coincides with the usual Euler-Lagrange flow of the Tonelli Lagrangian $L+\theta$ within $\Tan U$. This is a good definition, since the Euler-Lagrange flow of a Lagrangian $L$ does not change if we add a closed 1-form to $L$. Since the energy function $E:\Tan M\to\R$ associated to $L$ is also the energy function associated to $L+\theta$, the Euler-Lagrange flow of $(L,\sigma)$ preserves $E$. This dynamics can be equivalently described in the Hamiltonian formulation as the Hamiltonian flow of the  Tonelli Hamiltonian dual to $L$ computed with respect to the twisted symplectic form $\diff q\wedge\diff p-\pi^*\sigma$ on $\Tan^*M$, where $\pi:\Tan^*M\to M$ is the base projection.

The periodic orbits of the Euler-Lagrangian flow of $(L,\sigma)$ on the energy hypersurface $E^{-1}(k)$ are in one-to-one correspondence with the zeroes of the action 1-form $\eta_k$ on $W^{1,2}(\R/\Z,M)\times (0,\infty)$, which is given by
\begin{align*}
\eta_k(\Gamma,\tau)\big[(\Psi,\mu)\big] = \diff \SSS_k(\Gamma,\tau)\big[(\Psi,\mu)\big] + \int_0^1 \sigma_{\Gamma(t)}(\dot\Gamma(t),\Psi(t))\, \diff t. 
\end{align*}
Here, we have denoted by $\SSS_k:W^{1,2}(\R/\Z,M)\times (0,\infty)\to\R$ the free-period action functional of $L$ at energy $k$. When the 2-form $\sigma$ admits a primitive $\theta$ on $M$, the action 1-form $\eta_k$ coincides with the differential of the free-period action functional with energy $k$ of the Lagrangian $L+\theta$. We refer the reader to \cite{Asselle:2014hc, Asselle:2015ij, Asselle:2015sp} and to the references therein for more background on the action 1-form. From now on, we assume that $M$ is a closed surface, i.e.,
\[\dim(M)=2.\]
Given a zero $(\Gamma,\tau)$ of the action 1-form, there exists a neighborhood $U\subset M$ of $\Gamma(\R/\Z)$ such that $\sigma|_U$ is exact. In particular, on the neighborhood $W^{1,2}(\R/\Z,U)\times (0,\infty)$ of $(\Gamma,\tau)$, $\eta_k$ is the differential of the free-period action functional at energy $k$ of the Lagrangian $L+\theta$, for any primitive $\theta$ of $\sigma|_U$. 

We fix $k>e_0(L)$ and define the space $\Mult_k(n)$ as in Section~\ref{s:local_minimizers}, except that in condition~(D1) we now require the restriction $\gamma_i|_{[\tau_{i,j},\tau_{i,j+1}]}$ 
to be a unique local minimizer with energy $k$ for the Lagrangian $L+\theta$, for some primitive $\theta$ of $\sigma$ defined on a neighborhood of $\gamma_i|_{[\tau_{i,j},\tau_{i,j+1}]}$. We define a free-period action functional $\SSS_k:\Mult_k(n)\to\R$ associated to the Lagrangian system $(L,\sigma)$ by
\begin{align*}
\SSS_k(\ggamma)
:=
\sum_{i=1}^m \left(\int_0^{\tau_i}L(\gamma_i(t),\dot\gamma_i(t))\,\diff t + \tau_ik\right) 
+
\inf_{\{\ggamma_\alpha\}} \lim_{\alpha\to\infty} \int_{\Sigma_\alpha} \sigma,\\
\forall \ggamma=(\gamma_1,...,\gamma_m)\in\Mult_k(n),
\end{align*}
where the infimum in the above expression is over all the sequences of embedded multicurves $\{\ggamma_\alpha\ |\ \alpha\in\N\}\subset\Mult_k(n)\cap\ACMult(m)$ converging to $\ggamma$, and $\Sigma_\alpha\subset M$ is a possibly disconnected, embedded, compact surface whose orientation agrees with the one of $M$ and whose oriented boundary is the embedded multicurve $\ggamma_\alpha$. Notice that the restriction $\SSS_k|_{\Mult_k(n)\cap\ACMult(m)}$ is continuous. However, $\SSS_k$ is only lower semicontinuous. This lack of continuity is not an issue: indeed, $\SSS_k$ can only be discontinuous on certain multicurves $\ggamma=(\gamma_1,...,\gamma_m)\in\Mult_k(n)$ such that, for some $i\neq j$, $\gamma_i$ and $\gamma_j$ are the same geometric curve with opposite orientation (see the example in Figure~\ref{f:torus}); in this case, when looking for minimizers of $\SSS_k$, one can replace $\ggamma$ with the multicurve $\ggamma':=\ggamma\setminus\{\gamma_i,\gamma_j\}$ that still belongs to $\Mult_k(n)$ and satisfies $\SSS_k(\ggamma')\leq\SSS_k(\ggamma)$. The same arguments as in the proofs of Lemmas~\ref{l:components_are_periodic_orbits} and~\ref{l:components_are_embedded} imply that a multicurve $\ggamma\in\Mult_k(n-3)$ that is a global minimizer of $\SSS_k$ over $\Mult_k(n)$ is embedded, and its connected components lift to periodic orbits of the Euler-Lagrange flow of $(L,\sigma) $ on $E^{-1}(k)$.

\begin{figure}
\begin{center}
\begin{small}
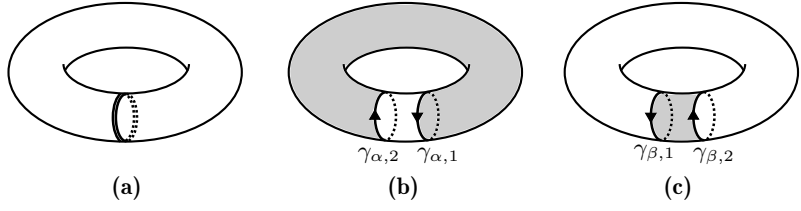 
\caption{In \textbf{(a)} we see an example of multicurve $\ggamma=(\gamma_1,\gamma_2)$ such that $\gamma_1$ and $\gamma_2$ are the same geometric curve with opposite orientation. Such a $\ggamma$ can be approached by embedded multicurves $\ggamma_\alpha=(\gamma_{\alpha,1},\gamma_{\alpha,2})$ as in \textbf{(b)} or $\ggamma_\beta=(\gamma_{\beta,1},\gamma_{\beta,2})$ as in \textbf{(c)}. The shaded regions are the compact surfaces $\Sigma_\alpha$ and $\Sigma_\beta$ respectively, with $\partial\Sigma_\alpha=\ggamma_\alpha$ and $\partial\Sigma_\beta=\ggamma_\beta$.}
\label{f:torus}
\end{small}
\end{center}
\end{figure}

For every finite covering space $M'\to M$, we can lift the Lagrangian systems of $(L,\sigma)$ to a Lagrangian system $(L',\sigma')$ on the configuration space $M'$. We denote by $\Mult_k'(n)$ and $\SSS_k'$ the space of multicurves and the free-period action functional associated to  $(L',\sigma')$. We define $e(L,\sigma)$ as the supremum of $e(M',L,\sigma)\in\R$ over all the finite covering spaces $M'\to M$, where
\begin{align*}
e(M',L,\sigma) & := \sup \bigg\{ k\geq e_0(L)\  \bigg|\ \inf_{\Mult_k' (n)}\SSS_k'< 0 \mbox{ for some } n\in\N \bigg\}.
\end{align*}
If $\sigma$ is exact with primitive $\theta$, Lemma~\ref{l:minima_are_negative} implies that  $e(L,\sigma)\geq \cu(L+\theta)$. 
The arguments in Sections~\ref{ss:embedded_global_min} and~\ref{s:compactness} can be carried over for the Lagrangian system of $(L,\sigma)$, and prove the following generalization of Theorem~\ref{t:local_minimizers}.

\vspace{4pt}

\begin{thm}\label{t:local_minimizers_non_exact}
Let $M$ be a closed surface, $L:\Tan M\to \R$ a Tonelli Lagrangian, and $\sigma$ a 2-form on $M$. For every $k\in (e_0(L),e(L,\sigma))$, the Lagrangian system of $(L,\sigma)$ possesses a periodic orbit $\gamma_k$ with energy $k$ that is a local minimizer of a local  primitive of the action 1-form $\eta_k$. Moreover, $\gamma_k$ lifts to a simple closed curve in some finite cover of $M$.
\hfill\qed
\end{thm}

\begin{rem}
For some pairs $(L,\sigma)$,  the energy interval $(e_0(L),e(L,\sigma))$ is empty, and in such case the assertion of Theorem~\ref{t:local_minimizers_non_exact} becomes void. However, when $L$ has the form of a kinetic energy $L(q,v)=\tfrac12 g_q(v,v)$ for some Riemannian metric $g$,  we have $e(L,\sigma)>e_0(L)$ if and only if $\sigma$ is \textbf{oscillating}, that is, $\sigma$ changes sign on $M$ (see \cite[Lemma 6.2]{Asselle:2015ij}). Moreover, given such a pair $(L,\sigma)$, for all Tonelli Lagrangians $L':\Tan M\to\R$ sufficiently $C^1$-close to $L$, we have $e(L',\sigma)>e_0(L')$ as well.
\hfill\qed
\end{rem}

\vspace{4pt}

If $M$ is a closed surface other than the 2-sphere, the 2-form $\sigma$  lifts to an exact form $\diff\theta$ on the universal cover $\widetilde M$ of $M$. 
We define
\begin{align*}
e_*(L,\sigma):= \min \big\{e(L,\sigma),c(\widetilde{L}+\theta)\big\}.
\end{align*}
Here, $\widetilde L:\Tan \widetilde M\to \R$ is the lift of $L$, and $c(\widetilde L+\theta)$ is the Ma\~n\'e critical value of the Lagrangian 
$\widetilde L+\theta$. By combining the arguments of Sections~\ref{s:discretizations} and \ref{s:iterated_mountain_passes} with the proofs of the main results in \cite{Asselle:2014hc, Asselle:2015ij, Asselle:2015sp}, we obtain the following generalization of Theorem~\ref{t:multiplicity}.

\begin{thm}\label{t:multiplicity_non_exact}
Let $M$ be a closed surface, $L:\Tan M\to \R$ a Tonelli Lagrangian, and $\sigma$ a 2-form on $M$.
\begin{itemize}
\item If $M\neq S^2$ or $\sigma$ is exact, for almost every $k\in (e_0(L),e_*(L,\sigma))$ the Lagrangian system of $(L,\sigma)$ possesses infinitely many periodic orbits with energy $k$.
\item If $M= S^2$ and $\sigma$ is not exact, for almost every $k\in (e_0(L),e(L,\sigma))$ the Lagrangian system of $(L,\sigma)$ possesses at least two\footnote{After the first version of this paper appeared online, the authors, together with Abbondandolo, Benedetti, and Taimanov, completed the case of the 2-sphere in Theorem~\ref{t:multiplicity_non_exact}: when $M=S^2$, for almost every energy value $k\in (e_0(L),e(L,\sigma))$ the Lagrangian system of $(L,\sigma)$ possesses infinitely many periodic orbits with energy $k$, see \cite[Theorem~1.1]{Abbondandolo:2016bh}.} periodic orbits with energy $k$.
\hfill\qed
\end{itemize}
\end{thm}

\bibliography{_biblio}
\bibliographystyle{amsalpha}

\end{document}